\documentclass[11pt]{amsart}

\usepackage{geometry}
\usepackage{colortbl}
\usepackage{amsmath,amssymb}
\usepackage{mathrsfs}
\usepackage{graphicx}
\usepackage{amsmath}
\usepackage{dsfont}
\usepackage{multirow}
\usepackage{graphicx}
\usepackage{epstopdf}
\usepackage{color}
\usepackage{ulem}
\usepackage{hyperref}
\usepackage{enumerate}
\usepackage{caption}
\usepackage{subfig}
\usepackage{bm}
\usepackage{overpic}

\usepackage{tikz}
\usepackage{varioref}
\usepackage{hyperref}
\usepackage{cleveref}

\makeatletter

\newcommand{\Rmnum}[1]{\expandafter\@slowromancap\romannumeral #1@}
\makeatother

\crefformat{figure}{Fig.~#2#1#3}
\crefmultiformat{figure}{Figures~(#2#1#3)}{ and~(#2#1#3)}{, (#2#1#3)}{ and~(#2#1#3)}
\crefformat{table}{Table~#2#1#3}

\def\bx{\mbox{\boldmath $x$}}
\def\by{\mbox{\boldmath $y$}}

\newtheorem{theorem}{Theorem}[section]

\theoremstyle{definition}
\newtheorem{assumption}{Assumption}[section]
\newtheorem{lemma}{Lemma}[section]
\newtheorem{proposition}{Proposition}[section]

\theoremstyle{remark}

\newtheorem{remark}{Remark}

\numberwithin{equation}{section}
\begin{document}

\title[Efficient FEMs for semiclassical NLSE with random potentials]{Efficient finite element methods for semiclassical nonlinear Schr\"odinger equations with random potentials}

\author{Panchi Li}
\address{Department of Mathematics, The University of Hong Kong, Hong Kong, P.R. China.}
%\curraddr{}
\email{lipch@hku.hk}
%\thanks{}

%    author two information
\author{Zhiwen Zhang$^{\ast}$}\thanks{*Corresponding author}
\address{Department of Mathematics, The University of Hong Kong, Hong Kong, P.R. China.  Materials Innovation Institute for Life Sciences and Energy (MILES), HKU-SIRI, Shenzhen, P.R. China. }
\email{zhangzw@hku.hk}
%\THANKS{}

\graphicspath{{figures/}}

%    \subjclass is required.

\date{\today}

\dedicatory{}
%    Abstract is required.
%========================================================================%%
\begin{abstract}
  In this paper, we propose two time-splitting finite element methods to solve the semiclassical nonlinear Schr\"odinger equation (NLSE) with random potentials. We then introduce the multiscale finite element method (MsFEM) to reduce the degrees of freedom in the physical space. We construct multiscale basis functions by solving optimization problems and rigorously analyze two time-splitting MsFEMs for the semiclassical NLSE with random potentials. We provide the $L^2$ error estimate of the proposed methods and show that they achieve second-order accuracy in both spatial and temporal spaces and an almost first-order convergence rate in the random space. Additionally, we present a multiscale reduced basis method to reduce the computational cost of constructing basis functions for solving random NLSEs. Finally, we carry out several 1D and 2D numerical examples to validate the convergence of our methods and investigate wave propagation behaviors in the NLSE with random potentials.
\end{abstract}
\maketitle
%%========================================================================%%

%%========================================================================%%
\section{Introduction}
The nonlinear Schr\"odinger equation~(NLSE) is a prototypical dispersive nonlinear equation that has been extensively used to study the Bose-Einstein condensation, laser beam propagation in nonlinear optics, particle physics, semi-conductors, superfluids, etc. In the presence of random potentials, the interaction of nonlinearity and random effect poses challenges to understanding complex phenomena, such as localization and delocalization~\cite{PhysRevA.96.043630,PhysRevE.79.026205,PhysRevE.95.042142,PhysRevLett.70.1787} and the soliton propagation~\cite{fishman2012nonlinear,PhysRevE.81.017601,PhysRevLett.100.094101}. Owing to the inherent challenges in obtaining analytical solutions and the limited experimental observations in nonlinear random media, numerical simulations play a crucial role in studying and understanding the nonlinear dynamics in such regimes, particularly for long-time behaviors in high-dimensional physical space. This necessitates efficient numerical methods for solving the NLSEs with random potentials.

In the past decades, many numerical methods have been proposed to solve the NLSEs with deterministic potentials, and recent comparisons can be found in~\cite{ANTOINE20132621,WeizhuBao2013KRMs,Henning2019KRMs}. For the time-dependent NLSE, the implicit Crank-Nicolson~(CN) schemes were extensively employed to conserve the mass and energy of the system. The CN method is known for its lower efficiency in handling nonlinearity because of the requirements of iteration methods and time step conditions~\cite{Akrivis1991,article1988,sanz1984methods}. To enhance computational efficiency, several promising approaches, including linearized implicit methods~\cite{Wang2014,refId02001}, relaxation methods~\cite{Besse2004relaxation,10.1093/imanum/drz067} and time-splitting methods~\cite{doi:10.1137/S1064827501393253,Besse2003SIAM,doi:10.1137/060674636}, have been proposed. Among these, time-splitting methods exhibit outstanding performance in terms of efficiency since linear equations with constant coefficients are solved at each time step. To reach optimal accuracy, time-splitting type schemes
require enough smoothness on both the potential and the initial condition. For instance, Strang splitting methods require the initial condition to possess $H^4$ regularity~\cite{Besse2003SIAM}. The low-regularity time-integrator methods~\cite{doi:10.1137/18M1198375,ostermann2022second,Zhao2021} are proposed to alleviate such constraint. Nevertheless, the low-regularity time-integrator methods rely on the Fourier discretization in space with a periodical setup, and their integration with finite difference methods~(FDM) and finite element methods~(FEM) has not been established.

The spatial Fourier discretization allows the spectral methods to have exponential convergence for smooth potentials and competitive efficiency in simulations. However, in the case of non-smooth potentials, spectral methods may lose their optimal convergence rate, and the FDM or FEM is then recommended. In this paper, we invest to develop efficient numerical methods based on the FEM. Over the past several decades, to develop efficient FEM methods for partial differential equations, intense research efforts in dimensionality reduction methods by constructing the multiscale reduced basis functions, known as the multiscale finite method (MsFEM), have been invested (see, e.g., \cite{Altmann_Henning_Peterseim_2021,CHUNG201454,weinan2003heterognous,EFENDIEV2013116,efendiev2009multiscale,
doi:10.1137/130933198,HOU1997169,doi:10.1137/100813968}). Incorporating the local microstructures of the differential operator into the basis functions, MsFEMs capture the large-scale components of the multiscale solution on a coarse mesh without the need to resolve all the small-scale features on a fine mesh.

Recently, the localized orthogonal decomposition (LOD) method \cite{Altmann_Henning_Peterseim_2021} has been proposed to approximate the minimizes of the energy \cite{henning2014two,doi:10.1137/22M1516300} and simulate the time-dependent dynamics \cite{doding2022efficient} for the NLSE with deterministic potential, which achieves a superconvergence rate in space. With the random potential further being considered, the time-splitting spectral discretization with the Monte Carlo~(MC) sampling~\cite{Zhao2021} and quasi-Monte Carlo~(qMC) sampling~\cite{wu2023error} have been employed for the 1D NLSE. Considering the limitation of the spectral methods, and developing efficient numerical methods in the framework of the FEM, here we combine the time-splitting temporal discretization and MsFEM to solve the NLSE with random potentials.

We generate the multiscale basis functions by solving a set of equality-constrained quadratic programs. This idea was motivated by the MsFEM for elliptic problems with random coefficients~\cite{doi:10.1137/18M1205844,Hou2017} and the linear Schr\"odinger equation with multiscale and random potentials~\cite{doi:10.1137/19M127389X}. We find that the localized orthogonal normalization constraints in optimization problems imply a mesh-dependent scale in the basis functions. This scale in the linear algebraic equation is eliminated naturally. However, when the cubic nonlinear term appears, the balance of such scale in the equation is broken, which produces an indispensable scale in the numerical solution. To overcome this problem, we add a mesh-dependent factor to the orthogonality constraints to eliminate this scale of basis functions. We use these new basis functions to discrete the deterministic NLSE that reduces the degrees of freedom~(dofs) required for FEM without sacrificing accuracy.
For the time-marching, we present two Strang splitting methods. One solves the linear Schr\"odinger equation using the eigendecomposition method~\cite{doi:10.1137/19M127389X} and handles the cubic ordinary differential equation at each time step. This method maintains the convergence rate even for the discontinuous potential (see the numerical demonstration shown in \Cref{fig:convergence-MsFEM-discontinuous}). The other is the time-splitting CN method. Meanwhile, we parameterize the random potential using the Karhunen-Lo\`{e}ve~(KL) expansion method. We use the qMC method to generate random samples. It is demonstrated that the proposed approaches reach the second-order convergence rate in both time and space, and achieve almost a first-order convergence rate with respect to the sampling number in the random space. Theoretically, we provide a convergence analysis of the time-splitting FEM~(TS-FEM) for the deterministic NLES in \Cref{lem:TS-MsFEM-convergence}. Furthermore, we extend this analysis to estimate the time-splitting MsFEM~(TS-MsFEM) for the NLSEs with random potentials as \Cref{thm:L2-estimate-TS-MsFEM-qMC}.
We verify several theoretical aspects through numerical experiments. Besides, we propose a multiscale reduced basis method to decrease the computational cost of the construction of multiscale basis functions for random potentials, which is detailed in Appendix~\ref{sec:reduced-MsFEM}. Using the proposed numerical methods, we investigate wave propagation in the NLSE with parameterized random potentials in both 1D and 2D physical spaces. We observe the localized phenomena of mass density for the linear cases, whereas the NLSE with strong nonlinearity exhibits significant delocalization.

The rest of the paper is organized as follows. In Section 2, we describe fundamental model problems. In Section 3, we present the FEM and the MsFEM with time-splitting methods for the deterministic NLSEs. Then, the analysis results are presented in Section 4. Numerical experiments, including 1D and 2D examples, are conducted in Section 5. Finally, conclusions are drawn in Section 6.

\section{The semiclassical NLSE with random potentials}
We consider the following model problem
\begin{equation}
  \left\{\begin{aligned}
    i\epsilon\partial_t\psi^{\epsilon} &= -\frac{\epsilon^2}{2}\Delta\psi^{\epsilon} + v(\bx,{\omega})\psi^{\epsilon} + \lambda|\psi^{\epsilon}|^2\psi^{\epsilon}, \quad \bx\in\mathcal{D}, \quad {\omega}\in\Omega, \quad t\in(0, T],\\
    \psi^{\epsilon}|_{t=0} &= \psi_{\mathrm{in}}(\bx),
  \end{aligned}\right.
  \label{equ:NLS_equ}
\end{equation}
where $0 < \epsilon \ll 1$ is an effective Planck constant, $\mathcal{D} \subset \mathds{R}^d (d = 1,2,3)$ is a bounded domain, ${\omega} \in \Omega$ is the random sample with $\Omega$ being the random space, $T$ is the terminal time, $\psi_{\mathrm{in}}(\bx)$ denotes the initial state, $v(\bx,{\omega})$ is a given random potential, and $\lambda~(\geq 0)$ is the nonlinearity coefficient. The periodic boundary is considered in this work. Physically, $|\psi^{\epsilon}|^2$ denotes the mass density and the system's total mass $m_T = \int_{\mathcal{D}}|\psi_{\mathrm{in}}|^2\mathrm{d}\bx$ is conserved by~\eqref{equ:NLS_equ}.
Note that the wave function $\psi^{\epsilon}: [0, T]\times\mathcal{D}\times\Omega\rightarrow \mathds{C}$, and the function space $H^1_P(\mathcal{D}) = H^1_P(\mathcal{D}, \mathds{C})$, in which the functions are periodic over domain $\mathcal{D}$. The inner product is defined as $(v,w) = \int_{\mathcal{D}}v\overline{w}\mathrm{d}\bx$ with $\overline{w}$ denoting the complex-conjugate of $w$, and the $L^2$ norm is $\|w\|^2 = \||w|\|^2 = (w,w)$.

We denote the Hamiltonian operator $\mathcal{H}$ of the NLSE
\begin{equation}
  \mathcal{H}(\cdot) = -\frac{\epsilon^2}{2}\Delta(\cdot)  + v(\cdot) + \lambda|\cdot|^2(\cdot).
\end{equation}
Since the Hamiltonian operator is not explicitly dependent on time and the commutator $[\mathcal{H}, \mathcal{H}] = 0$, the energy of the system,
\begin{equation}
  E(t) = (\mathcal{H}\psi^{\epsilon}, \psi^{\epsilon}) = \frac{\epsilon^2}{2}\|\nabla\psi^{\epsilon}\|^2 + (v(\bx, {\omega}), |\psi^{\epsilon}|^2) + \frac{\lambda}{2}\|\psi^{\epsilon}\|_{L^4}^4,
\end{equation}
remains unchanged as time evolves, i.e., $\mathrm{d}_tE(t) = 0$ for all $t > 0$.

\begin{assumption}
  We assume that the potential $v(\bx, \omega)$ is bounded in $L^{\infty}(\Omega; H^{s})$ with $0 \leq s \leq 2$. More precisely, the bound of $\|v(\bx, \omega)\|_{\infty}$ satisfies
  \begin{equation}
    \|v(\bx, \omega)\|_{\infty} \lesssim \frac{\epsilon^2}{H^2},
  \end{equation}
  where $\lesssim$ means bounded by a constant and $H$, which will be defined later, is the coarse mesh size of the MsFEM.
  \label{assump:potential}
\end{assumption}

We first consider the deterministic potential, i.e., $v(\bx, \omega) = v(\bx)$. Assume that there exists a finite time $T$ such that $\psi^{\epsilon}\in L^{\infty}([0, T]; H^{4}) \cap L^{1}([0, T]; H^2)$ and by Sobolev embedding theorem, we have $\|\psi^{\epsilon}\|_{\infty} \leq C\|\psi^{\epsilon}\|_{H^2}$ for $d \leq 3$.
In the sequel, we will use a uniform constant $C$ to denote all the controllable constants that are independent of $\epsilon$ for simplicity of notation.

\begin{lemma}
  Let $\psi^{\epsilon}$ be the solution of~\eqref{equ:NLS_equ}, and assume $\psi^{\epsilon}\in L^{\infty}([0, T]; H^{4}) \cap L^{1}([0, T]; H^2)$. If $\partial_t\psi^{\epsilon}(t)\in H^s$ with $s = 0,1,2$ for all $t\in [0, T]$, there exists a constant $C_{\lambda,\epsilon}$ such that
  \begin{equation}
    \|\partial_t\psi^{\epsilon}\|_{H^s} \leq C_{\lambda,\epsilon},
  \end{equation}
  where $C_{\lambda,\epsilon}$ mainly depends on $\epsilon$ and $\lambda$.
  In particular, for $d = 3$ and $s = 1, 2$, we have a compact formulate
  \begin{equation*}
    \|\partial_t\nabla^s\psi^{\epsilon}\| \leq
    \left( \frac{\|\nabla v\|_{\infty} + C\lambda \|\nabla^{s+1}\psi^{\epsilon}\|}{\epsilon} \right) \|\partial_t\nabla^{s-1}\psi^{\epsilon}\| \exp\left( \frac{C\lambda T (\|\nabla
    ^2\psi^{\epsilon}\| + \|\psi^{\epsilon}\|_{\infty}^2)}{\epsilon} \right),
  \end{equation*}
  where
  \begin{equation}
    \|\partial_{t}\psi^{\epsilon}\| \leq \frac{C}{\epsilon} \exp\left( \frac{2\lambda T\| \psi^{\epsilon} \|_{\infty}^2}{\epsilon} \right).
  \end{equation}
  \label{lem:regularity-NLSE}
\end{lemma}

The proof is detailed in Appendix \ref{appsec:proof-regularity-NLSE}. Note that for $\lambda = 0$, the result of this lemma degenerates to the estimate of the linear Schr\"odinger equation as in~\cite{BAO2002487,Wu2022}.

Next, we assume that $v(\bx, \omega)$ is a second-order random field, i.e., $\mathds{E}[|v(\bx, \omega)|^2] < \infty$, with a mean value $\mathds{E}[v(\bx, \omega)] = v(\bx)$ and a covariance kernel denoted by $C(\bx,\by)$. In this study, we adopt the covariance kernel
\begin{equation}
  C(\bx, \by) = \sigma^2\exp\left(-\sum_{i=1}^d\frac{|x_i-y_j|^2}{2l_i^2}\right),
\end{equation}
where $\sigma$ is a constant and $l_i$ denotes the correlation lengths in each dimension. Moreover, we assume that the random potential is almost surely bounded. Using the KL expansion method~\cite{karhunen1947lineare,loeve1978probability}, random potentials take the form
\begin{equation}
  v(\bx, {\omega}) = \bar{v}(\bx) + \sum_{j=1}^{\infty}\sqrt{\lambda_j}\xi_j(\omega)v_j(\bx),
\end{equation}
where $\xi_i(\omega)$ represents mean-zero and uncorrelated random variables, and $\{\lambda_i, v_i(\bx)\}$ are the eigenpairs of the covariance kernel $C(\bx, \by)$. The eigenvalues are sorted in descending order and the decay rate depends on the regularity of the covariance kernel~\cite{SCHWAB2006100}. Hence the random potential can be parameterized by the truncated form
\begin{equation}
  v_m(\bx, {\omega}) = \bar{v}(\bx) + \sum_{j=1}^{m}\sqrt{\lambda_j}\xi_j(\omega)v_j(\bx).
  \label{equ:trcunated-randon-potential}
\end{equation}

Once the parameterized form of the random potential is defined, the corresponding wave function $\psi^{\epsilon}_m$ satisfies
\begin{equation}
  \left\{\begin{aligned}
    &i\epsilon\partial_t\psi^{\epsilon}_m= -\frac{\epsilon^2}{2}\Delta\psi^{\epsilon}_m + v_m(\bx, {\omega})\psi^{\epsilon}_m + \lambda|\psi^{\epsilon}_m|^2\psi^{\epsilon}_m, \quad \bx\in\mathcal{D}, \omega\in\Omega, t\in(0, T],\\
    &\psi^{\epsilon}_m(t=0)= \psi_{\mathrm{in}}.
  \end{aligned}\right.
  \label{equ:NLS_equ_parameterized}
\end{equation}
The error $|v_m(\bx, {\omega}) - v(\bx, {\omega})|$ depends on the regularity of eigenfunctions and the decay rate of eigenvalues. We make the following assumption for the parameterized random potentials.
\begin{assumption}
  \begin{enumerate}
    \item In the KL expansion~\eqref{equ:trcunated-randon-potential}, assume that there exist constants $C > 0$ and $\Theta > 1$ such that $\lambda_j \leq Cj^{-\Theta}$ for all $j \geq 1$.
    \item The eigenfunctions $v_j(\bx)$ are continuous and there exist constants $C > 0$ and $0 \leq \eta \leq \frac{\Theta -1}{2\Theta}$ such that $\|v_j\|_{H^2} \leq C\lambda_j^{-\eta}$ for all $j \ge 1$.
    \item Assume that the parameterized potential $v_m$ satisfies
  \begin{equation*}
    \|v - v_m\|_{\infty} \leq Cm^{-\chi}, \quad \sum_{j=1}^{\infty}(\sqrt{\lambda_j}\|v_j\|_{H^2})^p < \infty,
  \end{equation*}
  for some positive constants $C$ and $\chi$, and $p\in (0,1]$.
  \end{enumerate}
  \label{assump:random-potential-assumption}
\end{assumption}

In~\cite{wu2023error}, the authors provide the $L^{\infty}([0, T], H^1)$ error between wave functions to~\eqref{equ:NLS_equ} and \eqref{equ:NLS_equ_parameterized} for the 1D case. Here we get a similar result for the $L^2$ error between the wave functions for $d \leq 3$.
\begin{lemma}
  The error between wave functions to \eqref{equ:NLS_equ} and \eqref{equ:NLS_equ_parameterized} satisfies
  \begin{equation}
    \|\psi^{\epsilon}_m - \psi^{\epsilon}\| \leq \frac{2\|v_m-v\|_{\infty}}{\epsilon}\exp\left(\frac{2T\lambda}{\epsilon}
  \|\psi^{\epsilon}\|_{\infty}\|\psi^{\epsilon}_m\|_{\infty}\right).
  \end{equation}
\end{lemma}
\begin{proof}
  Define $\delta\psi = \psi^{\epsilon}_m - \psi^{\epsilon}$ and it satisfies
  \begin{equation}
    i\epsilon\partial_t\delta\psi= -\frac{\epsilon^2}{2}\Delta\delta\psi + v_m\delta\psi + (v_m-v)\psi^{\epsilon} + \lambda(|\psi^{\epsilon}_m|^2\psi^{\epsilon}_m - |\psi^{\epsilon}|^2\psi^{\epsilon})
  \label{equ:difference-equation-random-deter}
\end{equation}
with the initial condition $\delta\psi(t=0) = 0$.
For the nonlinear term, we have
$$
|\psi^{\epsilon}_m|^2\psi^{\epsilon}_m - |\psi^{\epsilon}|^2\psi^{\epsilon} = |\psi^{\epsilon}_m|^2\delta\psi + \psi^{\epsilon}\psi^{\epsilon}_m\delta\bar\psi + |\psi^{\epsilon}|^2\delta\psi.
$$
Taking the inner product of \eqref{equ:difference-equation-random-deter} with $\delta\psi$ yields
\begin{equation*}
  i\epsilon\mathrm{d}_t\|\delta\psi\|^2 = \big((v_m-v)\psi^{\epsilon},\delta\psi \big) - \big((v_m-v)\bar\psi^{\epsilon},\delta\bar\psi \big) + \lambda\big((\psi^{\epsilon}\delta\bar\psi, \bar\psi^{\epsilon}_m\delta\psi) - (\bar\psi^{\epsilon}\delta\psi, \psi^{\epsilon}_m\delta\bar\psi)\big).
\end{equation*}
We further get
\begin{align*}
  \mathrm{d}_t\|\delta\psi\|^2 &\leq \frac{2\|v_m-v\|_{\infty}}{\epsilon}\int_{\mathcal{D}}|\psi^{\epsilon}||\delta\psi|\mathrm{d}\bx + \frac{2\lambda}{\epsilon}\int_{\mathcal{D}}|\psi^{\epsilon}\delta\psi| |\psi^{\epsilon}_m\delta\psi|\mathrm{d}\bx \\
  &\leq \frac{2\|v_m-v\|_{\infty}}{\epsilon}\|\psi^{\epsilon}\|\|\delta\psi\| + \frac{2\lambda}{\epsilon}\|\psi^{\epsilon}\|_{\infty}\|\psi^{\epsilon}_m\|_{\infty}\|\delta\psi\|^2.
\end{align*}
Owing to the $L^{\infty}([0, T]\times\Omega; H^s)$ bound of both $\psi^{\epsilon}$ and $\psi^{\epsilon}_m$, an application of Gronwall inequality yields
\begin{equation*}
  \|\delta\psi\| \leq \frac{2T\|v_m-v\|_{\infty}}{\epsilon}\exp\left(\frac{2T\lambda}{\epsilon}
  \|\psi^{\epsilon}\|_{\infty}\|\psi^{\epsilon}_m\|_{\infty}\right).
\end{equation*}
\end{proof}
Owing to the assumption $\|v_m-v\|_{\infty} \leq Cm^{-\chi}$, this lemma implies that $\psi^{\epsilon}_m \rightarrow \psi^{\epsilon}$ as $m \rightarrow \infty$.

\section{Numerical methods}
Consider the regular mesh $\mathcal{T}_h$ of $\mathcal{D}$. The standard $P_1$ finite element space on the mesh $\mathcal{T}_h$ is given by
$P_1(\mathcal{T}_h) = \{v\in L^2(\bar{\mathcal{D}})| \text{ for all } K\in\mathcal{T}_h, v|_K \text{ is a polynomial of total degree} \leq 1\}$.
Then the $H^1_P(\mathcal{D})$-conforming finite element spaces are
$V_h = P_1(\mathcal{T}_h)\cap H^1_P(\mathcal{D})$ and $V_H = P_1(\mathcal{T}_H)\cap H^1_P(\mathcal{D})$.
Denote $V_h = span\{\phi_1^h, \cdots, \phi_{N_h}^h\}$ and $V_H = span\{\phi_1^H, \cdots, \phi_{N_H}^H\}$, where $N_h$ and $N_H$ are the number of vertices of the fine mesh and the coarse mesh, respectively. The wave function is approximated by $\psi_h^{\epsilon}(t, \bx) = \sum_p^{N_h} U_p(t)\phi_p^h(\bx)$ on the fine mesh, where $U_p(t)\in\mathds{C},  p = 1,\cdots,N_h$ and $t\in[0, T]$.

\subsection{TS-FEM for the NLSE}
In the case of nontrivial potentials, the numerical mass density may decay towards zero with an exponential rate when utilizing the direct Backward Euler method. Time-splitting manners can maintain the mass of the system. Therefore, we adopt Strang splitting methods for time-stepping. The NLSE is rewritten to
\begin{align}
  i\epsilon\partial_t\psi^{\epsilon} = (\mathcal{L}_1 + \mathcal{L}_2)\psi^{\epsilon},
\end{align}
and its exact solution has the form $\psi^{\epsilon}(t) = S^t\psi_{\mathrm{in}}$, where $S^t = \exp(-i(\mathcal{L}_1+\mathcal{L}_2)t/\epsilon)$. To efficiently handle the nonlinear term, we present two approaches as follows, both of which require solving linear equations:
\begin{enumerate}
  \item Option 1,
  \begin{equation}
    \mathcal{L}_1(\cdot) = -\frac {\epsilon^2}{2}\Delta(\cdot) + v(\cdot), \quad \mathcal{L}_2(\cdot) = \lambda|\cdot|^2(\cdot).
    \label{equ:choice-1}
  \end{equation}
  \item Option 2,
  \begin{equation}
    \mathcal{L}_1(\cdot) = -\frac {\epsilon^2}{2}\Delta(\cdot), \quad \mathcal{L}_2(\cdot) = v(\cdot) + \lambda|\cdot|^2(\cdot).
    \label{equ:choice-2}
  \end{equation}
\end{enumerate}
Computing the commutator $[\mathcal{L}_1, \mathcal{L}_2] = \mathcal{L}_1\mathcal{L}_2 - \mathcal{L}_2\mathcal{L}_1$ shows that the regularity of potential $v \in C^2(\mathcal{D})$ is required for Option 2, whereas Option 1 does not need this requirement.

From $t_n$ to $t_{n+1}$, the Strang splitting yields
\begin{equation}
  \psi^{\epsilon,n+1}
  := \mathcal{L}\psi^{\epsilon, n}
  = \exp\left( -\frac{i\Delta t}{2\epsilon}\mathcal{L}_2(\cdot) \right)
  \circ \exp\left( -\frac{i\Delta t}{\epsilon}\mathcal{L}_1 \right)
  \exp\left( -\frac{i\Delta t}{2\epsilon} \mathcal{L}_2(\cdot) \right)
  \circ \psi^{\epsilon, n}.
  \label{equ:one-step-Srange-splitting1}
\end{equation}
This formulation can be written as
\begin{align}
  \psi^{\epsilon,n+1}
  = \exp\left(-\frac{i\Delta t}{\epsilon}(\mathcal{L}_1 + \mathcal{L}_2(\psi^{\epsilon,n}))\right)\psi^{\epsilon, n} + \mathcal{R}_1^n.
  \label{equ:Strange-splitting-remnant}
\end{align}
By the Taylor expansion, we have $\|\mathcal{R}_1^n\| = \mathcal{O}\left( \frac{\Delta t^3}{\epsilon^3} \right)$. Furthermore, we define the $n$-fold composition
\begin{equation}
    \psi^{\epsilon, n} = \mathcal{L}^n\psi_{\mathrm{in}} = \underbrace{\mathcal{L}(\Delta t, \cdot)\circ\cdots\circ
    \mathcal{L}(\Delta t, \cdot)}_{n \text{ times}}\psi_{\mathrm{in}}.
\end{equation}

Next, we introduce the classical finite element discretization for the operator $\mathcal{L}_1$. Define the weak form
\begin{equation}
  i\epsilon(\partial_t\psi^{\epsilon},\phi) = a(\psi^{\epsilon}, \phi), \quad \forall\phi\in H^1_P(\mathcal{D}),
\end{equation}
where $a(\psi^{\epsilon}, \phi)$ is determined by the option of $\mathcal{L}_1$. For example, setting $\mathcal{L}_1 = -\frac {\epsilon^2}{2}\Delta + v$, we have $a(\psi^{\epsilon}, \phi) = \frac{\epsilon^2}{2}(\nabla\psi^{\epsilon}, \nabla\phi) + (v\psi^{\epsilon}, \phi)$ and the Galerkin equations
\begin{equation}
  i\epsilon\sum_p\mathrm{d}_t{U}_p(\phi_p^h, \phi_q^h) = \frac{\epsilon^2}{2}\sum_pU_p(t)(\nabla\phi_p^h, \nabla\phi_q^h) + \sum_pU_p(t)(v\phi_p^h, \phi_q^h)
\end{equation}
with $q = 1,\cdots, N_h$. Its matrix form is
\begin{equation}
  i\epsilon M^h\mathrm{d}_t{U}(t) = \left(\frac{\epsilon^2}{2}S^h + V^h\right)U(t),
\end{equation}
where $U(t)$ is a vector with $U(t) = (U_1(t),\cdots, U_{N_h}(t))^T$, $M^h = [M_{pq}^h]$ is the mass matrix with $M_{pq}^h = (\phi_p^h, \phi_q^h)$, $S^h = [S_{pq}^h]$ is the stiff matrix with $S_{pq}^h = (\nabla\phi_p^h, \nabla\phi_q^h)$, and $V^h = [V_{pq}^h]$ is the potential matrix with $V_{pq}^h = (v\phi_p^h, \phi_q^h)$.

We now present the formal TS-FEM methods for the deterministic NLSE. The first one is the discretized counterpart of Option 1:
\begin{align}
  &\tilde{U}^{n} = \exp\left( -\frac{i\lambda\Delta t}{2\epsilon} |{U}^{n}|^2 \right){U}^{n}, \nonumber\\
  &\tilde{U}^{n+1} = P \exp\left( -\frac{i\Delta t}{\epsilon}\Lambda \right)(P^{-1}\tilde{U}^{n}),\\
  &{U}^{n+1} =\exp\left( -\frac{i\lambda\Delta t}{2\epsilon} |\tilde{U}^{n+1}|^2 \right)\tilde{U}^{n+1}, \nonumber
\end{align}
where $(M^h)^{-1}(\frac{\epsilon^2}{2}S^h + V^h) = P\Lambda P^{-1}$ with $\exp(-i\Delta t\Lambda/\epsilon)$ being a diagonal matrix. We call it {\bf SI} in the remainder of this paper. Owing to the application of the eigendecomposition method~\cite{doi:10.1137/19M127389X}, the error in time is mainly contributed by the time-splitting manner. Meanwhile, this scheme does not require time step size $\Delta t = {o}(\epsilon)$, although the full linear semiclassical Schr\"odinger equation must be solved.

Option 2 has been extensively used in previous works, such as~\cite{BAO2003318,doi:10.1137/S1064827501393253}. In the FEM framework, it solves the NLES in the following procedures:
\begin{align}
  &\tilde{U}^n =\exp\left( -\frac{i\Delta t}{2\epsilon} (v + \lambda|{U}^n|^2) \right) {U}^n, \nonumber\\
  &iM^h\left(\frac{\tilde{U}^{n+1} - \tilde{U}^{n}}{\Delta t}\right) = \frac{\epsilon}{2} S^h\left(\frac{\tilde{U}^{n+1} + \tilde{U}^{n}}{2}\right), \\
  &U^{n+1} =\exp\left( -\frac{i\Delta t}{2\epsilon} (v+\lambda|\tilde{U}^{n+1}|^2) \right) \tilde{U}^{n+1}. \nonumber
\end{align}
This method requires the mesh size $h = \mathcal{O}(\epsilon)$ and time step size $\Delta t = \mathcal{O}(\epsilon)$~\cite{doi:10.1137/S1064827501393253}, and we call it {\bf SII} in the remaining of this paper.

Denote $L$ the discretized counterpart of $\mathcal{L}$, and similarly, $L_1$ and $L_2$ their respective discretized versions.
From $t_n$ to $t_{n+1}$, the discretized solution in both time and space can be determined by the recurrence
\begin{equation}
  U^{n+1} = {L}(\Delta t, U^n)U^n = {L}_2\left(\frac{\Delta t}{2}, {L}_1(\Delta t){L}_2\left(\frac{\Delta t}{2}, U^n\right)\right)U^n.
  \label{equ:one-step-evolution-TS}
\end{equation}
Denote $\psi^{\epsilon, n}_h = \sum_{p=1}^{N_h}U_p^n\phi_p^h$, and for simplicity we employ a formal notation for the $n$-fold composition
\begin{equation}
  \psi^{\epsilon, n}_h = {L}^{n}\psi^{0}_{h} = \underbrace{{L}(\Delta t, \cdot)\circ\cdots\circ{L}(\Delta t, \cdot)}_{n \text{ times}}\psi_h^{0},
  \label{equ:fully-TS-FEM}
\end{equation}
where $\psi_h^0 = {R}_h\psi_{\mathrm{in}}$ with $R_h$ being the Ritz projection operator.

\subsection{MsFEM for the deterministic NLSE}
Instead of the FEM, we construct the multiscale basis functions to reduce dofs in computations. The $P_1$ FEM basis functions on both the coarse mesh $\mathcal{T}_H$ and fine mesh $\mathcal{T}_h$ are required simultaneously. To describe the localized property of multiscale basis functions, we define a series of nodal patches $\{D_{\ell}\}$ associated with $\bx_p\in\mathcal{N}_H$ as
\begin{align*}
  &D_0(\bx_p) := \mathrm{supp}\{\phi_p\} = \cup\{K\in\mathcal{T}_H\;|\; \bx_p\in K\}, \\
  &D_{\ell} := \cup\{K\in\mathcal{T}_H\;|\;K\cap\overline{D_{\ell-1}}\neq \emptyset\},\quad \ell = 1,2,\cdots.
\end{align*}

The multiscale basis functions are obtained by solving the optimization problems
\begin{align}
  \phi_p = &\arg\min_{\phi\in H^1_P(\mathcal{D})}a(\phi, \phi),\label{equ:optimal-problem-objective}\\
  &\quad \text{s.t.}\int_{\mathcal{D}}\phi\phi_q^H\mathrm{d}\bx = \lambda(H)\delta_{pq}, \quad \forall 1\leq q \leq N_H,
  \label{equ:optimal-problem}
\end{align}
where $a(\phi, \phi) = \frac{\epsilon^2}{2}(\nabla\phi, \nabla\phi) + (v\phi, \phi)$, and $\lambda(H) = 1$ in the previous work~\cite{CHEN2020113232,doi:10.1137/19M1236989,doi:10.1137/19M127389X,doi:10.1137/18M1205844,li2018computing}. Note that the localized constraint is not considered in the optimal problems, thus we obtain the global basis functions.

In this work, we set $\lambda(H) = (1, \phi_q^H)$, and it can be computed explicitly.
To explain this setup, we introduce the weighted Cl\'ement-type quasi-interpolation operator~\cite{doi:10.1137/130933198}
\begin{equation}
  I_{H}: H^{1}_P(\mathcal{D})\rightarrow V_{H}, \quad f \mapsto I_{H}(f):= \sum_p \frac{(f, \phi^H_p)}{(1,\phi^H_p)} \phi^H_p.
\end{equation}
The high-resolution finite element space $V_h = V_H \oplus W_h$, where $W_h$ is the kernel space of $I_H$. And for all $f \in H^1_P \cap H^2$, it holds~\cite{Malqvist2015}
\begin{equation}
  \|f - I_H(f)\| \leq H^2\|f\|_{H^2}.
\end{equation}

In the MsFEM space, the wave function $\psi^{\epsilon}$ is approximated with
\begin{equation}
  \psi^{\epsilon}(\bx) \approx \sum_{p=1}^{N_H} \hat{U}_p\phi_p.
  \label{equ:solution-ms_basis}
\end{equation}
It can be projected onto the coarse mesh through
\begin{equation*}
    I_H(\psi^{\epsilon})
    =
    \sum_{p=1}^{N_H}\frac{(\sum_{q=1}^{N_H}\hat{U}_q\phi_q, \phi_p^H)}{(1,\phi_p^H)}\phi_p^H
    =
    \sum_{p=1}^{N_H}\frac{\lambda(H)\hat{U}_p}{(1,\phi_p^H)}\phi_p^H.
\end{equation*}
If $\psi^{\epsilon}$ is continuous at $\bx_p$, the above formula indicates that
$$\psi^{\epsilon}(\bx_p) \approx \frac{\lambda(H)\hat{U}_p}{(1,\phi_p^H)}.$$
Let $\lambda(H) = 1$, and we can see that it holds $\psi^{\epsilon}(\bx_p) \approx \hat{U}_p/(1,\phi_p^H)$ in the MsFEM space. Define $\hat{\phi}_p = (1, \phi^H_p){\phi}_p$, where $\hat{\phi}_p$ is independent of the mesh size $H$. Then, \eqref{equ:solution-ms_basis} can be rewritten to
\begin{equation}
 \psi^{\epsilon}(\bx) \approx \sum_{p=1}^{N_H} \psi^{\epsilon}(\bx_p)(1, \phi_p^H)\phi_p = \sum_{p=1}^{N_H} \psi^{\epsilon}(\bx_p)\hat{\phi_p}.
 \label{equ:solution-ms_basis2}
\end{equation}
Note that $\hat{\phi_p}$ is still the multiscale basis function at $\bx_p$.
We consider the following two equations
\begin{equation}
 i\epsilon\sum_{p=1}^{N_H}(\phi_p, \phi_q)\mathrm{d}_t\hat{U}_p = \sum_{p=1}^{N_H}(\mathcal{H}\phi_p, \phi_q)\hat{U}_p
 \label{equ:multiscale-equation1}
\end{equation}
and
\begin{equation}
 i\epsilon\sum_{p=1}^{N_H}(\hat{\phi}_p, \hat{\phi}_q)\mathrm{d}_t\hat{U}_p = \sum_{p=1}^{N_H}(\mathcal{H}\hat{\phi}_p, \hat{\phi}_q)\hat{U}_p.
 \label{equ:multiscale-equation2}
\end{equation}
If $\lambda = 0$, the two equations have the same solution with a given initial condition, while for $\lambda \neq 0$, the factor $(1, \phi^H_p)$ in the basis functions cannot be eliminated in both sides of~\eqref{equ:multiscale-equation2}, and the two equations give different solutions. This issue is addressed by the setup $\lambda(H) = (1, \phi_p^H)$.

Solving the optimal problems~\eqref{equ:optimal-problem} on the fine mesh, we get
\begin{equation*}
  \phi_p = \sum_{s=1}^{N_h}c_p^s\phi_s^h, \quad p = 1, \cdots, N_H.
\end{equation*}
Define the MsFEM space $V_{ms} = span\{\phi_1,\cdots, \phi_{N_H}\}$,
and it holds true that $V_{ms} \subset V_h$. Hence the solution of optimal problems defines a linear transformation $\mathcal{C}: V_{h} \mapsto V_{ms}$. On the other hand, the solution on the fine mesh can be reconstructed utilizing this linear mapping, which is essential in the formulation of the cubic nonlinear matrix. Note that the factor $\lambda(H)$ is a rescaling factor, and it does not change the basis function space. Thus we have the following propositions.
% By the following lemma, we have $V_h = V_{ms} \oplus W_h$.
\begin{proposition}[\cite{Wu2022}, Lemma 3.2]
  For all $\phi \in V_{ms}$ and $w \in W_h$, $a(\phi, w) = 0$ and $V_h = V_{ms} \oplus W_h$.
\end{proposition}
\begin{proof}
  As the same procedures in~\cite{Wu2022}, we directly obtain $a(f, w) = 0, \forall f \in V_{ms}, w \in W_h$.
  For any $f \in V_h$, define
  $$f^* = \sum_{p=1}^{N_H}\frac{(f, \phi_p^H)}{(1, \phi_p^H)}\phi_p.$$
  Then $f^* \in V_{ms}$ and $(f - f^*, \phi_p^H) = 0$ for $p = 1,\cdots, N_H$. Thus $f - f^* \in W_h$ and we get the decomposition $V_h = V_{ms} \oplus W_h$.
\end{proof}

Due to $V_h = V_{ms} \oplus W_h$, $W_h$ is also the kernel space of the linear map $\mathcal{C}$.
Furthermore, combining an iterative Caccioppoli-type argument~\cite{Hou2017,li2018computing,owhadi2017multigrid,PETERSEIM2017} and some refined assumption for the potential, and the multiscale finite element basis functions have the following exponential decaying property.

\begin{proposition}[\cite{Wu2022}, Theorem 3.2]
  Under \cref{assump:potential} with a resolution constant as in \cite{Wu2022}, there exists a constant $\beta \in (0, 1)$ independent of $H$ and $\epsilon$, such that
  \begin{equation}
    \|\nabla\phi_p\|_{L^2(\mathcal{D} \backslash D_{\ell})} \leq \beta^{\ell} \|\nabla\phi_p\|,
    \label{equ:basis-functions-proposition}
  \end{equation}
  for all $p = 1, \cdots, N_H$.
\end{proposition}

In practice, localized basis functions are frequently used to reduce the computational complexity. Here we aim to construct the MsFEM for the NLSE with random potentials and estimate its approximation error, and thus the global basis functions are used. To accelerate computations, we also present a novel approach in Appendix A, in which the multiscale basis functions are approximated using the low-dimensional POD basis.

In the MsFEM space, the weak form of the full NLSE is discretized as
\begin{multline}
  i\epsilon\left(\sum_{p = 1}^{N_H}\sum_{s = 1}^{N_h}\mathrm{d}_t\hat{U}_pc_p^s\phi_s^h, \sum_{s = 1}^{N_h}c_l^s\phi_s^h\right) = \frac{\epsilon^2}{2}\left( \sum_{p = 1}^{N_H}\sum_{s = 1}^{N_h}\hat{U}_pc_p^s\nabla\phi_s^h, \sum_{s = 1}^{N_h}c_l^s\nabla\phi_s^h\right) \\+ \lambda\left( \left|\sum_{p = 1}^{N_H}\sum_{s = 1}^{N_h}\hat{U}_pc_p^s\phi_s^h\right|^2\sum_{p = 1}^{N_H}\sum_{s = 1}^{N_h}\hat{U}_pc_p^s\phi_s^h, \sum_{s = 1}^{N_h}c_l^s\phi_s^h\right)
  \label{equ:weak-form-multiscale-basis-fine-mesh}
\end{multline}
for all $l = 1,\cdots, N_H$. The stiff matrix and mass matrix constructed by the multiscale basis functions satisfy $M^{ms} = \mathcal{C}^T M^h\mathcal{C}$ and $S^{ms} = \mathcal{C}^T S^h\mathcal{C}$.
For the nonlinear term, the solution on the fine mesh is reconstructed by $\mathcal{C}\hat{U}$, and we then get the similar form ${N}^{ms} = \mathcal{C}^T {N}^h \mathcal{C}$.
The construction of ${N}^h$ suffers from heavy computation, especially for high-dimensional problems. The application of time-splitting methods can avoid this issue. Thus, we only need to solve linear equations at each time step, achieving high efficiency.

According to \eqref{equ:solution-ms_basis} and \eqref{equ:solution-ms_basis2}, the numerical solution on the coarse mesh can be denoted by $\{\hat{U}_p(t)\}_{p=1}^{N_H}$, while on the fine mesh, it is denoted by $\left\{\sum_{p = 1}^{N_H}\hat{U}_p(t)c_p^s\right\}_{s=1}^{N_h}$. For the sake of clarity, in the sequel, we denote the $\psi^{\epsilon}_h$ the classical FEM solution, and $\psi^{\epsilon}_H$ and $\psi^{\epsilon}_{H,h}$ the numerical solution constructed by the multiscale basis functions on the coarse mesh and fine mesh, respectively.

\section{Convergence analysis}
\subsection{Convergence analysis of the TS-FEM}
In this part, the {\bf SI} is mainly considered and the $L^2$ error will be estimated. We start the convergence analysis from the temporal error estimate at the initial time step.

\begin{lemma}
  If $\psi_{\mathrm{in}}\in H^4$, the error at the initial time step is bounded in the $L^2$ norm by
  $$\|\psi^{\epsilon}(\Delta t) - \psi^{\epsilon,1}\| = \|S^{\Delta t}\psi_{\mathrm{in}} - \mathcal{L}(\Delta t)\psi_{\mathrm{in}}\| \leq C\|\psi_{\mathrm{in}}\|_{H^4}\frac{\Delta t^3}{\epsilon^3},$$
  where $C$ is a constant.
\end{lemma}
\begin{proof}
  According to \eqref{equ:Strange-splitting-remnant}, we have
  \begin{align*}
    \psi^{\epsilon,1} &= \exp\left(-\frac{i\Delta t}{2\epsilon}\mathcal{L}_2(\hat{\psi})-\frac{i\Delta t}{\epsilon}\mathcal{L}_1-\frac{i\Delta t}{2\epsilon}\mathcal{L}_2(\psi^{\epsilon}_{\mathrm{in}})\right)\psi^{\epsilon}_{\mathrm{in}} \\
    &= \exp\left(-\frac{i\Delta t}{2\epsilon}\left(\mathcal{L}_2(\psi^{\epsilon}_{\mathrm{in}}) + \mathcal{O}(\frac{\Delta t^2}{\epsilon^2})\right)-\frac{i\Delta t}{\epsilon}\mathcal{L}_1-\frac{i\Delta t}{2\epsilon}\mathcal{L}_2(\psi^{\epsilon}_{\mathrm{in}})\right)\psi^{\epsilon}_{\mathrm{in}}\\
    &= \exp\left(-\frac{i\Delta t}{\epsilon}\mathcal{L}_1-\frac{i\Delta t}{\epsilon}\mathcal{L}_2(\psi^{\epsilon}_{\mathrm{in}})\right)\exp\left(-\frac{\Delta t^3}{\epsilon^3}\Gamma(2\mathcal{L}_1 + \mathcal{L}_2)^2\right)\psi^{\epsilon}_{\mathrm{in}},
  \end{align*}
  where $\Gamma$ depends on the form of $\mathcal{L}_2$. Use the expansion
  \begin{equation*}
    \exp\left(-\frac{\Delta t^3}{\epsilon^3}\Gamma(2\mathcal{L}_1 + \mathcal{L}_2)^2\right) = I - \frac{\Delta t^3}{\epsilon^3}\Gamma(2\mathcal{L}_1 + \mathcal{L}_2)^2 + \mathcal{O}\left(\frac{\Delta t^6}{\epsilon^6}\right)
  \end{equation*}
  and the dominant reminder has the form
  $$\mathcal{R}_1^0 = - \frac{\Delta t^3}{\epsilon^3}\Gamma(2\mathcal{L}_1 + \mathcal{L}_2)^2\psi^{\epsilon}_{\mathrm{in}}.$$
  Since the exact solution at $t = \Delta t$ is given by
  \begin{equation*}
    \psi^{\epsilon}(\Delta t) = S^{\Delta t}\psi^{\epsilon}_{\mathrm{in}} = \exp\left( -\frac{i\Delta t}{\epsilon}(\mathcal{L}_1+\mathcal{L}_2(\psi^{\epsilon}_{\mathrm{in}})) \right)\psi^{\epsilon}_{\mathrm{in}}.
  \end{equation*}
  There exists a constant such that
  \begin{equation*}
    \|\psi^{\epsilon}(\Delta t) - \psi^{\epsilon,1}\| \leq C\|\psi^{\epsilon}_{\mathrm{in}}\|_{H^4}\frac{\Delta t^3}{\epsilon^3}.
  \end{equation*}
\end{proof}

In turn, we prove the stability of the Strang splitting operator. Due to $\exp\left( -\frac{i\mathcal{L}_1t}{\epsilon} \right)$ being unitary, for any $f_1$, $f_2\in H^2$, we have
\begin{align*}
  \left\|\exp\left( -\frac{i\mathcal{L}_1t}{\epsilon} \right)f_1 - \exp\left( -\frac{i\mathcal{L}_1t}{\epsilon}\right)f_2\right\| = \left\|\exp\left( -\frac{i\mathcal{L}_1t}{\epsilon} \right)(f_1-f_2)\right\|
  = \|f_1 - f_2\|.
\end{align*}
Define $F(\psi) = -i\mathcal{L}_2(\psi)\psi$, the splitting solution for $\mathcal{L}_2$ is solved by the equation
\begin{equation}
    \epsilon\partial_t\psi - F(\psi) = 0.
\end{equation}
The nonlinear flow solved from this equation has the form
\begin{equation}
  Y^t\psi = \psi + \frac{1}{\epsilon}\int_0^tF(Y^s\psi)\mathrm{d}s.
  \label{equ:nonlinear-flow}
\end{equation}
Assume that $F$ is Lipschitz with a Lipschitz constant $M$, and repeat the proof in~\cite{Besse2003SIAM}. For all $f_1$, $f_2 \in L^2$, there exists a constant that depends on $F$ such that for all $0\leq \tau \leq 1$
\begin{align*}
  \|Y^{\tau}f_1 - Y^{\tau}f_2\| &\leq \|f_1 - f_2\| + \frac{1}{\epsilon}\int_0^{\tau}\|F(Y^sf_1) - F(Y^sf_2)\|\mathrm{d}s \\
  & \leq \|f_1 - f_2\| + \frac{M}{\epsilon}\int_0^{\tau}\|Y^sf_1 - Y^sf_2\|\mathrm{d}s.
\end{align*}
An application of the Gronwall lemma leads to
\begin{equation}
  \|Y^{\tau}f_1 - Y^{\tau}f_2\| \leq \exp\left(\frac{M {\tau}}{\epsilon}\right)\|f_1 - f_2\|.
\end{equation}
In particular, for $F(\psi) = \lambda|\psi|^2\psi$ we get
\begin{equation}
  \|\mathcal{L}({\tau})f_1 - \mathcal{L}({\tau})f_2\| \leq \exp\left(\frac{M\lambda {\tau}}{\epsilon}\right)\|f_1 - f_2\|.
  \label{equ:bound-splitting-operator}
\end{equation}
Besides, for the nonlinear flow \eqref{equ:nonlinear-flow}, we have the following lemma.
\begin{lemma}
    Let $\psi \in H^2$; if $F(\psi) = \lambda |\psi|^2\psi$, there exists a constant $C$ such that for all $0\leq \tau \leq 1$
    \begin{equation}
        \|Y^{\tau}\psi\|_{H^2} \leq \exp\left(\frac{ \lambda\tau\|\psi\|_{\infty}^2}{\epsilon}\right)\|\psi\|_{H^2}.
        \label{equ:nonlinear-flow-option1}
    \end{equation}
    If $F(\psi) = \lambda |\psi|^2\psi + v\psi$, there exists a constant $C$ such that for $v \in H^2$ and for all $0\leq \tau \leq 1$
    \begin{equation}
        \|Y^{\tau}\psi\|_{H^2} \leq \exp\left(\frac{\tau(\|v\|_{H^2} + \lambda\|\psi\|_{\infty}^2)}{\epsilon}\right)\|\psi\|_{H^2}.
        \label{equ:nonlinear-flow-option2}
    \end{equation}
    \label{lem:H2-norm-nonlinear-flow}
\end{lemma}
\begin{proof}
    Consider $F(\psi) = \lambda |\psi|^2\psi + v\psi$. For the nonlinear flow \eqref{equ:nonlinear-flow}, we have
    \begin{equation*}
        \|Y^{\tau}\psi\|_{\infty} \leq \|\psi\|_{\infty} + \frac{1}{\epsilon}\int_0^{\tau}\|F(Y^s\psi)\|_{\infty}\mathrm{d}s \leq  \|\psi\|_{\infty} + \frac{\|v\|_{\infty} + \lambda \|\psi\|_{\infty}^2}{\epsilon}\int_0^{\tau}\|Y^s\psi\|_{\infty}\mathrm{d}s.
    \end{equation*}
    Then the application of Gronwall inequality yields
    \begin{equation*}
        \|Y^{\tau}\psi\|_{\infty} \leq \exp\left(\frac{\tau(\|v\|_{\infty} + \lambda\|\psi\|_{\infty}^2)}{\epsilon}\right)\|\psi\|_{\infty}.
    \end{equation*}
    Similarly, for the $H^2$ norm, we directly have
    \begin{equation*}
        \|Y^{\tau}\psi\|_{H^2} \leq \|\psi\|_{H^2} + \frac{\|v\|_{H^2} + \lambda \|\psi\|_{\infty}^2}{\epsilon}\int_0^{\tau}\|Y^s\psi\|_{H^2}\mathrm{d}s,
    \end{equation*}
    which also leads to
    \begin{equation*}
        \|Y^{\tau}\psi\|_{H^2} \leq \exp\left(\frac{\tau(\|v\|_{H^2} + \lambda\|\psi\|_{\infty}^2)}{\epsilon}\right)\|\psi\|_{H^2}.
    \end{equation*}
    Let $v = 0$ and we get \eqref{equ:nonlinear-flow-option1}. This completes the proof.
\end{proof}

For the semi-discretized time-splitting methods, we have the convergence theorem of temporal accuracy.

\begin{theorem}
  Let $\psi_{\mathrm{in}}\in H^4$, $T > 0$ and $\Delta t\in (0, \epsilon)$. For $n\Delta t \leq T$, there exists a constant $C$ such that
  \begin{equation}
    \| \mathcal{L}^n\psi_{\mathrm{in}} - S^{n\Delta t}\psi_{\mathrm{in}} \|
    \leq
    CT \| \psi_{\mathrm{in}} \|_{H^4} \left(1 + \frac{T}{\epsilon}\right)\frac{\Delta t^2}{\epsilon^3}.
  \end{equation}
  \label{thm:temporal-convergence}
\end{theorem}
\begin{proof}
  Similar to the proof in~\cite{Besse2003SIAM,Stephane2001TSRDS}. The triangle inequality yields
  \begin{equation*}
    \|\mathcal{L}^n\psi_{\mathrm{in}} - S^{n\Delta t}\psi_{\mathrm{in}}\|
    \leq
    \sum_{j=0}^{n-1}\| \mathcal{L}^{n-j}S^{j\Delta t}\psi_{\mathrm{in}} - \mathcal{L}^{n-j-1}S^{(j+1)\Delta t}\psi_{\mathrm{in}} \|.
  \end{equation*}
  Due to $S^t$ being the Lie formula for all $t \leq T$ and $\psi_{\mathrm{in}} \in H^4$, $S^t\psi_{\mathrm{in}}$ belongs to $H^4$ and is uniformly bounded in this space, thus for all $j$ such that $j\Delta t\leq T$, we have
  $$\| \mathcal{L} S^{j\Delta t}\psi_{\mathrm{in}} - S^{(j+1)\Delta t}\psi_{\mathrm{in}} \|
  =
  \| (\mathcal{L} - S^{\Delta t})S^{j\Delta t}\psi_{\mathrm{in}} \|
  \leq
  C \| \psi_{\mathrm{in}} \|_{H^4} \frac{\Delta t^3}{\epsilon^3}.$$
  Combine with \eqref{equ:bound-splitting-operator} and we get
  \begin{align*}
    \| \mathcal{L}^n \psi_{\mathrm{in}} - S^{n\Delta t}\psi_{\mathrm{in}} \|
    &\leq
    \sum_{j=0}^{n-1}\Big(\exp\Big(\frac{M\lambda\Delta t}{\epsilon}\Big)\Big)^{n-j-1} \| (\mathcal{L} - S^{\Delta t})S^{j\Delta t}\psi_{\mathrm{in}} \|.
  \end{align*}
  Since $0 < \Delta t < \epsilon$, for all $j \geq 0$, we have
  \begin{equation*}
    \Big(\exp\Big(\frac{M\lambda\Delta t}{\epsilon}\Big)\Big)^{j} \leq \left( 1 + C_0\frac{\Delta t}{\epsilon} \right)^j \leq 1 + Cj\frac{\Delta t}{\epsilon}.
  \end{equation*}
  Consequently, we arrive at
  \begin{align*}
    &\| \mathcal{L}^n \psi_{\mathrm{in}} - S^{n\Delta t}\psi_{\mathrm{in}} \|
    \leq
    \sum_{j=0}^{n-1}\Big(\exp\Big(\frac{M\lambda\Delta t}{\epsilon}\Big)\Big)^{n-j-1}C \| \psi_{\mathrm{in}} \|_{H^4} \frac{\Delta t^3}{\epsilon^3} \\
    \leq&
    C \| \psi_{\mathrm{in}} \|_{H^4} \frac{\Delta t^3}{\epsilon^3} \sum_{j=0}^{n-1} \Big( 1 + C(n-j-1)\frac{\Delta t}{\epsilon}\Big)
    \leq
    CT \| \psi_{\mathrm{in}} \|_{H^4} \left(1 + \frac{T}{\epsilon}\right)\frac{\Delta t^2}{\epsilon^3}.
  \end{align*}
  It concludes the proof of this theorem.
\end{proof}

Next, we give the convergence of the full TS-FEM method.
Consider the problem
$$i\epsilon\partial_t\psi^{\epsilon} = \mathcal{L}_2\psi^{\epsilon}$$
with the initial condition $\psi_{\mathrm{in}}$ and the periodical boundary condition. The solution has the form
\begin{equation}
  \psi^{\epsilon}(\bx, t) = \exp\left(-\frac{it}{2\epsilon}\mathcal{L}_2\right)\psi_{\mathrm{in}}.
\end{equation}
If $\mathcal{L}_2$ consists of potential and nonlinear term, the regularity of $\psi^{\epsilon}(t, \bx)$ depends on the regularity of both the potential $v$ and $\psi_{\mathrm{in}}$, otherwise it only depends on $\psi_{\mathrm{in}}$.

Assume that the numerical solution $\psi^{\epsilon}_h$ is given by \eqref{equ:fully-TS-FEM} and $\psi^{\epsilon}(t_n) = S^{n\Delta t}\psi_{\mathrm{in}}$ is the solution of \eqref{equ:NLS_equ}. We write
\begin{equation}
  \psi^{\epsilon, n}_h - \psi^{\epsilon}(t_n)
  =
  {L}^n\psi_h^0 - S^{n\Delta t}\psi_{\mathrm{in}} =
  ({L}^n\psi_h^0 - \mathcal{L}^n\psi_{\mathrm{in}}) + (\mathcal{L}^{n}\psi_{\mathrm{in}} - S^{n\Delta t}\psi_{\mathrm{in}}).
  \label{equ:fully-error}
\end{equation}
The first term denotes the error attributable to the space discretization and the second term is the splitting error of temporal discretization.

We first estimate the spatial error accommodation from $t = 0$ to $t = \Delta t$,
\begin{align*}
  &\psi^{\epsilon, 1}_h - \psi^{\epsilon}(\Delta t)
  =
  {L}_2\left( \frac{\Delta t}{2},\cdot \right)\circ {L}_1(\Delta t) {L}_2\left( \frac{\Delta t}{2},\cdot \right) \circ \psi_h^0 - \mathcal{L}(\Delta t)\psi_{\mathrm{in}}.
\end{align*}
Let $\hat{\psi}_0 = \mathcal{L}_2(\frac{\Delta t}{2},\cdot)\circ\psi_{\mathrm{in}}$, and consider the problem
\begin{equation}
  i\epsilon\partial_t\psi^{\epsilon} = -\frac{\epsilon^2}{2}\Delta\psi^{\epsilon} + v\psi^{\epsilon}
  \label{equ:linear-schrodinger-equation}
\end{equation}
with the initial condition $\psi^{\epsilon}(t=0) = \hat{\psi}_0$ and the periodical boundary condition. The corresponding weak form is
\begin{align}
  i\epsilon(\partial_t(\psi^{\epsilon} - \psi^{\epsilon}_h), \phi^h) = \frac{\epsilon^2}{2}(\nabla(\psi^{\epsilon}-\psi^{\epsilon}_h), \nabla\phi^h) + (v(\psi^{\epsilon}-\psi^{\epsilon}_h), \phi^h), \quad \forall \phi^h \in V_h.
  \label{equ:error-weak-form}
\end{align}
Let $ \psi^{\epsilon} - \psi^{\epsilon}_h = (\psi^{\epsilon} - R_h\psi^{\epsilon}) + \theta$, where $\theta = R_h\psi^{\epsilon} - \psi^{\epsilon}_h$ and $R_h\psi^{\epsilon}$ denotes the Ritz projection. According to \eqref{equ:error-weak-form}, we get
\begin{equation}
  i\epsilon(\partial_t[(\psi^{\epsilon}-R_h\psi^{\epsilon})+\theta], \phi^h) = \frac{\epsilon^2}{2}(\nabla\theta,\nabla\phi^h) + (v(\psi^{\epsilon}-R_h\psi^{\epsilon}), \phi^h) + (v\theta,\phi^h).
\end{equation}
Take $\phi^h = \theta$ in the above equation,
\begin{equation*}
  i\epsilon(\partial_t\theta,\theta) = -i\epsilon(\partial_t(\psi^{\epsilon}-R_h\psi^{\epsilon}),\theta) + \frac{\epsilon^2}{2}\|\nabla\theta\|^2 + (v(\psi^{\epsilon}-R_h\psi^{\epsilon}), \theta) + (v\theta,\theta),
\end{equation*}
and we have
\begin{equation*}
  i\epsilon\mathrm{d}_t\|\theta\|^2 = i\epsilon(\partial_t\theta,\theta) + i\epsilon(\partial_t\bar{\theta}, \bar{\theta}) = 2i\epsilon \Re{(\partial_t({\psi^{\epsilon}}-R_h{\psi^{\epsilon}}), {\theta})} + 2i\Im{(v(\psi^{\epsilon}-R_h\psi^{\epsilon}), \theta)},
\end{equation*}
which induces
\begin{equation}
  \mathrm{d}_t\|\theta\| \leq 2\|\partial_t({\psi^{\epsilon}}-R_h{\psi^{\epsilon}})\| + \frac{2}{\epsilon}\|v\|_{\infty}\|\psi^{\epsilon}-R_h\psi^{\epsilon}\|.
\end{equation}
Integrating from $0$ to $t$ yields
\begin{equation}
  \|\theta(t)\|\leq \|\theta(0)\| + 2\int_0^t\|\partial_t({\psi^{\epsilon}}-R_h{\psi^{\epsilon}})\|\mathrm{d}t + \frac{2}{\epsilon}\|v\|_{\infty}\int_0^t\|\psi^{\epsilon}-R_h\psi^{\epsilon}\|\mathrm{d}t.
  \label{equ:error-evolution}
\end{equation}
Assume $\|\theta(0)\| = \|\hat{\psi}_{\mathrm{in}} - R_h\hat{\psi}_{\mathrm{in}}\| = \|{\psi}_{\mathrm{in}} - R_h{\psi}_{\mathrm{in}}\| = 0$.
Since $\|R_h\partial_t\psi^{\epsilon} - \partial_t\psi^{\epsilon}\| \leq Ch^2\|\partial_t\psi^{\epsilon}\|_{H^2}$, we have
% \begin{equation*}
%   \|R_h\partial_t\psi^{\epsilon} - \partial_t\psi^{\epsilon}\| \leq Ch^2\|\partial_t\psi^{\epsilon}\|_{H^2},
% \end{equation*}
% then we have
\begin{equation}
  \|\theta(t)\|\leq Cth^2\|\partial_t\psi^{\epsilon}\|_{H^2} + \frac{Ch^2}{\epsilon}\int_0^t\|\psi^{\epsilon}\|_{H^2}\mathrm{d}s \leq C_{\lambda,\epsilon}th^2 + \frac{C th^2}{\epsilon^{3}} \leq CC_{\lambda,\epsilon}th^2,
  \label{equ:estimate-thetat}
\end{equation}
where $t \leq \Delta t$, and $C_{\lambda,\epsilon}$ is the leading order term with respect to $\epsilon^{-1}$.

Let $\hat{\psi}_{h,1}$ be the numerical solution of \eqref{equ:linear-schrodinger-equation} with $t = \Delta t$, we can obtain
\begin{align*}
  \|\psi^{\epsilon, 1}_h - \psi^{\epsilon}(\Delta t)\|
  &=
  \left\|\exp\left( -\frac{i\Delta t \mathcal{L}_2(\hat{\psi}_{h,1})}{2\epsilon} \right) \hat{\psi}_{h,1} - \exp\left( -\frac{i\Delta t \mathcal{L}_2(\hat{\psi}_{1})}{2\epsilon} \right)\hat{\psi}_{1}\right\|
  \\&\leq
  C\exp\left( \frac{M\lambda\Delta t}{2\epsilon} \right) \| \theta(t) \|,
\end{align*}
where
$\hat{\psi}_{1} = \exp\left( -\frac{i\epsilon\Delta t \mathcal{L}_1}{\epsilon} \right) \exp\left( -\frac{i\epsilon\Delta t \mathcal{L}_2}{2\epsilon} \right)\psi_{\mathrm{in}}$.
This indicates the spatial error accumulation in a one-time step.
We next estimate the error accumulation in both time and space from $t = 0$ to $T$.

\begin{theorem}
  Assume that $\psi_h^{\epsilon, n} = {L}^n \psi_{\mathrm{in}}$ and $\psi^{\epsilon}(n\Delta t) = S^{n\Delta t}\psi_{\mathrm{in}}$ are the numerical solution and exact solution of the NLSE. Moreover, assume $\partial_t\psi^{\epsilon}\in H^2$ for all $t\in[0, T]$ and $\psi_{\mathrm{in}} \in H^4$. Then for a given $T > 0$, there exists a constant $h_0$ such that $h \leq h_0$ and for all $\Delta t < \epsilon$ with $ n \Delta t \leq T$, and the $L^2$ error estimate satisfies
  \begin{equation}
    \|\psi^{\epsilon, n}_h - \psi^{\epsilon}(n\Delta t)\| \leq CC_{\lambda,\epsilon} h^2 + CT\left(1 + \frac{T}{\epsilon}\right) \frac{\Delta t^2}{\epsilon^3},
    \label{equ:TSFEM-full-convergence}
  \end{equation}
  where the constant $C$ is independent of $\epsilon$ and $T$.
  \label{thm:classical-TSFEM-convergence}
\end{theorem}
\begin{proof}
  The error can be split into
  \begin{equation*}
    \psi^{\epsilon, n}_h - \psi^{\epsilon}(n\Delta t) = {L}^n \psi_h^0 - S^{n\Delta t}\psi_{\mathrm{in}} =
  ({L}^n \psi_h^0 - \mathcal{L}^n\psi_{\mathrm{in}}) + (\mathcal{L}^{n}\psi_{\mathrm{in}} - S^{n\Delta t}\psi_{\mathrm{in}}).
  \end{equation*}
  The first term on the right-hand side satisfies
  \begin{align*}
    \| L^{n} \psi_h^0 - \mathcal{L}^{n}\psi_{\mathrm{in}} \|
    &\leq \Big\|\sum_{j=1}^{n} {L}^{n-j}({L}R_h - R_h\mathcal{L})\mathcal{L}^{j-1}\psi_{\mathrm{in}}\Big\| + \|(R_h - I)\mathcal{L}^n \psi_{\mathrm{in}}\|.
  \end{align*}
  Due to $\mathcal{L}_1$ conserving the $H^2$ norm of the solution and \cref{lem:H2-norm-nonlinear-flow}, we have $\mathcal{L}^n\psi_{\mathrm{in}} \in H^2$ and $\|(R_h - I)\mathcal{L}^n \psi_{\mathrm{in}}\| \leq Ch^2\|\mathcal{L}^n\psi_{\mathrm{in}}\|_{H^2}$.
  Meanwhile,
  \begin{equation*}
    \| {L}\psi^{\epsilon}\| \leq \|{L}\psi^{\epsilon} - \mathcal{L}(\Delta t)\psi^{\epsilon}\| + \|\mathcal{L}(\Delta t)\psi^{\epsilon}\| \leq CC_{\lambda,\epsilon}\Delta th^2 + \|\psi^{\epsilon}\|.
  \end{equation*}
  Similar to the Theorem 3.1 in \cite{refId0}, we denote the bound of the numerical solution by
  $$\max\limits_{1\leq m\leq n} \| {L}^m R_h \mathcal{L}^{n-m}\psi^{\epsilon}\| \leq a_{L}.$$
  Recall \eqref{equ:error-evolution}-\eqref{equ:estimate-thetat}, owing to $\Delta t < \epsilon$, then there exists a constant $C$ independent of $\epsilon$ such that
  \begin{align*}
    &\left\| \sum_{j=1}^{n} {L}^{n-j}( {L}R_h - R_h\mathcal{L})\mathcal{L}^{j-1}\psi_{\mathrm{in}} \right\|
    \leq
    n\exp\left( CTa_{L}^2 \right)\max\limits_{1\leq j\leq n} \|({L}R_h - R_h\mathcal{L})\mathcal{L}^{j-1}\psi_{\mathrm{in}}\|
    \\ \leq&
    n\exp\left( CTa_{L}^2 \right)\exp\left(\frac{\lambda M\Delta t}{\epsilon}\right)CC_{\lambda,\epsilon}\Delta t h^2
    \leq
    \exp\left( CTa_{L}^2 \right)\exp\left(\frac{\lambda M\Delta t}{\epsilon}\right)CC_{\lambda,\epsilon}T h^2.
  \end{align*}
  Thus we arrive at
  \begin{equation*}
    \| {L}^n \psi_{\mathrm{in}} - \mathcal{L}^n\psi_{\mathrm{in}}\| \leq CC_{\lambda,\epsilon} h^2,
  \end{equation*}
  where $C$ is independent of $\epsilon$ but depends on $T$ and $\lambda$. Note that the order of $\|\psi^{\epsilon}\|_{H^2}$ with respect to $\epsilon^{-1}$ is lower than $C_{\lambda,\epsilon}$, and it is ignored in this results.

  Furthermore, combine with \Cref{thm:temporal-convergence}, and we get the desired estimate
  \begin{align*}
    \| \psi^{\epsilon, n}_h - \psi^{\epsilon}(n\Delta t) \|
    &\leq \| {L}^n\psi_{\mathrm{in}} - \mathcal{L}^n\psi_{\mathrm{in}} \| + \| \mathcal{L}^{n}\psi_{\mathrm{in}} - S^{n\Delta t}\psi_{\mathrm{in}} \|
    \\&\leq
    CC_{\lambda,\epsilon} h^2 + CT\left(1 + \frac{T}{\epsilon}\right) \frac{\Delta t^2}{\epsilon^3}.
  \end{align*}
  This declares the \eqref{equ:TSFEM-full-convergence}.
\end{proof}

\begin{remark}
    Take a further simplification
  \begin{equation*}
    \frac{C}{\epsilon^3}\left( 1 + \frac{T}{\epsilon} \right) \leq \frac{CT}{\epsilon^4}.
  \end{equation*}
  We temporarily use $\psi^{\epsilon, n}_H$ to denote the FEM solution on the coarse mesh with mesh size $H$, the counterpart result of \cref{thm:classical-TSFEM-convergence} on the coarse space is
  \begin{equation}
      \|\psi^{\epsilon, n}_H - \psi^{\epsilon}(n\Delta t)\| \leq CC_{\lambda,\epsilon} H^2 + \frac{CT^2}{\epsilon^4} \Delta t^2.
    \label{equ:TSFEM-full-convergence2}
  \end{equation}
  \label{rem:FEM-convergence-coarse-space}
\end{remark}

We have obtained the $L^2$ error estimate of the TS-FEM applied to the deterministic NLSE. Next, we will further assess the convergence analysis of the MsFEM in space, in conjunction with the qMC method. Note that the convergence analysis for the TS-FEM combined with the qMC method follows a similar pattern. Therefore, we will not discuss the convergence analysis of the TS-FEM in random space in this section.

\subsection{Convergence analysis of the TS-MsFEM for NLSE with random potentials}

In this part, we first present a convergence analysis of the TS-MsFEM for the NLSE with the deterministic potential. Secondly, by employing the qMC method in the random space, we further obtain the error estimate of the TS-MsFEM applied to the NLSEs with random potentials.

\subsubsection{TS-MsFEM for the deterministic NLSE}
For {\bf SI}, we solve the linear Schr\"odinger equation by the MsFEM and the corresponding convergence analysis has been given in \cite{Wu2022}. We therefore have the following estimate.
\begin{lemma}
  Let $\psi_{H}^{\epsilon, n} = {L}_{ms}^n\psi_{\mathrm{in}}$ be the numerical solution solved in $V_{ms}$ by {\bf SI}, and $\psi^{\epsilon}(t_n) = S^{n\Delta t}\psi_{\mathrm{in}}$ be the exact solution of the NLSE. Let $\Delta t \in (0, \epsilon)$, and assume $\partial_t\psi^{\epsilon}\in L^2$ for all $t\in(0, T]$, and $\psi_{\mathrm{in}} \in H^4$. We have the estimate
  \begin{equation}
    \|\psi^{\epsilon, n}_{H} - \psi^{\epsilon}(t_n)\| \leq \frac{CTH^2}{\epsilon^3} + \frac{CT^2}{\epsilon^4}\Delta t^2,
    \label{equ:estimate-MsFEM-linear}
  \end{equation}
  where the constant $C$ is independent of $\epsilon$.
  \label{lem:TS-MsFEM-convergence}
\end{lemma}
\begin{proof}
  For the linear Schr\"odinger equation, the spatial error of multiscale solution and exact solution has the bound~\cite{Wu2022}
  \begin{align*}
    \|\psi^{\epsilon}_H - \psi^{\epsilon} \| \leq \frac{CH^2}{\epsilon^2}\|\epsilon\partial_t\psi^{\epsilon}\| \leq \frac{CH^2}{\epsilon}  \|\partial_{t}\psi_{\mathrm{in}}\|\exp\left( \frac{2\lambda t\| \psi^{\epsilon} \|_{\infty}^2}{\epsilon}\right).
  \end{align*}
  At the second step of {\bf SI}, we have
  \begin{equation*}
    \|\psi^{\epsilon}_H - \psi^{\epsilon} \| \leq \frac{CH^2}{\epsilon^2} \exp\left( \frac{2\lambda \Delta t \| \psi^{\epsilon} \|_{\infty}^2}{\epsilon}\right) \leq \frac{CH^2}{\epsilon^2}.
  \end{equation*}
  When the eigendocomposition method is applied, the solution can be solved exactly in time for linear problems.
  The accumulation of the spatial error at each time step satisfies
  \begin{align*}
    &\| {L}_{ms}\psi_{H}^{\epsilon, n} - \mathcal{L}\psi^{\epsilon, n} \|
    \leq
    \| {L}_{ms}\psi_{H}^{\epsilon, n} - \mathcal{L} I_H \psi^{\epsilon, n} \| + \| \mathcal{L} I_H \psi^{\epsilon, n} - \mathcal{L}\psi^{\epsilon, n} \|
    \\ \leq&
    \exp\left( \frac{\lambda M\Delta t}{2\epsilon} \right) \frac{CH^2}{\epsilon^2} + \exp\left( \frac{\lambda M\Delta t}{\epsilon} \right)\| I_H \psi^{\epsilon, n} - \psi^{\epsilon, n} \|
    \leq
    \exp\left( \frac{\lambda M\Delta t}{\epsilon} \right) \frac{CH^2}{\epsilon^2}.
  \end{align*}
   Meanwhile, by the Strang splitting method, repeat the procedures in \Cref{thm:temporal-convergence}, and we get the estimate as \eqref{equ:estimate-MsFEM-linear}.
\end{proof}

\begin{remark}
      In comparison to \cref{rem:FEM-convergence-coarse-space}, the MsFEM exhibits superior performance with respect to the dependence on $\epsilon$, as it requires only the bound $\|\partial_t\psi^{\epsilon}\|$. In contrast, the application of the classical FEM requires the bound of $\|\partial_t\psi^{\epsilon}\|_{H^2}$, which implies a stronger dependence on $\epsilon$. Consequently, the weaker dependence of MsFEM on $\epsilon$ demonstrates its superiority in effectively handling multiscale problems.
\end{remark}

\subsubsection{MsFEM for the NLSE with random potentials}
To carry out the convergence analysis for the qMC method, the regularity of the wave function with respect to random variables is required. The random potential is truncated by the $m$-order KL expansion, and we denote $\boldsymbol{\xi}(\omega) = (\xi_1(\omega), \cdots, \xi_m(\omega))^T$. Let $\boldsymbol{\nu} = (\nu_1, \cdots, \nu_m)$ be the multi-index with $\nu_j$ being the nonnegative integer, where $|\boldsymbol{\nu}| = \sum_{j=1}^m \nu_j$. Then $\partial^{\boldsymbol{\nu}}\psi^{\epsilon}_m$ denotes the mixed derivative of $\psi^{\epsilon}_m$ with respect to all random variables specified by the multi-index $\boldsymbol{\nu}$.

\begin{lemma}
  For any $\omega\in\Omega$ and multi-index $|\boldsymbol{\nu}| < \infty$, and for all $t \in (0, T]$, there exists a constant $C(T,\lambda,\epsilon, |\boldsymbol{\nu}|)$ depends on $T,\lambda,\epsilon, |\boldsymbol{\nu}|$ such that the partial derivative of $\psi^{\epsilon}_m(t,\bx,\omega)$ satisfies the priori estimate
  \begin{equation}
    \|\partial^{\boldsymbol{\nu}}\psi_m\|_{H^2} \leq C(T,\lambda,\epsilon, |\boldsymbol{\nu}|)\prod_{j}(\sqrt{\lambda_j}\|v_j\|_{H^2})^{\nu_j}.
  \end{equation}
  \label{lem:regualrity-psi-wrt-random-variables}
\end{lemma}

The proof of this lemma is given in the appendix.

We are interested in the expectation of linear functionals of the numerical solution in applications of uncertainty quantification. We will estimate the expected value $\mathds{E}[\mathcal{G}(\psi_m^{\epsilon}(\cdot, \omega))]$ of the random variable $\mathcal{G}(\psi_m^{\epsilon}(\cdot, \omega))$. Let $\mathcal{G}(\cdot)$ be a continuous linear functional on $L^2(\mathcal{D})$, then there exists a constant $C_{\mathcal{G}}$ such that

\begin{equation*}
  |\mathcal{G}(u)| \leq C_{\mathcal{G}}\|u\|
\end{equation*}
for all $u\in L^2(\mathcal{D})$. Consider the integral
\begin{equation}
  I_m(F) = \int_{\boldsymbol{\xi}\in[0,1]^m}F(\boldsymbol{\xi})\mathrm{d}\boldsymbol{\xi},
  \label{equ:integrand}
\end{equation}
where $F(\boldsymbol{\xi}) = \mathcal{G}(\psi_m^{\epsilon}(\cdot, \boldsymbol{\xi}))$. To approximate this integral, both the MC and qMC can be used. In our methods, it is approximated over the unit cube by randomly shifted lattice rules
\begin{equation*}
  Q_{m,n}(\boldsymbol{\Delta}; F) = \frac 1N\sum_{i=1}^NF\left(frac\left(\frac{iz}{N} + \boldsymbol{\Delta}\right)\right),
\end{equation*}
where $z\in\mathds{N}^m$ is the generating vector and $\boldsymbol{\Delta}\in [0, 1]^m$. Here $N$ denotes the number of random samples.

\begin{lemma}
  For the integral~\eqref{equ:integrand}, given $m, N\in\mathds{N}$ with $N \leq 10^{30}$, weights $\gamma = (\gamma_{\mathbf{u}})_{\mathbf{u}\subset\mathds{N}}$, a randomly shifted lattice rule with $N$ points in $m$ dimensional random space could be constructed by a component-by-component such that for all $\alpha\in (\frac 12, 1]$
  \begin{equation*}
    \sqrt{\mathds{E}^{\boldsymbol{\Delta}}|I_m(F) - Q_{m,N}(\cdot; F)|} \leq 9C^*C_{\gamma,m}(\alpha)N^{-1/2\alpha},
  \end{equation*}
  where
  \begin{equation*}
    C_{\gamma,m}(\alpha) = \left(\sum_{\emptyset\neq\mathbf{u}\subseteq\{1:m\}}
    \gamma^{\alpha}_{\mathbf{u}}\prod_{j\in\mathbf{u}}\varrho(\alpha)\right)^{1/2\alpha}
    \left(\sum_{\mathbf{u}\subseteq\{1:m\}}\frac{(C(\boldsymbol{\nu}))^2}{\gamma_{\mathbf{u}}}
    \prod_{j\in\mathbf{u}}\lambda_j\|v_j\|_{H^2}^2\right)^{1/2}.
  \end{equation*}
  \label{lem:error-of-qmc}
\end{lemma}
\begin{proof}
  The proof of the lemma is the same as in~\cite{doi:10.1137/19M127389X}. Here $C(\boldsymbol{\nu}) = C(t,\lambda,\epsilon, |\boldsymbol{\nu}|)$ is calculated in \Cref{lem:regualrity-psi-wrt-random-variables}. And
  \begin{equation}
    \varrho(\alpha) = 2\left( \frac{\sqrt{2\pi}}{\pi^{2-2\eta_{*}(1-\eta_*)\eta_*}} \right)^{\alpha}
    \zeta\left( \alpha + \frac 12 \right),
  \end{equation}
  where $\eta_* = \frac {2\alpha - 1}{4\alpha}$, $\zeta(x)$ is the Riemann zeta function and $C^* = \| \mathcal{G} \|$. The details of these estimates can be found in~\cite{dick_kuo_sloan_2013,Graham2015}.
\end{proof}

Employing the qMC method, the estimate between the wave functions of \eqref{equ:NLS_equ} and the truncated NLSE~\eqref{equ:NLS_equ_parameterized} satisfies the following lemma.
\begin{lemma}
  Under the \cref{assump:random-potential-assumption}, there exists a constant $C$ such that
  \begin{equation}
    \sqrt{\mathds{E}^{\boldsymbol{\Delta}}[|\mathds{E}[\mathcal{G}(\psi^{\epsilon})] - Q_{m,N}[\mathcal{G}(\psi^{\epsilon}_m)]|^2]} \leq C\left(\frac{m^{-\chi}}{\epsilon} + C_{\gamma,m}N^{-r}\right),
  \end{equation}
  where $0 \le \chi \leq (\frac 12-\eta)\Theta-\frac 12$, $r = 1-\delta$ for $0 < \delta < \frac 12$. Note that the constant $C$ is independent of $m$ and $n$ but depends on $T$.
  \label{lem:solution-estimate-KLequ-orequ}
\end{lemma}
\begin{proof}
  Since $\mathcal{G}$ is a linear functional, we have
  \begin{align*}
    |\mathds{E}[\mathcal{G}(\psi^{\epsilon})] - Q_{m,N}[\mathcal{G}(\psi^{\epsilon}_m)]| &\leq |\mathds{E}[\mathcal{G}(\psi^{\epsilon})] - I_m(\psi^{\epsilon})| + |I_m(\psi^{\epsilon}) - Q_{m,N}[\mathcal{G}(\psi^{\epsilon}_m)]| \\
    &= |\mathds{E}[\mathcal{G}(\psi^{\epsilon})] - \mathds{E}[\mathcal{G}(\psi_m^{\epsilon})]| + |I_m(\psi^{\epsilon}) - Q_{m,N}[\mathcal{G}(\psi^{\epsilon}_m)]|.
  \end{align*}
  The first term satisfies
  \begin{align*}
    |\mathds{E}[\mathcal{G}(\psi^{\epsilon})] - \mathds{E}[\mathcal{G}(\psi_m^{\epsilon})]| \leq \mathds{E}[|\mathcal{G}(\psi^{\epsilon}) -\mathcal{G}(\psi_m^{\epsilon})|]
    \leq C\frac{m^{-\chi}}{\epsilon},
  \end{align*}
  where $C$ depends on the time $T$.
  Let $\alpha = 1/(2-2\delta)$ for $0 < \delta < \frac 12$, according to \Cref{lem:error-of-qmc}, we then get
  \begin{align*}
    &\mathds{E}^{\boldsymbol{\Delta}}[|\mathds{E}[\mathcal{G}(\psi^{\epsilon})] - Q_{m,N}[\mathcal{G}(\psi^{\epsilon}_m)]|^2] \\
    \leq &
    \mathds{E}^{\boldsymbol{\Delta}}[|\mathds{E}[\mathcal{G}(\psi^{\epsilon})] - I_m(\psi^{\epsilon})|^2] + \mathds{E}^{\boldsymbol{\Delta}}[|I_m(\psi^{\epsilon}) - Q_{m,N}[\mathcal{G}(\psi^{\epsilon}_m)]|^2] \\
    \leq &C\frac{m^{-2\chi}}{\epsilon^2} + CC_{\gamma,m}^2N^{2-2\delta}.
  \end{align*}

\end{proof}

Employ the qMC method in the random space, for the numerical solution $\psi^{\epsilon,m}_H$ solved by MsFEM on the coarse mesh, and we have the following error estimate.
\begin{theorem}
  Let $\psi_{\mathrm{in}} \in H^4(\mathcal{D})$, $\psi^{\epsilon} \in L^{\infty}([0, T]; H^4(\mathcal{D})) \cap L^{1}([0, T]; H^2(\mathcal{D}))$, and parameterized potentials satisfy the \cref{assump:random-potential-assumption}. Consider $\mathds{E}[\mathcal{G}(\psi^{\epsilon}(t_n))]$ is approximated by $Q_{m,N}(\cdot; \mathcal{G}(\psi^{\epsilon,n}_{H,m}))$. Apply the random shifted lattice rule $Q_{m,N}$ to $\mathcal{G}(\psi^{\epsilon}(t_n))$. Then for any fixed $T > 0$, there exists a constant $H_0$ such that $H \leq H_0$ and for all $\Delta t < \epsilon$ with $n\Delta t \leq T$, we have the root-mean-square error as
  \begin{equation}
    \sqrt{\mathds{E}^{\boldsymbol{\Delta}}[|\mathds{E}[\mathcal{G}(\psi^{\epsilon}(t_n))] - Q_{m,N}[\mathcal{G}(\psi^{\epsilon,n}_{H,m})]|^2]} \leq C\left(\frac{H^2}{\epsilon^3} + \frac{\Delta t^2}{\epsilon^4} + \frac{m^{-\chi}}{\epsilon} + C_{\gamma,m}N^{-r}\right),
    \label{equ:globa-error-qMC-TSFEM}
  \end{equation}
  where $0 \le \chi \leq (\frac 12-\eta)\Theta-\frac 12$, and $r = 1-\delta$ for $0 < \delta < \frac 12$. Here $C$ is independent of $m$ and $N$ but depends on $\lambda$ and $T$, and $C_{\gamma,m}$ depends on $T$, $\lambda$ and $\epsilon$.
  \label{thm:L2-estimate-TS-MsFEM-qMC}
\end{theorem}
\begin{proof}
  We split the error \eqref{equ:globa-error-qMC-TSFEM} into
  \begin{align*}
    |\mathds{E}[\mathcal{G}(\psi^{\epsilon}(t_n))] - Q_{m,N}[\mathcal{G}(\psi^{\epsilon,n}_{H,m})]| \leq
    & |\mathds{E}[\mathcal{G}(\psi^{\epsilon}(t_n))] - Q_{m,N}[\mathcal{G}(\psi^{\epsilon}_m(t_n))]|
    \\&
    + |Q_{m,N}[\mathcal{G}(\psi^{\epsilon}_m(t_n))]] - Q_{m,N}[\mathcal{G}(\psi^{\epsilon,n}_{H,m})]|.
  \end{align*}
  The second term can be estimated by
  \begin{align*}
    |\mathcal{G}(\psi^{\epsilon}_m(t_n)) - \mathcal{G}(\psi^{\epsilon,n}_{H,m})|
    % &= |\mathcal{G}(\psi^{\epsilon}_m(t_n) - \psi^{\epsilon,n}_{H,m})|
    \leq C_{\mathcal{G}} \| \psi^{\epsilon}_m(t_n) - \psi^{\epsilon,n}_{H,m} \|
    \leq CC_{\mathcal{G}}\left(\frac{H^2}{\epsilon^3} + \frac{\Delta t^2}{\epsilon^4}\right),
  \end{align*}
  where the constant $C$ depends on $\lambda$ and $T$, and is independent of $m$ and $N$.
  Combine with \Cref{lem:solution-estimate-KLequ-orequ}, we get the \eqref{equ:globa-error-qMC-TSFEM}. This completes this proof.
\end{proof}

\begin{remark}
    \cref{thm:L2-estimate-TS-MsFEM-qMC} gives the $L^2$ estimate of TS-MsFEM for the NLSE with random potentials. For the employment of the TS-FEM, repeat the above procedures and we can get a similar result.
\end{remark}

In the proposed methods, when accounting for random potentials, constructing multiscale basis functions demands substantial computational cost as the number of samples grows. To improve the simulation efficiency, we propose a multiscale reduced basis method consisting of offline and online stages. In the offline stage, we utilize the proper orthogonal decomposition~(POD) method to derive a small set of multiscale POD bases. Using these multiscale POD bases, the computational cost of solving optimal problems in the online stage can be further reduced. Detailed information about this method can be found in Appendix~\ref{sec:reduced-MsFEM}.

\section{Numerical experiments}
In this part, we will present numerical experiments in both 1D and 2D physical space. The convergence rates of TS-FEM and TS-MsFEM are first verified. For the NLSE with the random potential, we compare the convergence rate in the random space. In addition, the delocalization of mass distribution due to disordered potentials and the cubic nonlinearity is investigated.

\subsection{Numerical accuracy of TS-FEMs}
Set $\psi_{\mathrm{in}}(x) = (10\pi)^{0.25}\exp(-20x^2)$ for the 1D case,
and $\psi_{\mathrm{in}}(x_1, x_2) = (10/\pi)^{0.25}\exp(-5(x_1-0.5)^2 - 5(x_2-0.5)^2)$ for the 2D case.
To begin with, we choose the harmonic potential $v(x) = 0.5x^2$, and verify the second-order accuracy of the TS-FEM with respect to the temporal step size $\Delta t$ and spatial mesh size $h$. Here we fix the terminal time $T = 1.0$, $\epsilon = \frac 1{16}$ and nonlinear parameter $\lambda = 0.1$. The reference solution $\psi^{\epsilon}_{\mathrm{ref}}$ is computed on the fine mesh with $h = \frac{2\pi}{2048}$ and $\Delta t$ = 1.0e-06. The $L^{2}$ absolute error and $H^1$ absolute error are recorded in~\Cref{tab:convergence-TS-FEM-space-time}.
\begin{table}[htbp]
  \centering
  \caption{Numerical convergence of TS-FEMs in space and time.}
  \begin{tabular}{||c|c|c|c|c|c|c||}
    \hline
     & $h$ & $\frac{2\pi}{128}$ & $\frac{2\pi}{256}$ & $\frac{2\pi}{512}$ & $\frac{2\pi}{1024}$ & order \\
    \hline
    \multirow{2}{*}{\bf SI} & $L^{2}$ error & 1.96e-02 & 5.22e-03 & 1.26e-03 & 2.54e-04 & 2.09 \\
     & $H^1$ error & 1.19e-01 & 3.36e-02 & 8.31e-03 & 1.68e-03 & 2.04 \\
    \hline
    \multirow{2}{*}{\bf SII} & $L^{2}$ error & 3.04e-02 & 8.07e-03 & 1.95e-03 & 3.92e-04 & 2.09 \\
     & $H^1$ error & 3.52e-01 & 9.95e-02 & 2.44e-02 & 4.92e-03 & 2.05 \\
    \hline
    \hline
     & $\Delta t$ & 4.0e-02 & 2.0e-02 & 1.0e-02 & 5.0e-03 & order \\
    \hline
    \multirow{2}{*}{\bf SI} & $L^{2}$ error & 4.53e-04 & 1.13e-04 & 2.81e-05 & 7.03e-06 & 2.00 \\
     & $H^1$ error & 2.09e-03 & 5.20e-04 & 1.30e-04 & 3.24e-05 & 2.00 \\
    \hline
    \multirow{2}{*}{\bf SII} & $L^{2}$ error & 7.16e-03 & 1.87e-03 & 4.71e-04 & 1.18e-04 & 1.98 \\
     & $H^1$ error & 1.12e-01 & 2.91e-02 & 7.26e-03 & 1.81e-03 & 1.99 \\
    \hline
  \end{tabular}
  \label{tab:convergence-TS-FEM-space-time}
\end{table}

For the 2D case, we employ the multiscale potential
\begin{equation}
  v(x_1, x_2) = \cos\left(x_1x_2+\frac{x_1}{\epsilon} + \frac{x_1x_2}{\epsilon^2}\right),
\end{equation}
over $\mathcal{D} = [0, 1]^2$ with $64\times 64$ spatial nodes. Here we set $\lambda = 1.0$ and multiscale coefficient $\epsilon = \frac 18$. We compare the numerical solution with the different $\Delta t$ for {\bf SI} and {\bf SII}. By the means of the numerical tests shown in~\Cref{fig:2d-TS-FEMs-different-time-step}, {\bf SI} allows a bigger time step size than {\bf SII}.
\begin{figure}[tbhp]
  \centering
  \subfloat[{\bf SI}, $\Delta t =$ 1.0e-2.]{\includegraphics[width=2.2in]{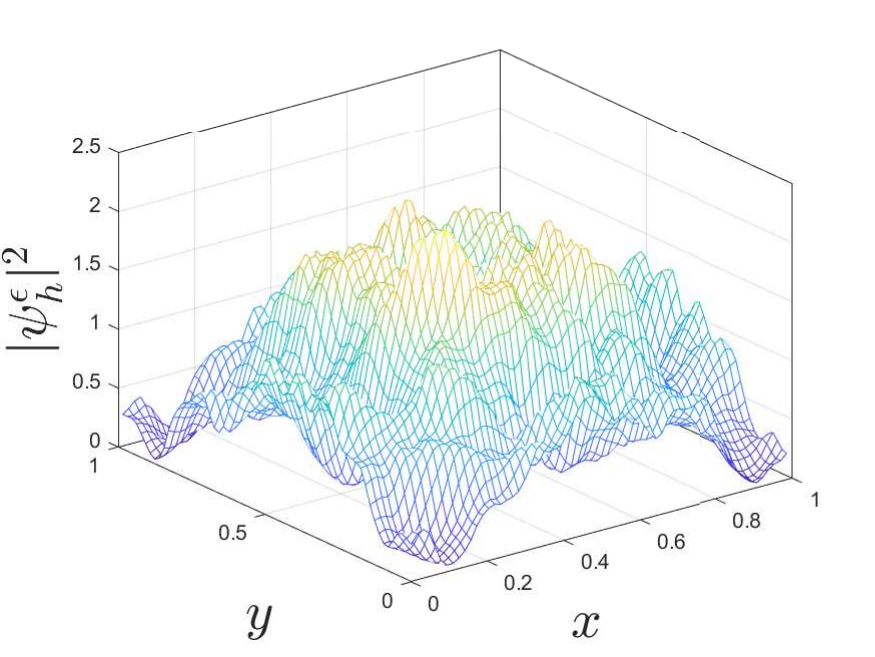}}
  \subfloat[{\bf SI}, $\Delta t =$ 1.0e-3.]{\includegraphics[width=2.2in]{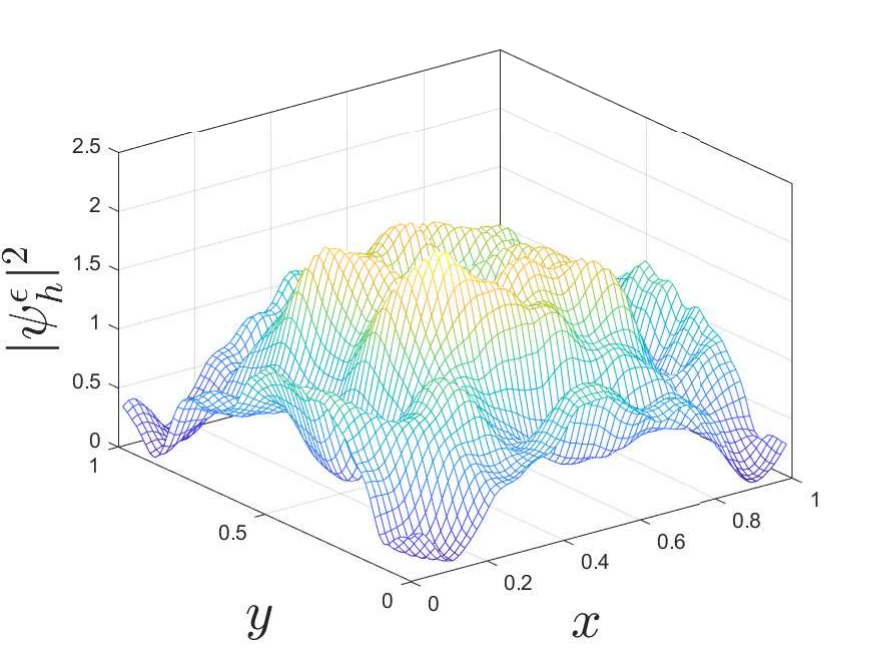}}
  \subfloat[{\bf SII}, $\Delta t =$ 1.0e-3.]{\includegraphics[width=2.2in]{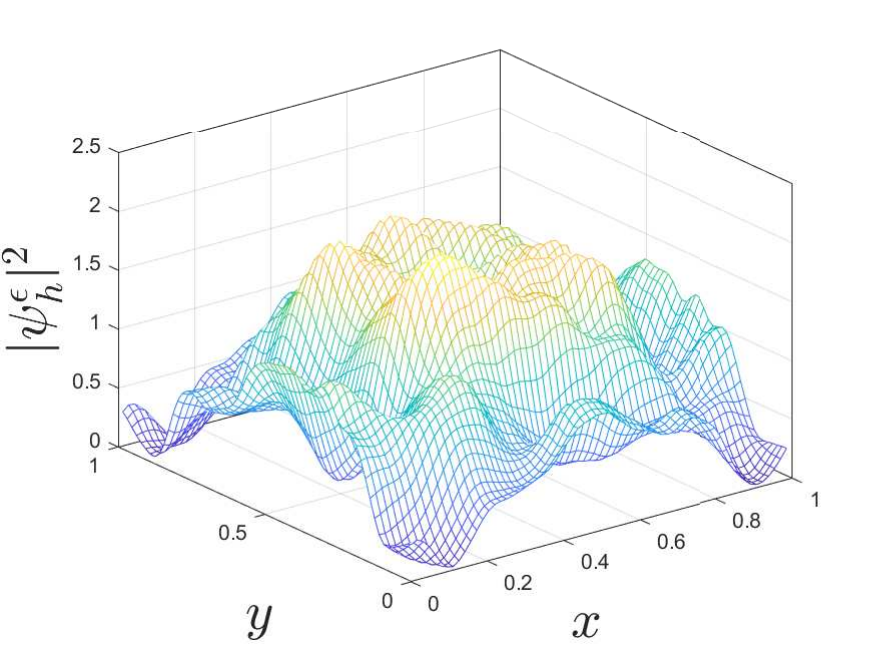}}
  \caption{Numerical solution computed by the two TS-FEMs with different $\Delta t$.}
  \label{fig:2d-TS-FEMs-different-time-step}
\end{figure}

\subsection{Numerical experiments of TS-MsFEMs}
In this study, we consider two forms of the multiscale solution: $\psi^{\epsilon}_H$ on the coarse mesh and $\psi^{\epsilon}_{H,h}$ on the fine mesh. We begin by employing the harmonic potential and varying the values of $H$. We then record the error between the numerical solution and the reference solution in~\Cref{tab:convergence-MsFEM-1D}. The simulation parameters used are: $\lambda = 0.1$, $\epsilon = \frac 1{16}$, $T = 1.0$, $\Delta t =$1.0e-03, and a fine mesh size of $h = \frac{2\pi}{4096}$. Our results show that {\bf SI} achieves a second-order convergence rate in both the coarse and fine spaces. Additionally, superconvergence is exhibited in the coarse space for {\bf SII}.

\begin{table}[htbp]
  \caption{Numerical convergence rate of the TS-MsFEMs for the NLSE with harmonic potential in space.}
  \begin{tabular}{||c|c|c|c|c|c||}
    \hline
    & $H$ & $\|\psi_{H,h}^{\epsilon} - \psi_{\mathrm{ref}}^{\epsilon}\|$ & $\|\psi_{H,h}^{\epsilon} - \psi_{\mathrm{ref}}^{\epsilon}\|_{H^1}$ & $\|\psi_H^{\epsilon} - \psi_{\mathrm{ref}}^{\epsilon}\|$ & $\|\psi_{H}^{\epsilon} - \psi_{\mathrm{ref}}^{\epsilon}\|_{H^1}$ \\
    \hline
    \multirow{5}{*}{\bf SI} & $\frac{2\pi}{2048}$ & 4.95e-05 & 4.69e-04 & 3.47e-05 & 3.31e-04 \\
    & $\frac{2\pi}{1024}$ & 1.68e-04 & 1.60e-03 & 1.18e-04 & 1.13e-03 \\
    & $\frac{2\pi}{512}$ & 6.44e-04 & 6.11e-03 & 4.52e-04 & 4.32e-03 \\
    & $\frac{2\pi}{256}$ & 2.56e-03 & 2.43e-02 & 1.80e-03 & 1.72e-02 \\
    \cline{2-6}
    & order & 1.90 & 1.90 & 1.90 & 1.90 \\
    \hline
    \multirow{5}{*}{\bf SII} & $\frac{2\pi}{2048}$ & 1.79e-05 & 1.73e-04 & 5.43e-12 & 1.88e-10 \\
    & $\frac{2\pi}{1024}$ & 6.10e-05 & 5.86e-04 & 7.85e-11 & 1.63e-09 \\
    & $\frac{2\pi}{512}$ & 2.33e-04 & 2.24e-03 & 5.68e-09 & 1.02e-07 \\
    & $\frac{2\pi}{256}$ & 9.24e-04 & 8.89e-03 & 4.49e-07 & 8.24e-06 \\
    \cline{2-6}
    & order & 1.90 & 1.90 & 5.52 & 5.22 \\
    \hline
  \end{tabular}
  \label{tab:convergence-MsFEM-1D}
\end{table}

Meanwhile, to demonstrate the advantages of Option 1, we examine the example of a discontinuous potential, as shown in~\Cref{fig:convergence-MsFEM-discontinuous}. We observe that {\bf SI} maintains its second-order spatial convergence rate, whereas the convergence rate of {\bf SII} deteriorates.

\begin{figure}[htbp]
  \centering
  \subfloat[$v(\bx)$.]{\includegraphics[width=2.2in]{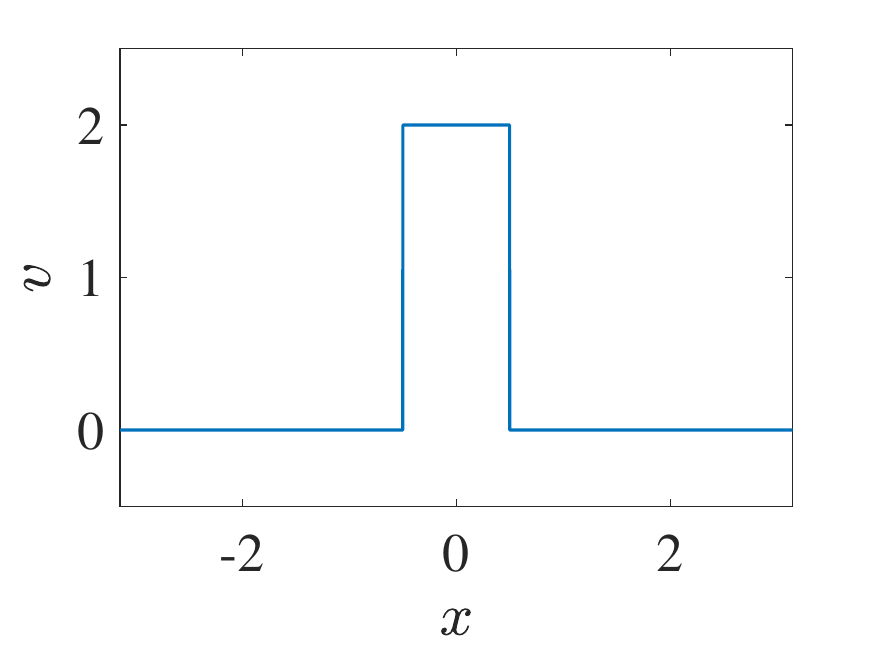}}
  \subfloat[{\bf SI}.]{\includegraphics[width=2.2in]{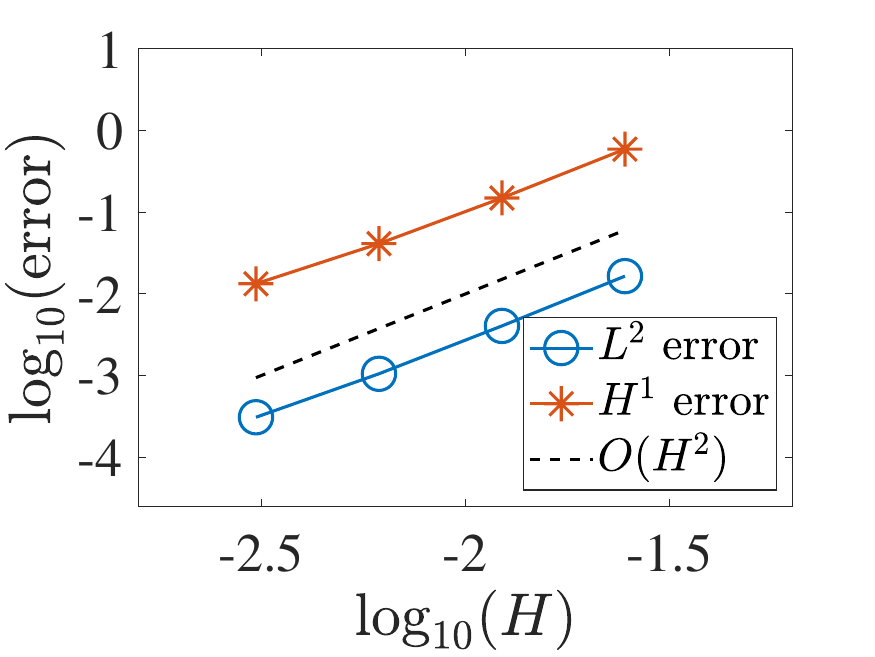}}
  \subfloat[{\bf SII}.]{\includegraphics[width=2.2in]{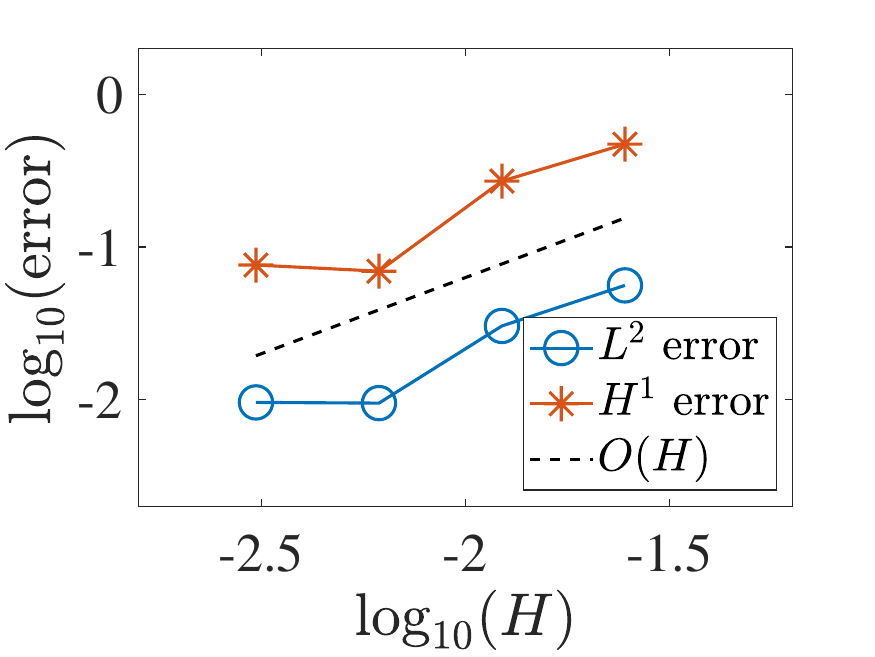}}
  \caption{Numerical convergence rate of {\bf SI} and {\bf SII} for the discontinuous potential. In the plots, the $L^2$ error and $H^1$ error on the coarse mesh are depicted.}
  \label{fig:convergence-MsFEM-discontinuous}
\end{figure}

Furthermore, we consider the small semiclassical constant $\epsilon = \frac 1{128}$ and the discontinuous potential as in \Cref{fig:convergence-MsFEM-discontinuous}(a). As shown in \Cref{fig:small-epsilon}, better approximations are provided by the MsFEM in the physical space.
\begin{figure}[htbp]
    \centering
    \subfloat[{\bf SI}.]{\includegraphics[width=2.5in]{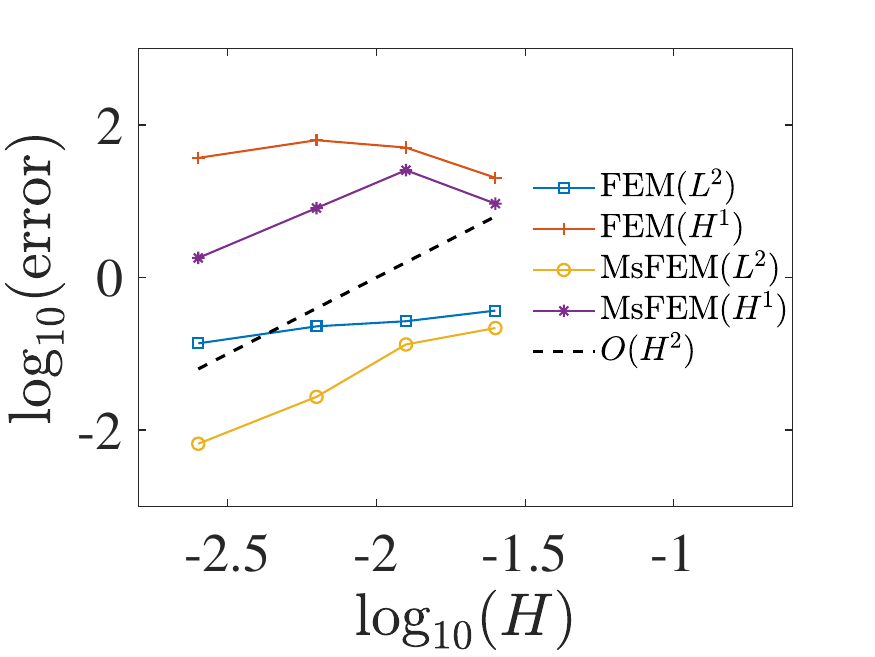}}
    \subfloat[{\bf SII}.]{\includegraphics[width=2.5in]{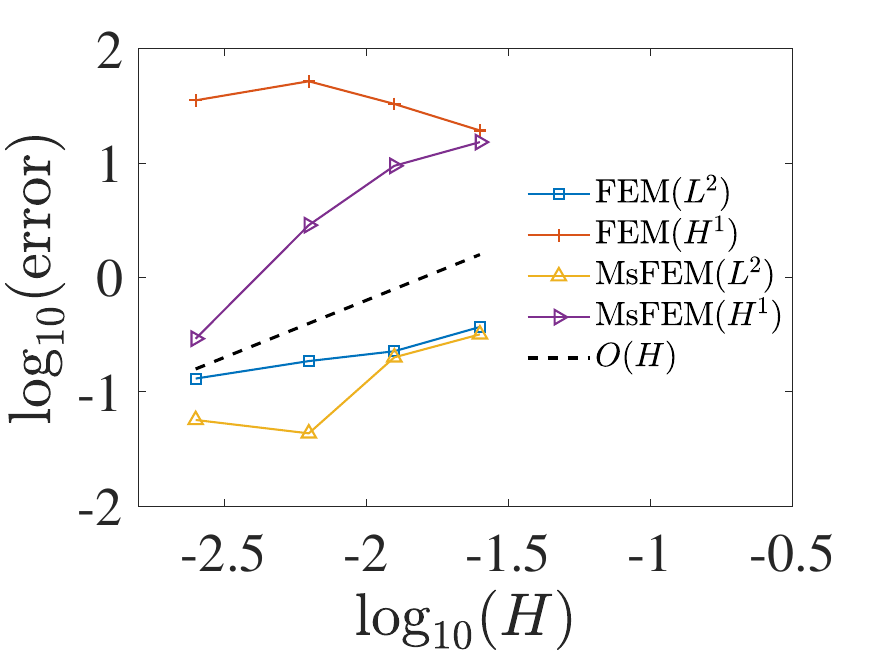}}
    \caption{The convergence rates of FEM and MsFEM for the NLSE with the discontinuous potential and semiclassical constant $\epsilon = \frac{1}{128}$.}
    \label{fig:small-epsilon}
\end{figure}

For the 2D case, we consider the discontinuous checkboard potential
\begin{equation*}
  v_2 = \left\{\begin{aligned}
    &\left(\cos\left(2\pi\frac{x_1}{\epsilon_2}\right) + 1\right)\left(\cos\left(2\pi\frac{x_2}{\epsilon_2}\right) + 1\right), &[0, 0.5]^2 \cup [0.5, 1]^2,\\
    &\left(\cos\left(2\pi\frac{x_1}{\epsilon_1}\right) + 1\right)\left(\cos\left(2\pi\frac{x_2}{\epsilon_1}\right) + 1\right), &\text{otherwise},
  \end{aligned}\right.
\end{equation*}
where $v = v_1 + v_2$ with $v_1 = |x_1-0.5|^2 + |x_2-0.5|^2$, $\epsilon_1 = \frac 18$ and $\epsilon_2 = \frac 16$. In the simulations, we set $h = \frac{1}{128}$, $\epsilon = \frac 1{4}$, $\lambda = 1.0$, $\Delta t = $1.0e-04 and $T = 1.0$. We employ {\bf SI}~(\cref{fig:2D-results-checkboard-SI}) and {\bf SII}~(\cref{fig:2D-results-checkboard-SIII}) for time evolution. We vary the coarse mesh size with $H = 4h$ and $H = 8h$ of the MsFEM and present the corresponding spatial error distribution. Here the reference solution is obtained using the FEM with a mesh size of $h$. In both \cref{fig:2D-results-checkboard-SI} and \cref{fig:2D-results-checkboard-SIII},
we observe a significant error when the MsFEM is used with a mesh size ratio of $H = 8h$. With the mesh being refined, the smaller error distribution in space can be obtained for \textbf{SI}. Hence this simulation demonstrates the superior performance of {\bf SI} when dealing with discontinuous potentials.

\begin{figure}[htbp]
    \centering
    \includegraphics[width=2.2in]{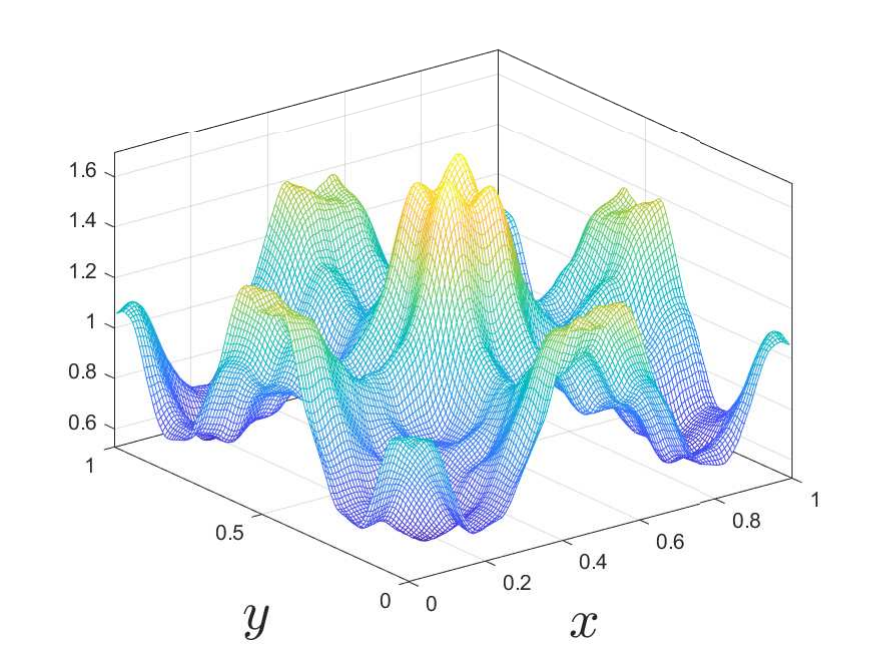}
    \includegraphics[width=2.2in]{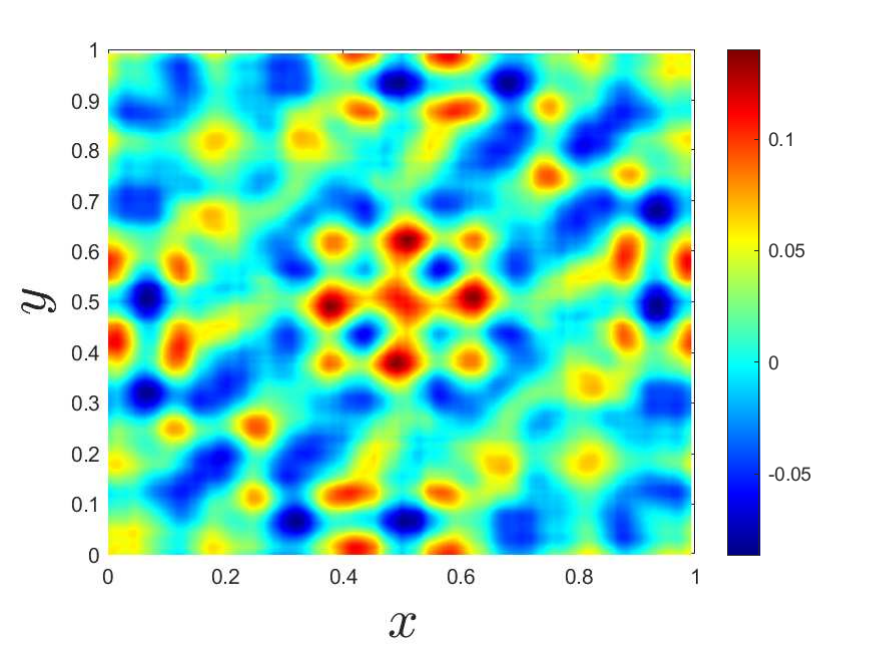}
    \includegraphics[width=2.2in]{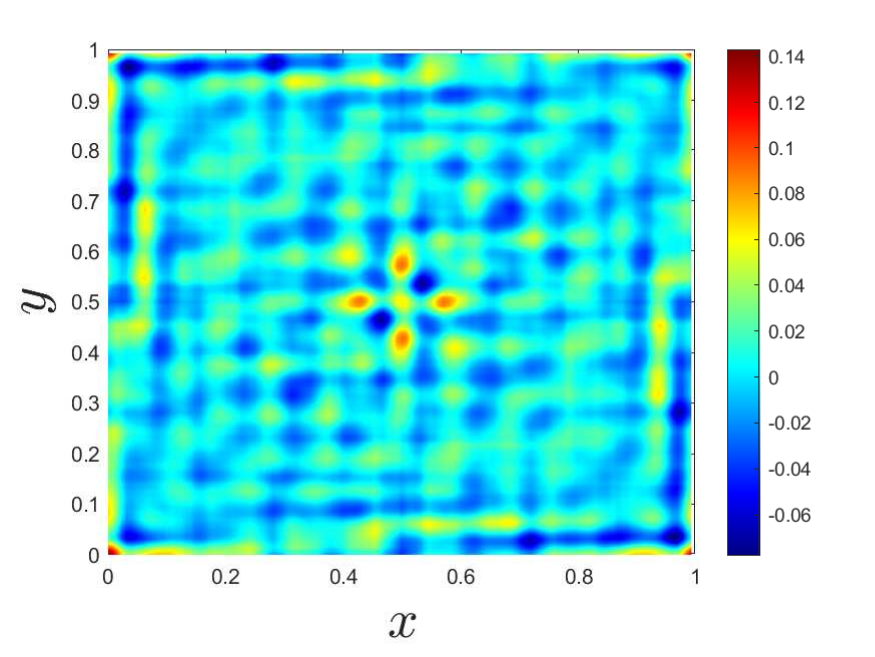}
    %\subfloat[Numerical solution computed by FEM, MsFEM with $H = 8h$ and MsFEM with $H = 4h$.]{
%    \includegraphics[width=2.2in]{u_h_nonlinear_SI_128times128_FEM.pdf}
%    \includegraphics[width=2.2in]{u_Hh_nonlinear_SI_128times128_k_8.pdf}
%    \includegraphics[width=2.2in]{u_Hh_nonlinear_SI_128times128_k_4.pdf}}
%    \quad
%    \subfloat[Spatial error distribution of MsFEM with $H = 8h$ and $H = 4h$.]{\includegraphics[width=2.2in]{error_fine_nonlinear_SI_128times128_k_8_msFEM_FEM.pdf} \hspace{4mm}
%    \includegraphics[width=2.2in]{error_fine_nonlinear_SI_128times128_k_4_msFEM_FEM.pdf}}
    \caption{Reference solution (FEM) and the spatial error distribution computed by {\bf SI}, in which the MsFEM is used with $H = 8h$ and $H = 4h$.}
    \label{fig:2D-results-checkboard-SI}
\end{figure}
\begin{figure}[htbp]
    \centering
    \includegraphics[width=2.2in]{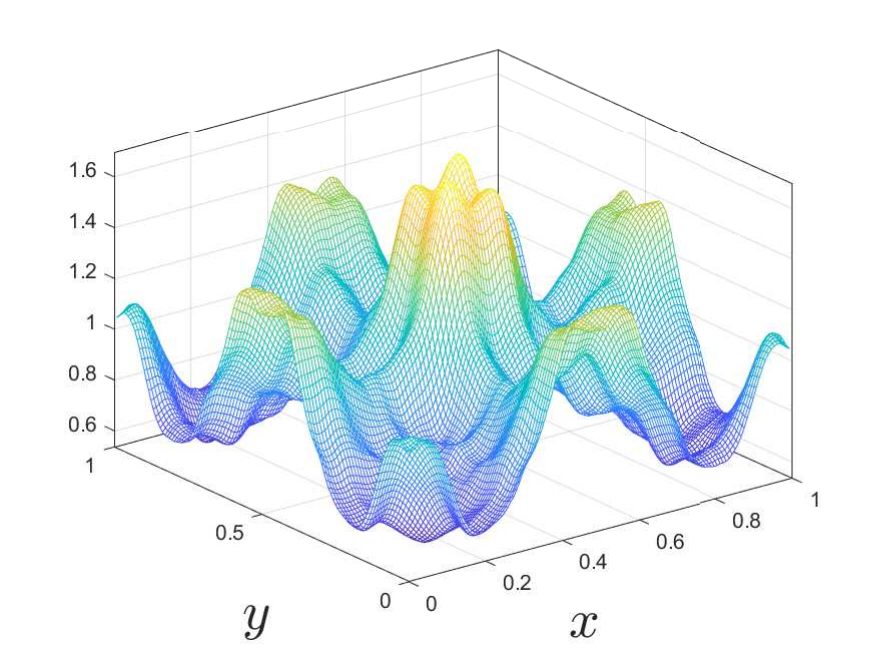}
    \includegraphics[width=2.2in]{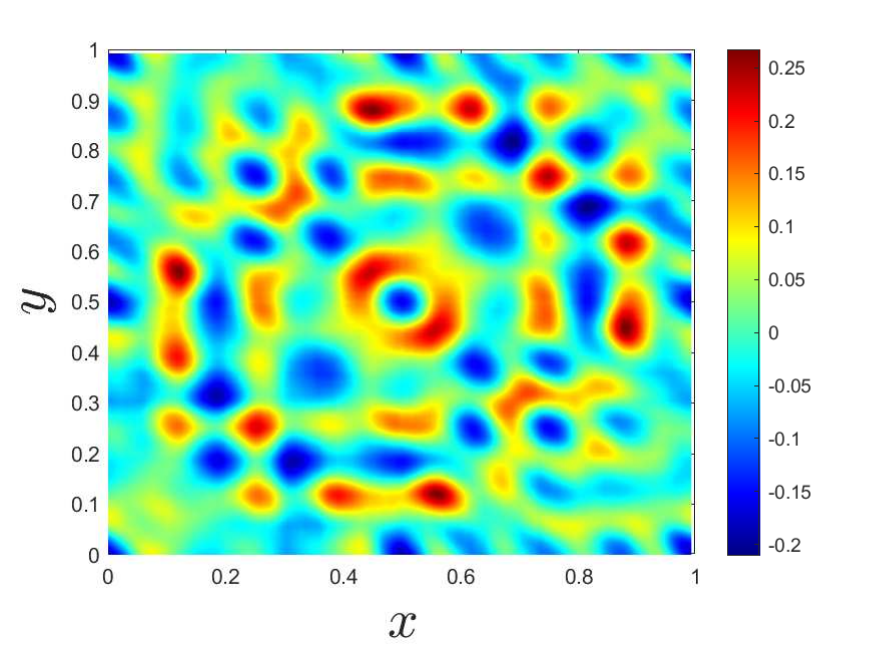}
    \includegraphics[width=2.2in]{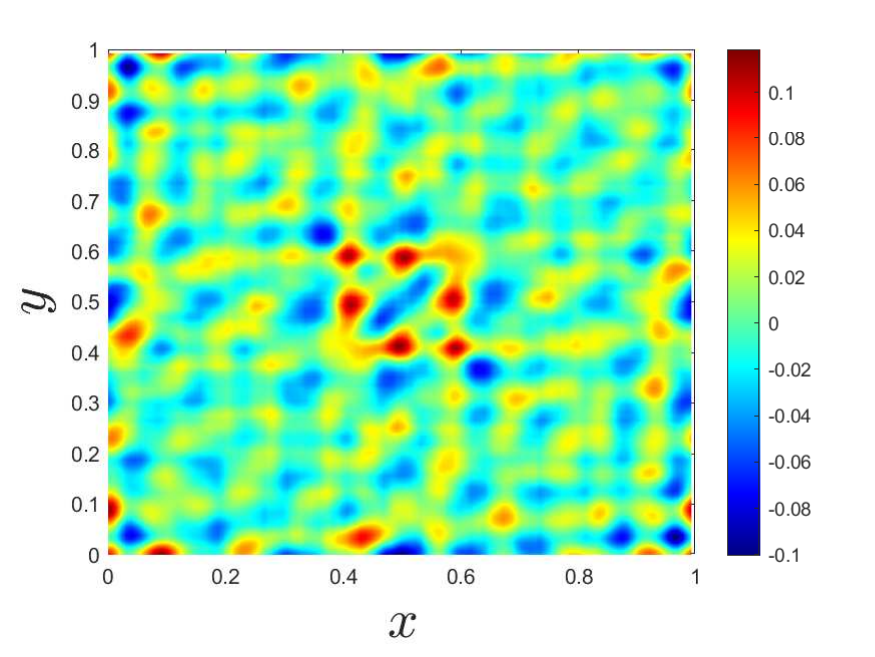}
    %\subfloat[Numerical solution computed by FEM, MsFEM with $H = 8h$ and MsFEM with $H = 4h$.]{
%    \includegraphics[width=2.2in]{u_h_nonlinear_SIII_128times128_FEM.pdf}
%    \includegraphics[width=2.2in]{u_Hh_nonlinear_SIII_128times128_k_8.pdf}
%    \includegraphics[width=2.2in]{u_Hh_nonlinear_SIII_128times128_k_4.pdf}}
%    \quad
%    \subfloat[Error distribution in space of MsFEM with $H = 8h$ and $H = 4h$.]{\includegraphics[width=2.2in]{error_fine_nonlinear_SIII_128times128_k_8_msFEM_FEM.pdf} \hspace{4mm}
%    \includegraphics[width=2.2in]{error_fine_nonlinear_SIII_128times128_k_4_msFEM_FEM.pdf}}
    \caption{Reference solution (FEM) and the spatial error distribution computed by {\bf SII}, in which the MsFEM is used with $H = 8h$ and $H = 4h$.}
    \label{fig:2D-results-checkboard-SIII}
\end{figure}

\subsection{Numerical simulations of NLSE with random potentials}
For the 1D case, we consider the random potential
\begin{equation}
  v(x, \omega) = \sigma\sum_{j=1}^m\sin(jx)\frac 1{j^{\beta}}\xi_j(\omega),
  \label{equ:1D-random-potential}
\end{equation}
where $\sigma$ controls the strength of randomness, and $\xi_j(\omega)$'s are mean-zero and i.i.d random variables uniformly distributed in $[-\sqrt{3}, \sqrt{3}]$. It is extended to 2D as
\begin{equation}
  v(x_1, x_2, \omega) = \sigma\sum_{j=1}^m\sin(jx_1)\sin(jx_2)\frac 1{j^{\beta}}\xi_j(\omega).
  \label{equ:2D-random-potential}
\end{equation}
For comparison, we employ the MC method and qMC method to generate the samples $\xi_j(\omega)$ in the simulations. And
we measure the states of the system by the expectation of mass density
\begin{equation*}
  \mathds{E}(|\psi^{\epsilon}_{H,h}|^2) = \frac {1}{N}\sum_i|\psi^{\epsilon}_{H,h}(\omega_i)|^2,
\end{equation*}
where $N$ denotes the number of MC or qMC samples. To observe the evolution in the mass distribution of the system, we introduce the definition
\begin{equation}
  A(t) = \mathds{E}\left( \int_{\mathcal{D}}|\bx|^2|\psi^{\epsilon}|^2\mathrm{d}\bx \right),
\end{equation}
which is extensively used to indicate the Anderson localization of the Schr\"odinger equation with random potentials.

\subsubsection{Comparison of FEM and MsFEM}\label{subsec:random-test1}
We set $\sigma = 1.0$, $\beta = 0$ and $m = 5$ in \eqref{equ:1D-random-potential}, and the number of qMC samples to be 500. The multiscale parameter is $\epsilon = \frac 18$, and the computational domain is $\mathcal{D} = [-2, 2]$. For the TS-FEMs, the solution is computed on the fine mesh with $h = \frac{2\pi}{600}$, and we set $H = 6h$ for the TS-MsFEMs. The terminal time is set to be $T = 10$. As shown in~\Cref{fig:FEM-MsFEM-random-potential-qMC500}, we show the evolution of $A(t)$ and $\mathds{E}(|\psi^{\epsilon}_{H,h}|^2)$ at $T = 10$. The localization of linear Schr\"odinger equation and weak delocalization of NLSE can be observed by both $A(t)$ and $\mathds{E}(|\psi^{\epsilon}_{H,h}|^2)$.
\begin{figure}[htbp]
  \centering
  \subfloat[Linear case.]{\includegraphics[width=1.6in]{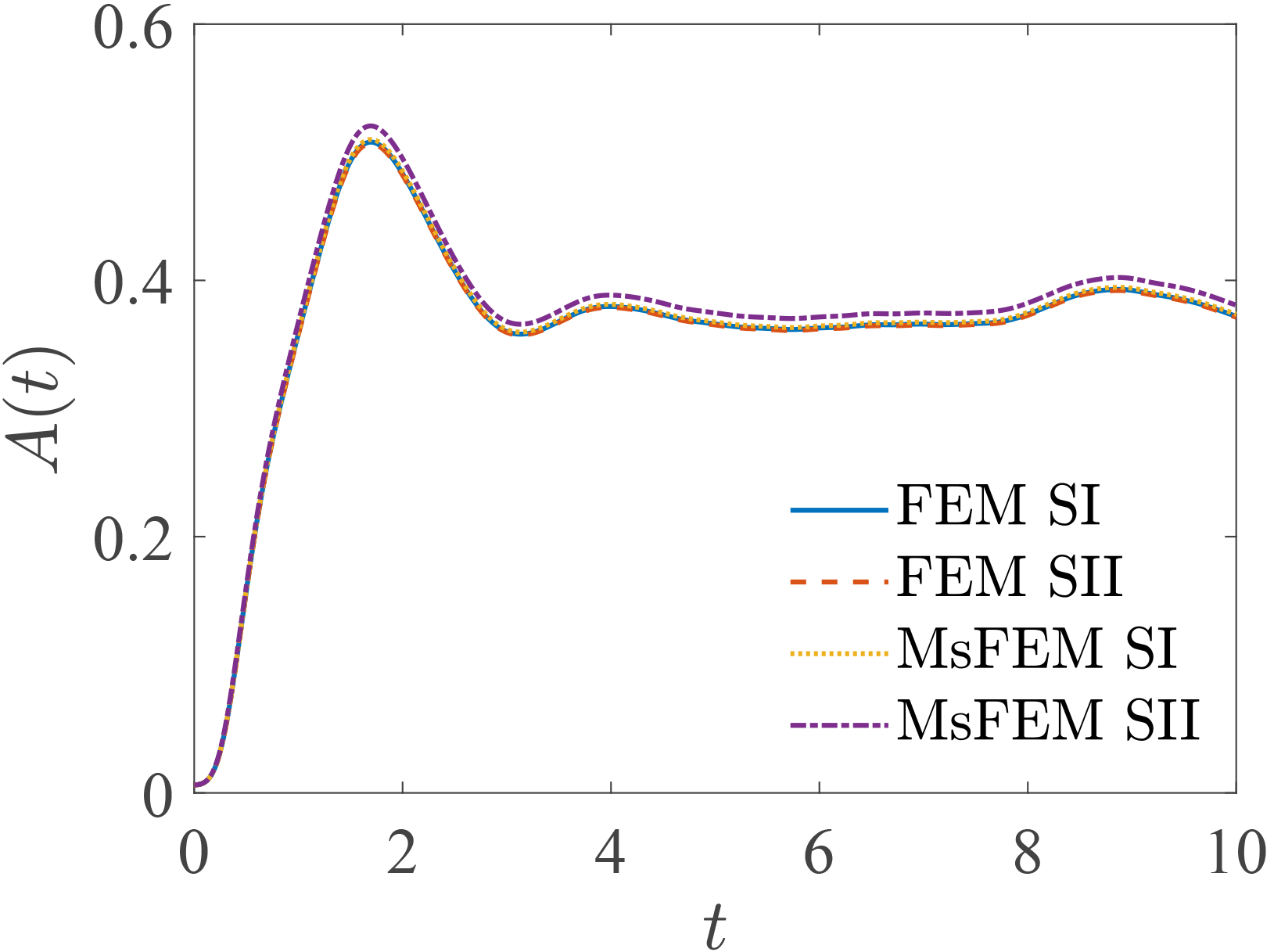}
  \includegraphics[width=1.6in]{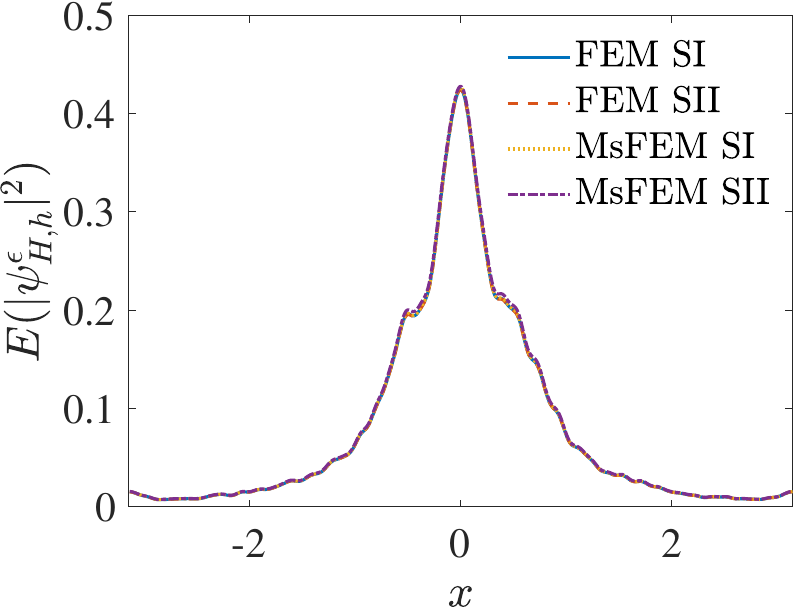}}
  \subfloat[Nonlinear case with $\lambda = 1.0$.]{\includegraphics[width=1.6in]{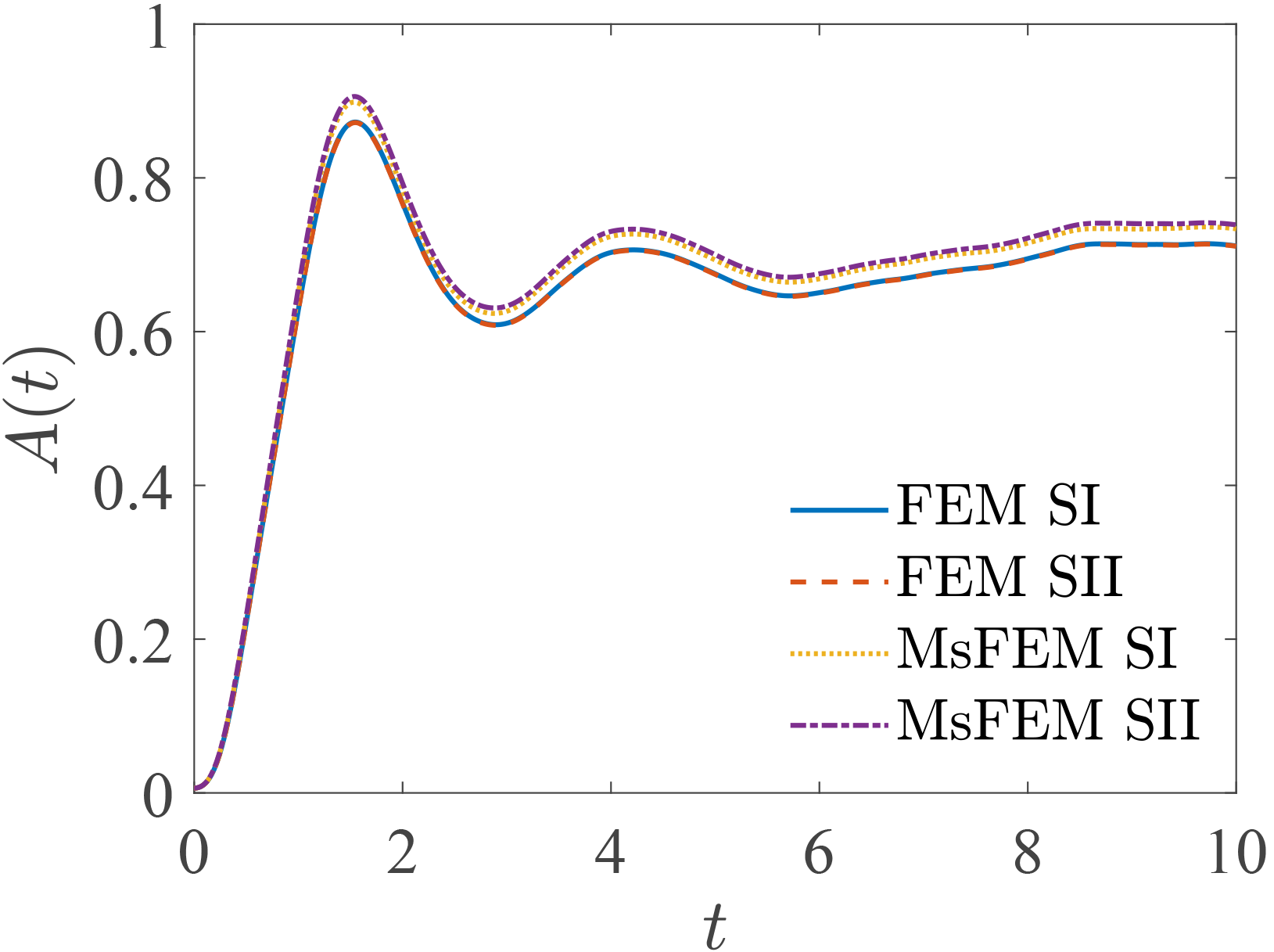}
  \includegraphics[width=1.6in]{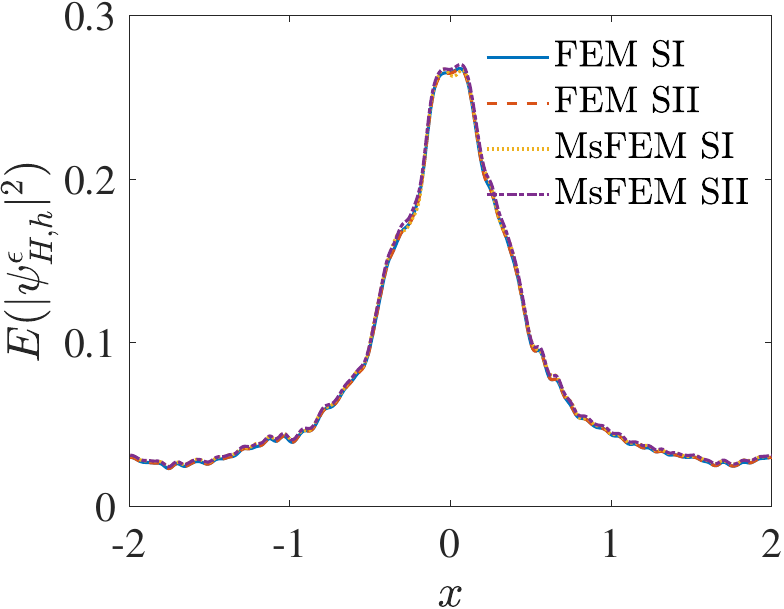}}
  \caption{Numerical results computed by FEM and MsFEM with different time-splitting methods for the NLSE with $\lambda = 0$ and $\lambda = 1.0$.}
  \label{fig:FEM-MsFEM-random-potential-qMC500}
\end{figure}

\subsubsection{Convergence of MC sampling and qMC sampling}
The MC method and qMC method exhibit different convergence rates. To eliminate the perturbation of a small sample size, we adopt the random potential
\begin{equation}
  v(x,\omega) = 1.0 + \sigma\sum_{j=1}^m\sin(jx)\frac 1{j^{\beta}}\xi_j(\omega),
\end{equation}
in which the parameters are: $\sigma = 1.0$, $\beta = 2.0$, $m = 5$. The other simulation settings are: $\lambda = 0.1$, $\epsilon = \frac 18$, $\mathcal{D} = [-\pi,\pi]$, $h = \frac{2\pi}{600}$, $H = 6h$, $T = 1.0$ and $\Delta t = $1.0e-03. In this experiment, we use $50000$ samples to compute the reference solution and record the $L^2$ error of the density
$\|\mathds{E}(|\psi_\mathrm{num}^{\epsilon}|^2) - \mathds{E}(|\psi_{\mathrm{ref}}^{\epsilon}|^2)\|$
as the sampling number varies with $N = 100$, $200$, $400$, $800$, $1600$ and $3200$ for both MC method and qMC method. The result is shown in \Cref{fig:convergence-rate-MC-qMC}.
\begin{figure}[htbp]
  \centering
  % Requires \usepackage{graphicx}
  \includegraphics[width=2.5in]{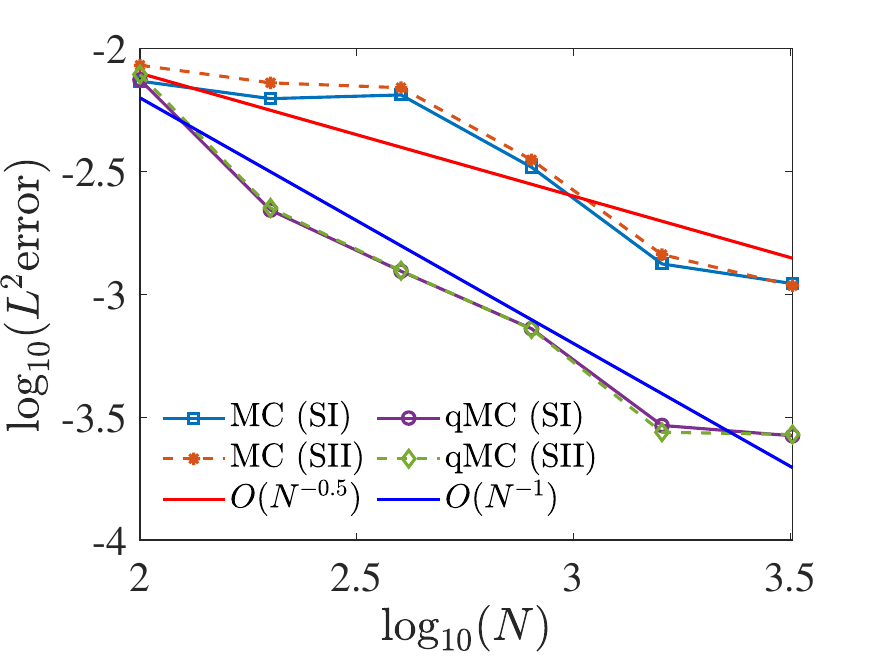}\\
  \caption{Numerical convergence rates of the MC and qMC methods.}
  \label{fig:convergence-rate-MC-qMC}
\end{figure}

\subsubsection{Investigation of wave propagation}
\label{subsec:Anderson-localization}
To observe the wave propagation phenomena, we vary the nonlinear coefficient $\lambda$ and record the evolution of $A(t)$. In addition, we depict $\mathds{E}(|\psi^{\epsilon}_{H,h}|^2)$ at the final time. In these simulations, we generate 500 qMC samples to approximate the random potential.
The parameters of simulations are: $\mathcal{D} = [-2\pi, 2\pi]$, $\sigma = 1.0$, $\beta = 0.0$, and $m = 5$. For the MsFEM, we fix $h = \frac{4\pi}{6000}$ and $H = 10h$. To observe the long-time behavior, we set the terminal time to $T = 20$. We vary $\lambda$ as $0$, $1$, $10$, and $20$, and the corresponding results are shown in \cref{fig:1D-vary-lambda}. One can see that $A(t)$ increases as time evolves for the nonlinear cases, while
for the linear case, it remains within the range of $(0.51, 0.57)$ during the time interval from $t = 10$ to $t = 20$.
\begin{figure}[htbp]
  \centering
  \includegraphics[width=2.5in]{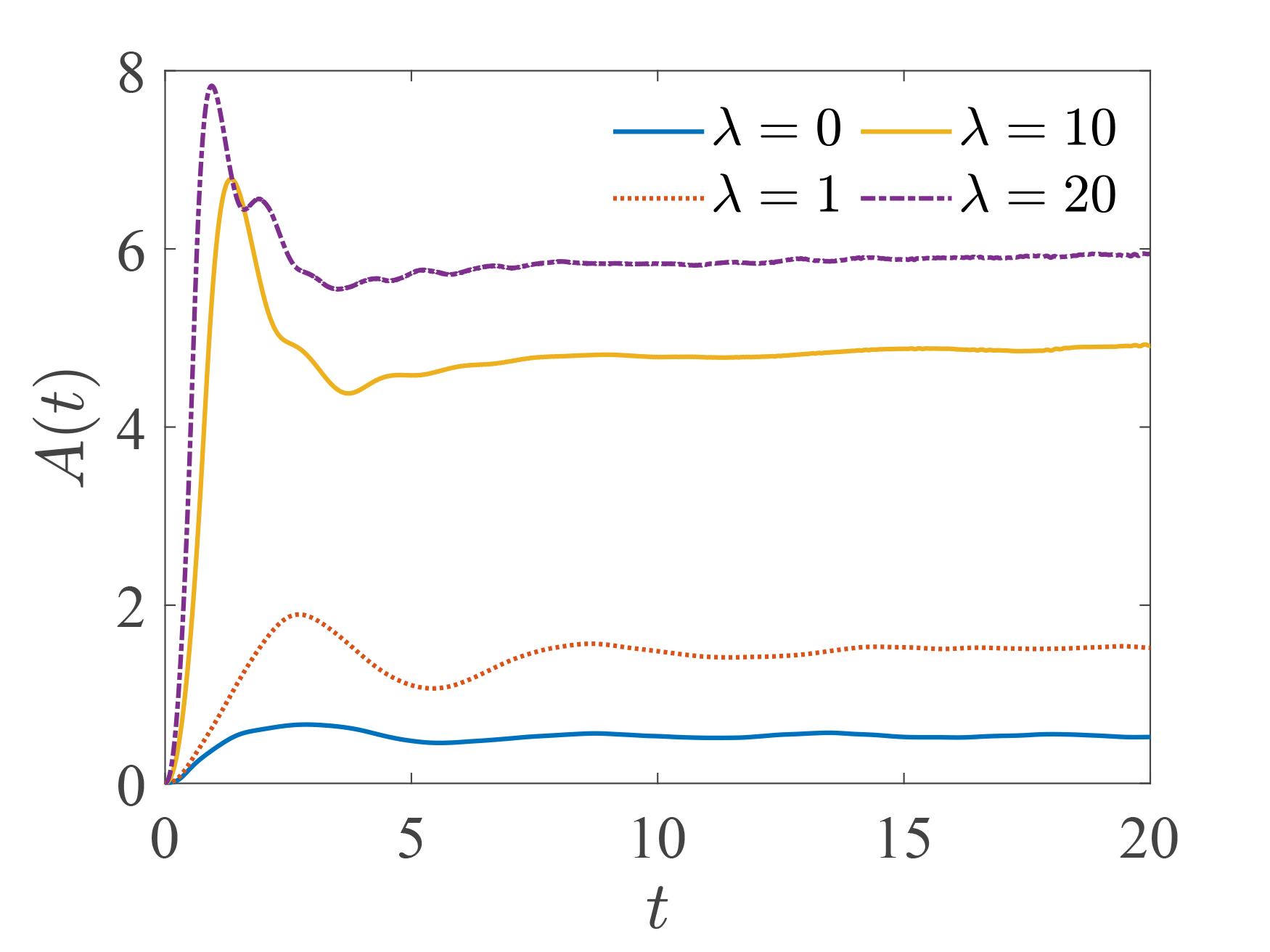}
  \includegraphics[width=2.5in]{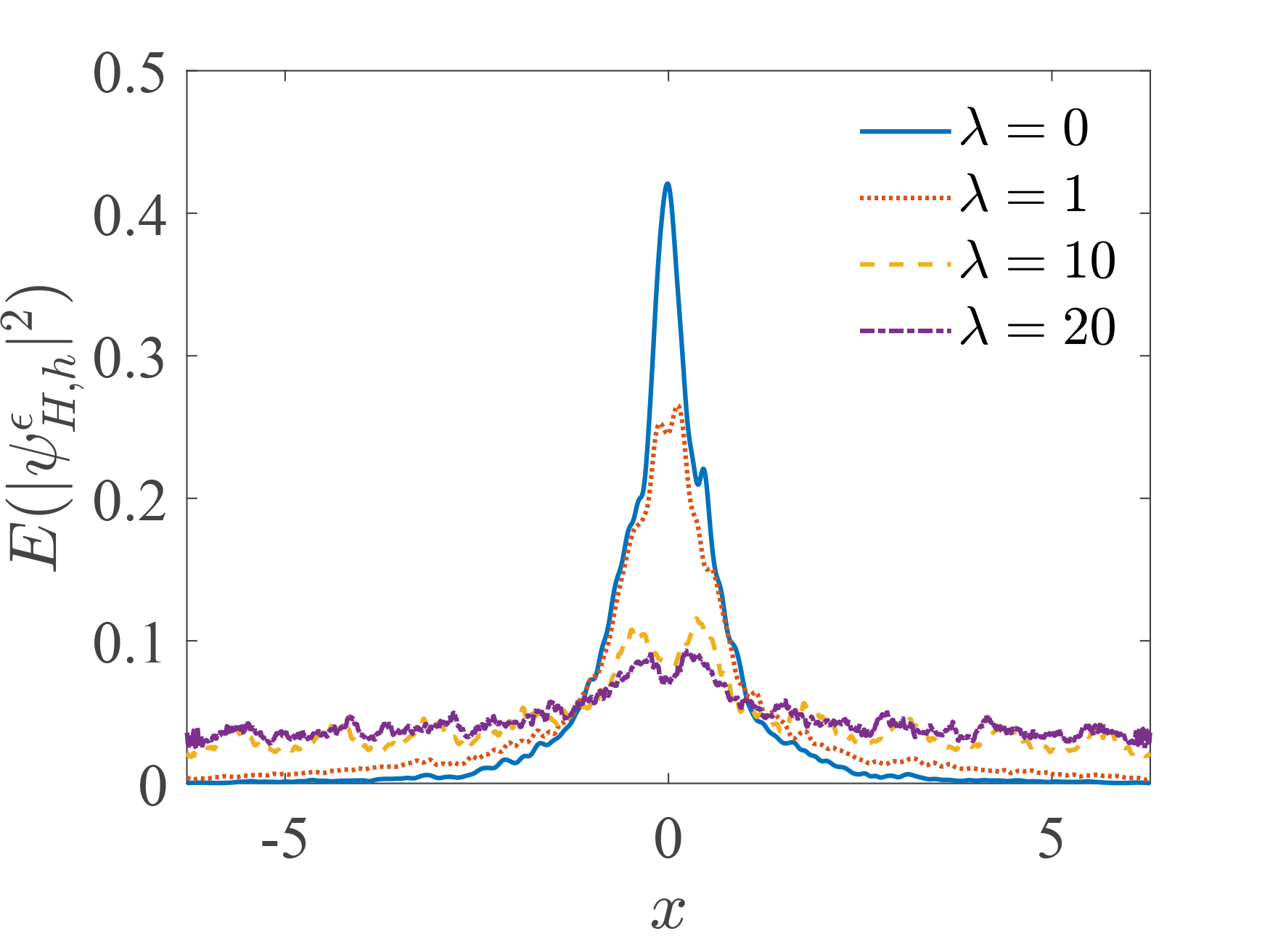}
  \caption{The evolution of $A(t)$ and density of expectation at $T = 20$, as the nonlinear coefficient $\lambda$ varies. Results computed by the {\bf SI} and MsFEM.}
  \label{fig:1D-vary-lambda}
\end{figure}

In the 2D case, we use the following settings in our numerical simulations: $h = \frac {1}{64}$, $\epsilon = \frac 1{4}$, $H = 4h$, $\beta = 0$, $m = 5$, and $\sigma = 5$. Our results, depicted in \Cref{fig:2D-nonlinear-At} and \Cref{fig:2D-nonlinear-Emass}, show that while the localization of mass distribution is observed for the linear case, the nonlinear case exhibits delocalization.

\begin{figure}[htbp]
  \centering
  \includegraphics[width=2.5in]{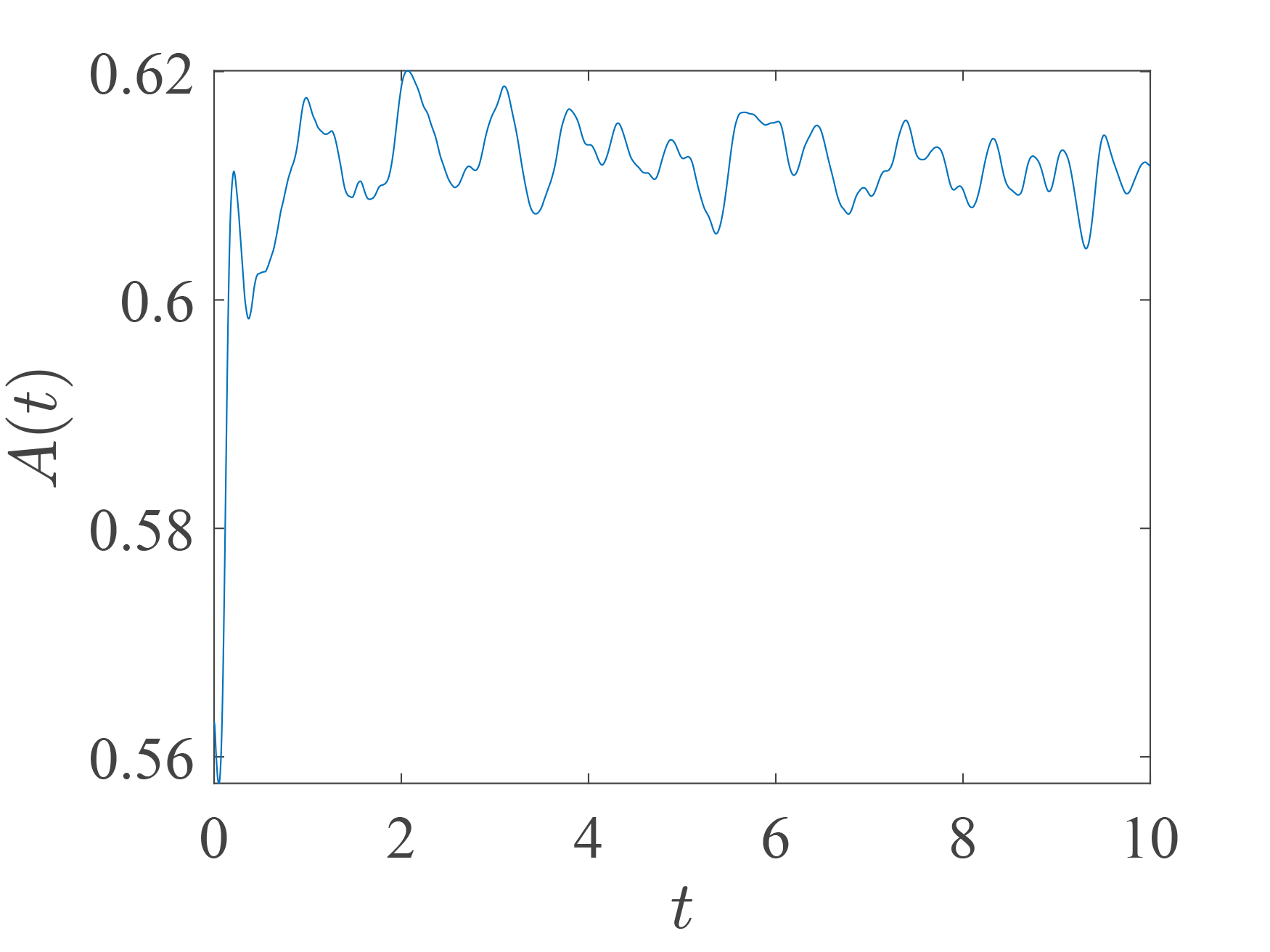}
  \includegraphics[width=2.5in]{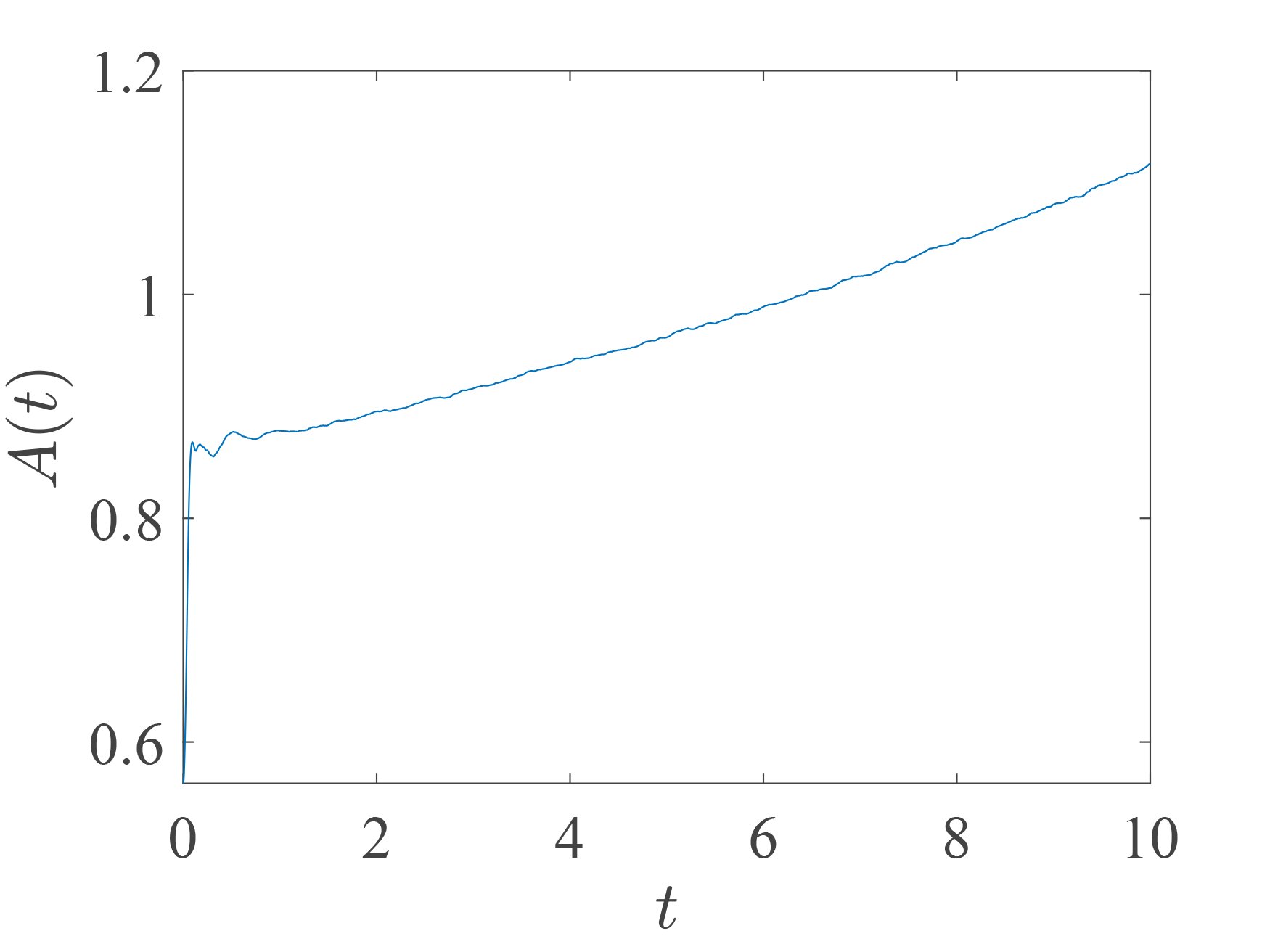}
  \caption{The evolution of $A(t)$ for 2D linear case and nonlinear case with $\lambda = 20$. Results are computed by {\bf SI} and MsFEM.}
  \label{fig:2D-nonlinear-At}
\end{figure}
\begin{figure}[htbp]
  \centering
  \includegraphics[width=2.5in]{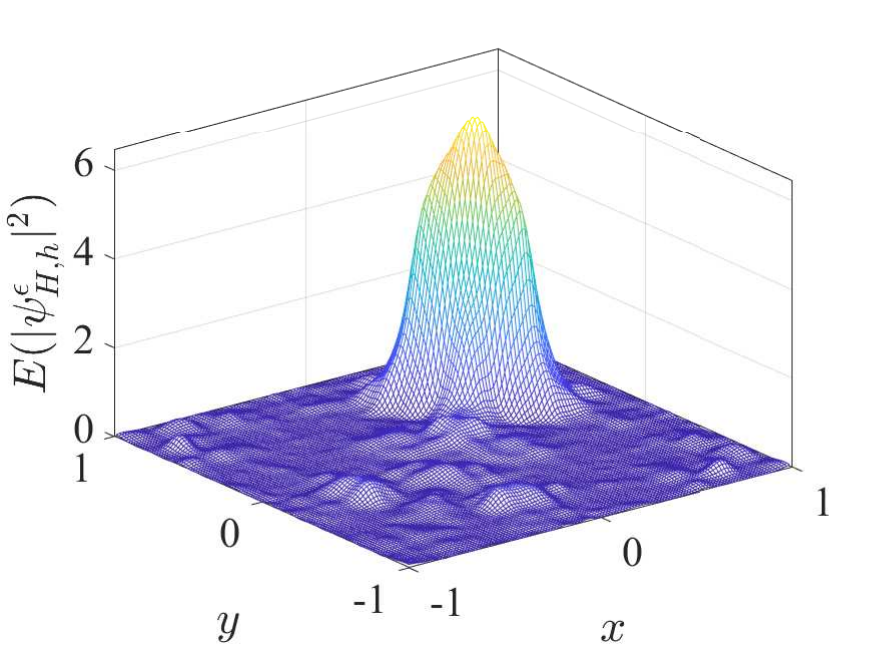}
  \includegraphics[width=2.5in]{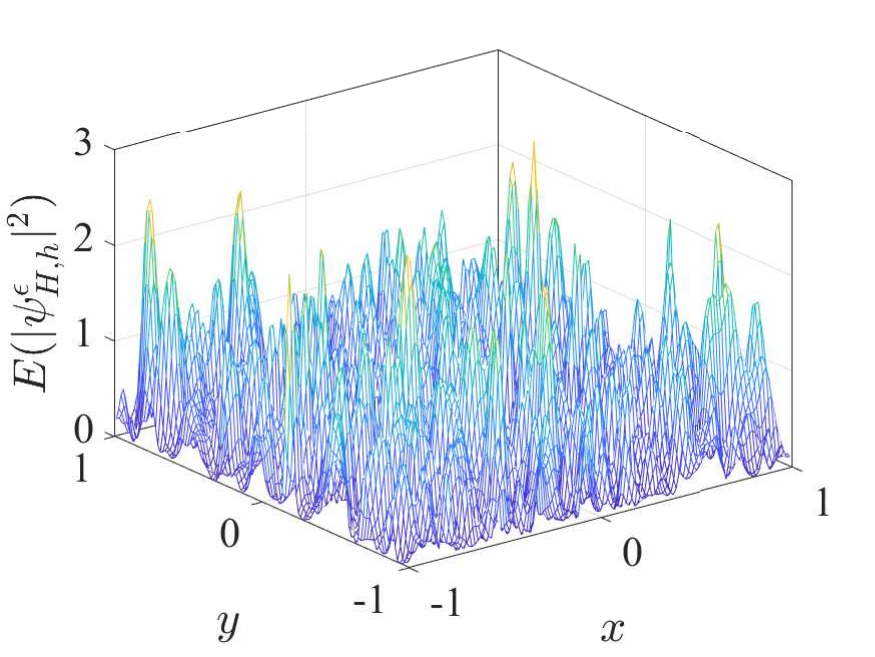}
  \caption{The localization and delocalization of mass distribution of the 2D linear Schr\"odinger equation and NLSE with random potentials, respectively.}
  \label{fig:2D-nonlinear-Emass}
\end{figure}

\section{Conclusion}
In this paper, we have introduced two time-splitting finite element methods~(TS-FEMs) for the cubic nonlinear Schr\"odinger equation~(NLSE), incorporating the multiscale finite element method (MsFEM) to reduce spatial degrees of freedom. We have refined the optimization problems to eliminate the mesh dependence of multiscale basis functions introduced by local orthogonal normalization constraints. For the temporal evolution, we employed two Strang time-splitting techniques in which one maintains the convergence rate of the NLSE with discontinuous potentials. Meanwhile, we utilized the quasi-Monte Carlo sampling method to generate random potentials. Hence the proposed methods have second-order accuracy in both time and space and nearly first-order convergence in the random space. Furthermore, we provided a convergence analysis for the $L^2$ error estimate, which was verified through numerical experiments. Additionally, we presented a multiscale reduced basis method to alleviate the computational burden of constructing multiscale basis functions for random potentials. Using these methods, we investigated the long-term wave propagation of the NLSE with parameterized random potentials in 1D and 2D physical spaces, observing localization in the linear case and delocalization in the nonlinear case. %In summary, our proposed TS-MsFEMs offer a valuable approach for simulating the NLSE with random potentials, delivering good accuracy and high efficiency.

\section*{Acknowledgments}
The research of Z. Zhang is supported by the National Natural Science Foundation of China (project 12171406), the Hong Kong RGC grant (Projects 17307921 and 17304324), the Outstanding Young Researcher Award of HKU (2020-21), Seed Funding for Strategic Interdisciplinary Research Scheme 2021/22 (HKU), and seed funding from the HKU-TCL Joint Research Center for Artificial Intelligence.

% \section*{Acknowledgements}
% The research of Z. Zhang is supported by the National Natural Science Foundation of China (project 12171406), the Hong Kong RGC grant (projects 17300318 and 17307921), the Outstanding Young Researcher Award of HKU (2020¨C21), Seed Funding for Strategic Interdisciplinary Research Scheme 2021/22 (HKU, Hong Kong), seed funding from the HKU-TCL Joint Research Center for Artificial Intelligence, and an R\&D Funding Scheme from the HKU-SCF FinTech Academy.
% We would like to acknowledge the assistance of volunteers in putting
% together this example manuscript and supplement.
\section*{Declaration of interest}
The authors report no conflict of interest.

\appendix
\section{A multiscale reduced basis method}
\label{sec:reduced-MsFEM}
As a supplement, we present an approach to reduce the computational effort required for construction basis functions for random potentials. This approach is motivated by the method proposed in \cite{doi:10.1137/19M127389X}, which consists of offline and online stages. In the offline stage, let $\{v(\bx, \omega_q)\}_{q = 1}^Q$ be the samples of potential with $Q$ representing the number of samples. At the node $\bx_p$, the sample mean of multiscale basis functions
is given by $\zeta_p^0 = \frac 1{Q}\sum_{q=1}^Q \phi_p(\bx, \omega_q)$, and the fluctuation is defined as $\tilde{\phi}_p(\bx, \omega_q) = \phi_p(\bx, \omega_q) - \zeta_p^0$. We employ the POD method on $\{\tilde{\phi}_p(\bx, \omega_q)\}_{q=1}^Q$ to build a set of reduced basis functions $\{\zeta_p^1(\bx), \cdots, \zeta_p^{m_p}(\bx)\}$ with $m_p \ll Q$. In the online stage, the multiscale basis function at $\bx_p$ has the following form
\begin{equation}
  \phi_p(\bx, \omega) = \sum_{l = 0}^{m_p} c_p^l(\omega) \zeta_p^l(\bx),
  \label{equ:basis-functions-random-operator}
\end{equation}
where $\{c_p^l\}_{l=0}^{m_p}$ are unknowns. Due to the wave function being represented by
\begin{equation}
  \psi^{\epsilon}_H(\bx, t, \omega) = \sum_{p=1}^{N_H}\sum_{l = 0}^{m_p} c_p^l(t, \omega) \zeta_p^l(\bx),
\end{equation}
the dofs in the Galerkin formulation is $\sum_{p=1}^{N_H}(m_p+1)$. To reduce the dofs of the Galerkin formulation, we compute $\{c_p^l\}_{l=0}^{m_p}$ in~\eqref{equ:basis-functions-random-operator} by solving the following reduced optimal problems
\begin{align}
  &\min a(\phi_p, \phi_p),\\
  \text{s.t.} &\int_{\mathcal{D}}\phi_p\phi_q^H\mathrm{d}\mathbf{x} = \lambda(H)\delta_{pq}, \quad \forall 1\leq q \leq N_H.
  \label{equ:optimal-problem2}
\end{align}
Since the value of $m_p$ is small~\cite{doi:10.1137/19M127389X}, the computation cost of constructing the multiscale basis functions can be reduced, while the dofs in the Galerkin formulation remain at $N_H$ in the online stage. In addition, we adopt parallel implementations with 12 cores in the following tests.

To demonstrate the improvement offered by the reduced MsFEM basis method, we carry out two numerical tests. We fix $m_p = 3$ for $p = 1,\cdots, N_H$, and generate 1000 samples using the qMC method, with 200 samples allocated for the offline stage and the remaining 800 samples used in the online stage. The {\bf SI} method is employed for time evolution.

%%%%% I stop here at 11pm on May 3 %%%%%%%%%%%%%%%%%%

Here the experiment of the nonlinear case in~\ref{subsec:random-test1} is conducted. We compare the numerical solution computed by the FEM, MsFEM, and the MsFEM with the POD reduction method as in~\cref{fig:reduced-methods-1}.
\begin{figure}[htbp]
  \centering
  \includegraphics[width=2.5in]{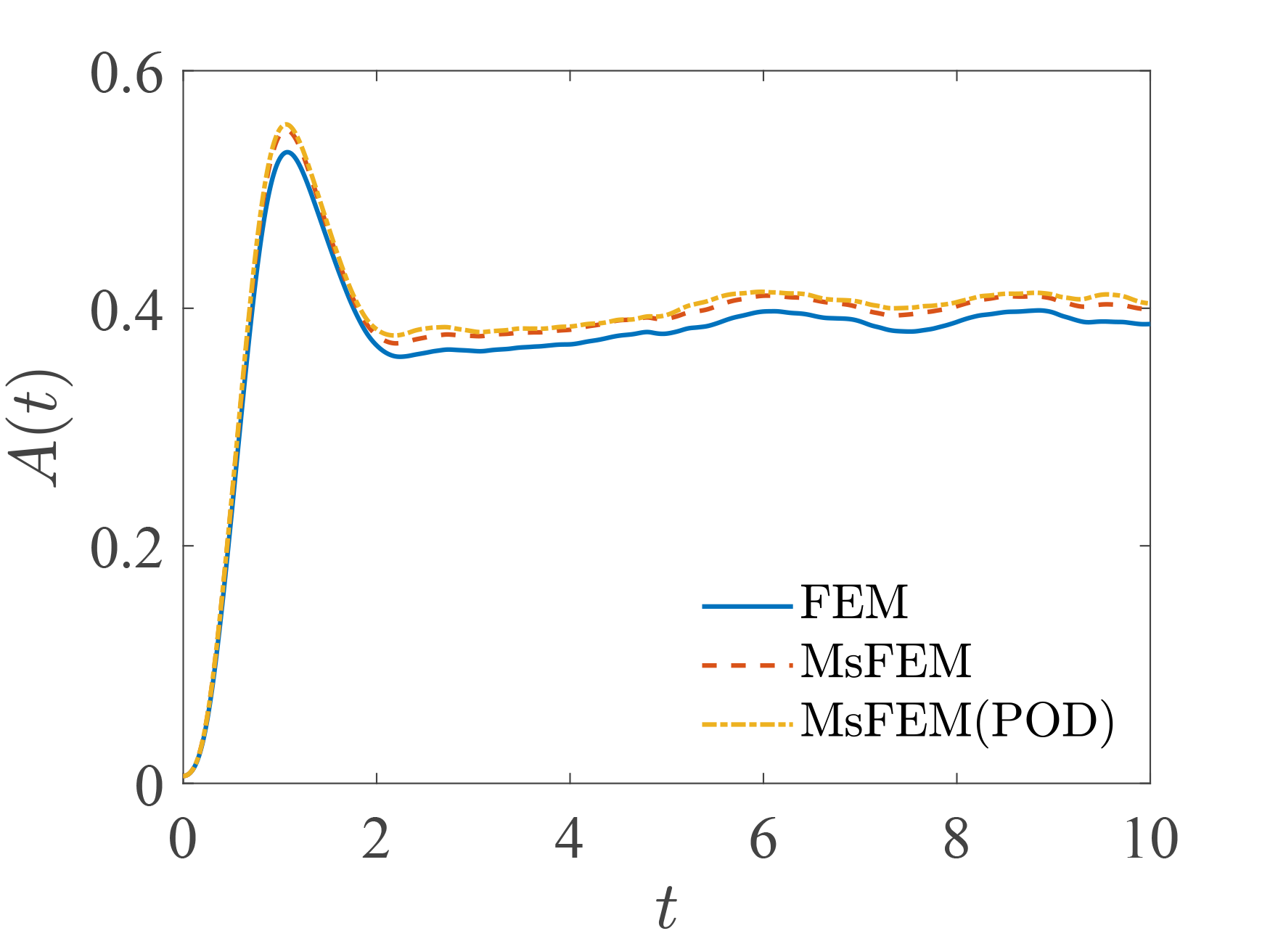}
  \includegraphics[width=2.5in]{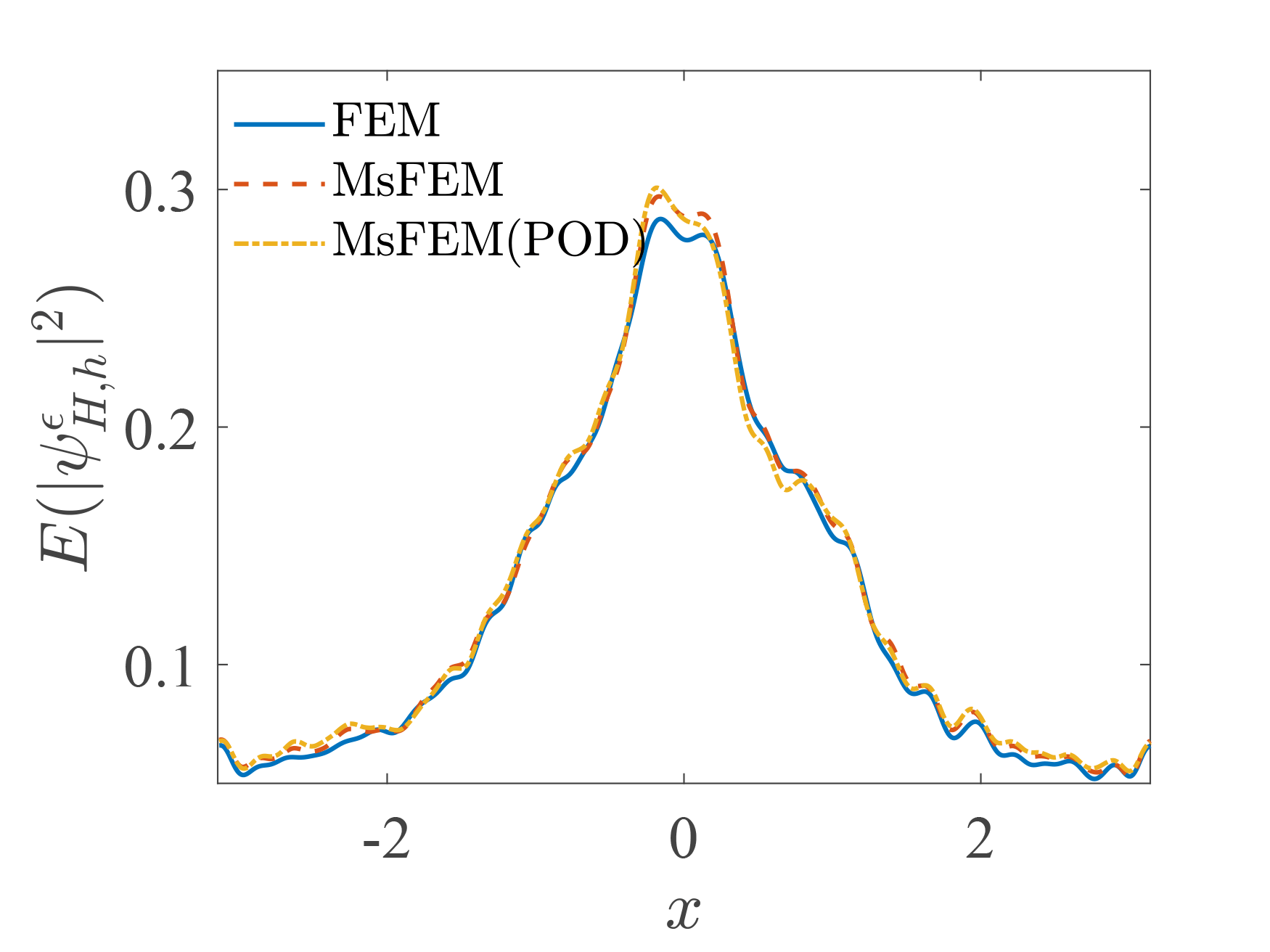}
  \caption{Numerical comparison of FEM, MsFEM and the MsFEM with POD reduction methods.}
  \label{fig:reduced-methods-1}
\end{figure}

Furthermore, we vary the qMC samples and record the corresponding time costs in~\cref{tab:time-costs-qMC-POD}. Note that the time costs of MsFEM with the POD reduction are attributed to both the offline and online stages of the computations. As illustrated in \cref{tab:time-costs-qMC-POD}, a considerable enhancement in simulation efficiency is achieved through the application of MsFEM, with additional improvements attained in the integration of the POD reduction method.
\begin{table}[htbp]
    \centering
    \caption{Comparison of time costs (second) for the FEM, MsFEM, and the MsFEM with POD reduction methods.}
    \begin{tabular}{|c|c|c|c|}
    \hline
    Sample number & FEM & MsFEM & MsFEM (POD)~(offline) \\ %& total saved \\
    \hline
    1000 & 2116 & 152 & 107~(35) \\ %& 45 \\
    2000 & 4205 & 308 & 243~(35) \\ %& 65 \\
    4000 & 8376 & 620 & 501~(34) \\ %& 119 \\
    8000 & 16633 & 1239 & 1020~(40) \\ %& 219 \\
    16000 & 33469 & 2466 & 2137~(43) \\ %& 329 \\
    \hline
    \end{tabular}
    \label{tab:time-costs-qMC-POD}
\end{table}

We repeat the experiment of NLSE with $\lambda = 20$ as in~\ref{subsec:Anderson-localization}. The corresponding numerical results are shown in~\cref{fig:reduced-methods-2}. The MsFEM combined with the POD reduction method takes approximately 14978 seconds (4.16 hours), with 1064 seconds spent on the offline stage. In contrast, the MsFEM without incorporating the POD method takes 20,061 seconds (5.57 hours).
\begin{figure}[htbp]
  \centering
  \includegraphics[width=2.5in]{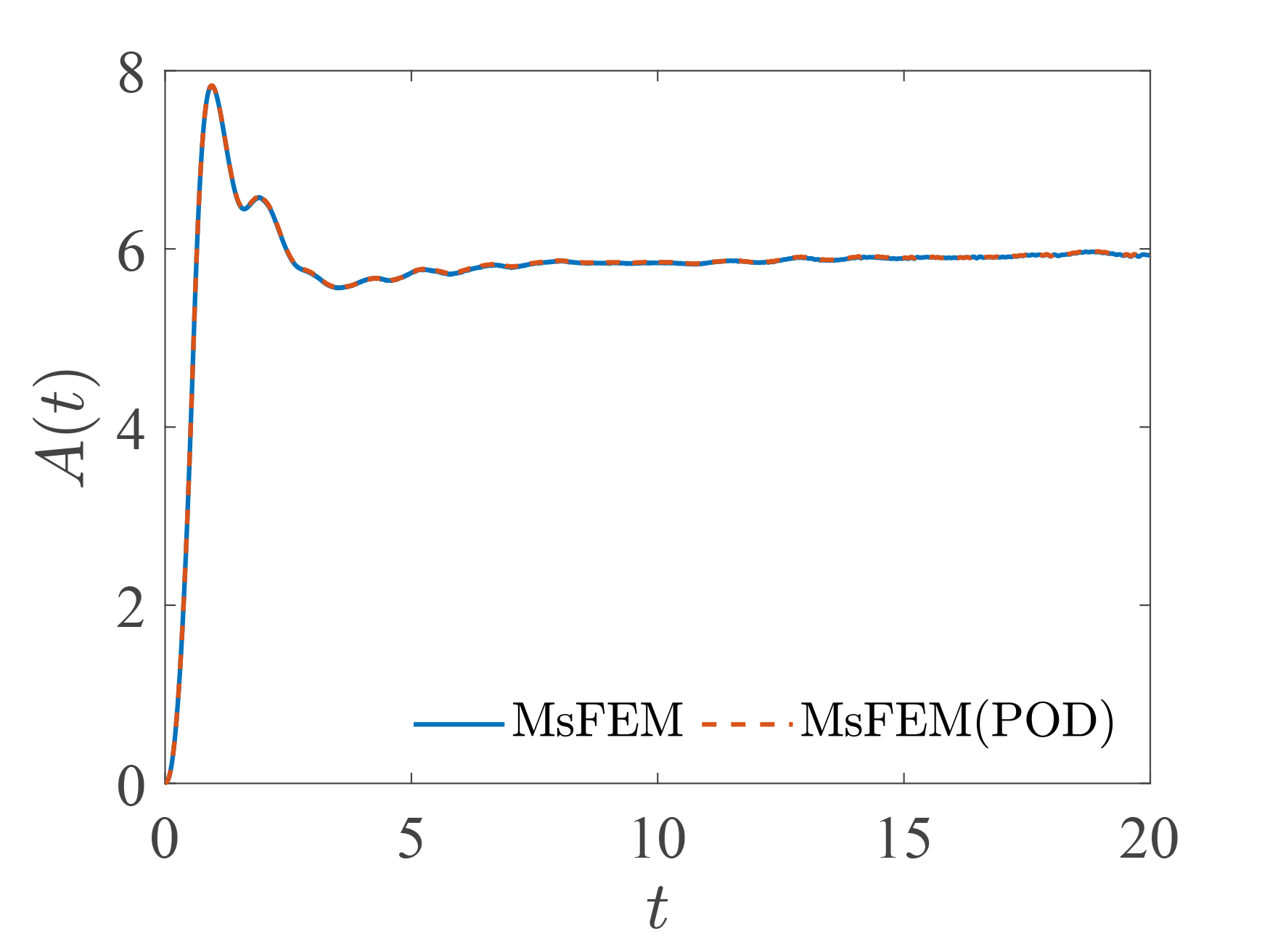}
  \includegraphics[width=2.5in]{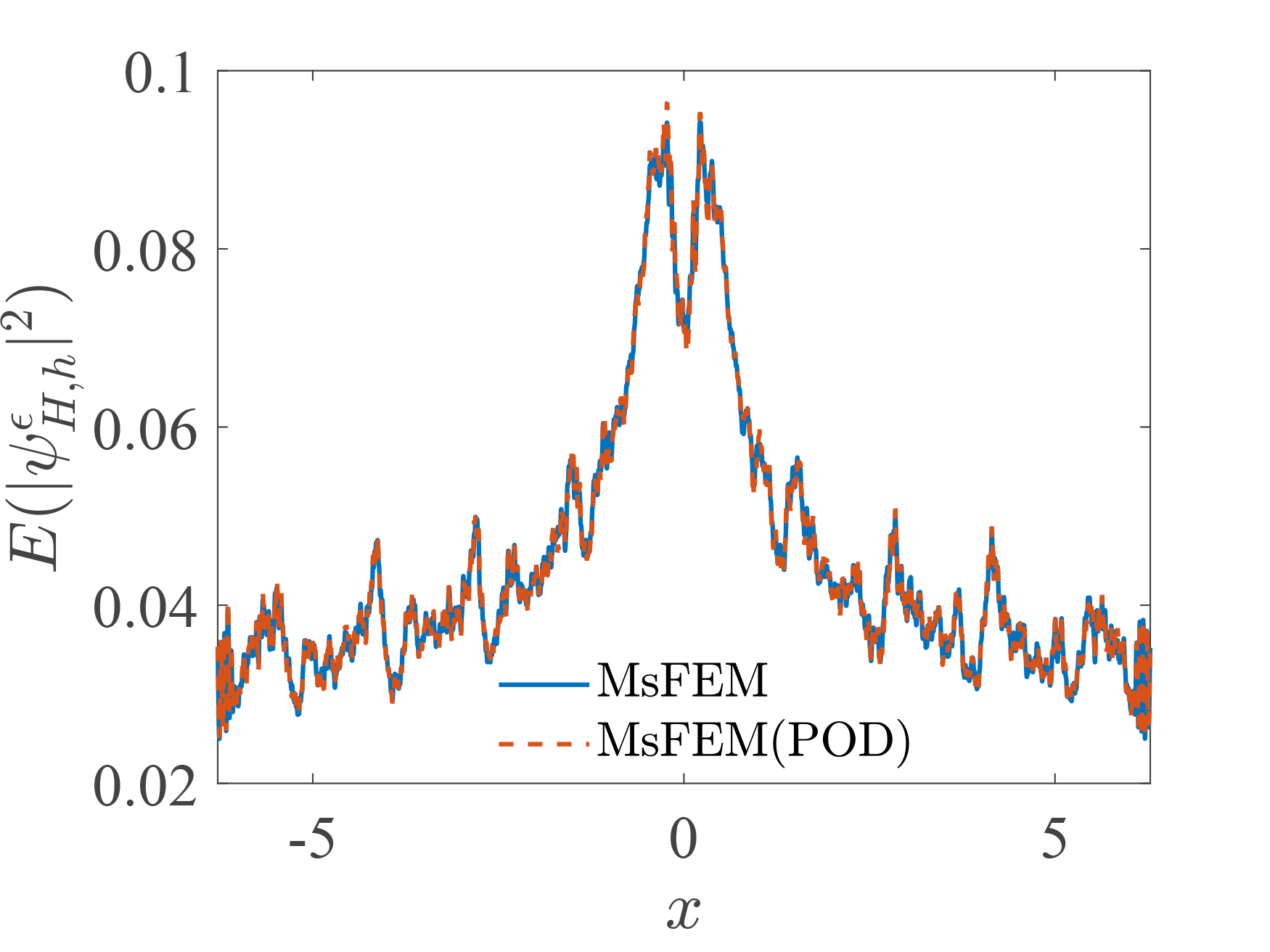}
  \caption{Numerical comparison of MsFEM method and the MsFEM with the POD reduction method for the 1D NLSE with $\lambda = 20$.}
  \label{fig:reduced-methods-2}
\end{figure}

%\section{Some useful inequalities}
%\begin{lemma}[Gagliardo-Nirenberg interpolation inequality~\cite{ASNSP_1959_3_13_2_115_0,brezis2011functional}]
%Let $f$ belong to $L_q(\mathcal{D})$ and its derivatives of order $m$, $D^mf$, belong to $L_r(\mathcal{D})$, $1 \leq q, r \leq \infty$. For the derivatives $D^jf$, $0 \leq j \le m$, the inequalities hold
%\begin{equation}
%  \|D^jf\|_{L^p} \leq C \|D^mf\|_{L^r}^a \|f\|_{L^q}^{1-a},
%\end{equation}
%where
%\begin{equation*}
%  \frac 1p = \frac jd + a\left(\frac 1r - \frac md\right) + (1-a)\frac 1q,
%\end{equation*}
%for all $a$ in the interval
%\begin{equation*}
%  \frac{j}{m} \leq a \leq 1.
%\end{equation*}
%Here the constant $C$ depends on $m, n, j, q, r, a$.
%
%In particular, consider $f \in L^p(\mathcal{D})\cap W^{2,r}(\mathcal{D})$ with $1\leq p \leq \infty$ and $1 \leq r \leq \infty$. Then $f \in W^{1,q}(\mathcal{D})$, where $q$ is determined by $\frac 1q = \frac 12(\frac 1p + \frac 1r)$ and
%\begin{equation}
%  \|Df\|_{L^q} \leq C\|f\|^{\frac 12}_{W^{2,r}}\|f\|^{\frac 12}_{L^p}.
%\end{equation}
%We have the particular case: $p = \infty$ and thus $q = 2r$,
%  \begin{equation}
%    \|Df\|_{L^q} \leq C\|f\|^{\frac 12}_{W^{2,r}}\|f\|_{\infty}^{\frac 12}.
%  \end{equation}
%\end{lemma}

\section{The proof of \cref{lem:regularity-NLSE}}
\label{appsec:proof-regularity-NLSE}
\begin{proof}
  We first study the regularity of $\psi^{\epsilon}$ in space. Since the energy is a constant
  \begin{equation*}
    E(t) = \frac{\epsilon^2}{2}\|\nabla\psi^{\epsilon}\|^2 + (v, |\psi^{\epsilon}|^2) + \frac{\lambda}{2} \|\psi^{\epsilon}\|_{L^4}^4 = E_0 < \infty
  \end{equation*}
  with $\lambda \geq 0$, we directly get
  \begin{align*}
    \frac{\epsilon^2}{2}\|\nabla\psi^{\epsilon}\|^2
    &=
    E_0 - (v, |\psi^{\epsilon}|^2) - \frac{\lambda}{2} \|\psi^{\epsilon}\|_{L^4}^4
    \leq
    E_0 + \| v \|_{\infty},
  \end{align*}
  which means
  \begin{equation*}
    \|\nabla\psi^{\epsilon}\| \leq \frac {C}{\epsilon}.
  \end{equation*}
  Meanwhile, we also have
  \begin{equation}
    \|\psi^{\epsilon}\|_{L^4}^4 \leq \frac{E_0 + \|v\|_{\infty} }{\lambda}.
  \end{equation}
  Owing to the Hamiltonian $\mathcal{H}$ is not explicitly dependent on time, and $[\mathcal{H}^2, \mathcal{H}] = 0$, the following average value of mechanics quantity is independent of time, i.e.,
  \begin{equation}
    (\mathcal{H}^2\psi^{\epsilon}, \psi^{\epsilon}) = E_1
  \end{equation}
  with $\mathrm{d}_tE_1 = 0$. Explicitly, we have
  \begin{align*}
    (\mathcal{H}^2\psi^{\epsilon}, \psi^{\epsilon}) = &\frac{\epsilon^4}{4}(\Delta^2\psi^{\epsilon}, \psi^{\epsilon}) + (v^2\psi^{\epsilon}, \psi^{\epsilon}) + \lambda^2(|\psi^{\epsilon}|^4\psi^{\epsilon}, \psi^{\epsilon})
    \\&
    - \epsilon^2(\Delta v\psi^{\epsilon}, \psi^{\epsilon}) +
    2\lambda(v|\psi^{\epsilon}|^2\psi^{\epsilon}, \psi^{\epsilon}) - \lambda\epsilon^2(\Delta|\psi^{\epsilon}|^2\psi^{\epsilon}, \psi^{\epsilon}).
  \end{align*}
  We then get
  \begin{align*}
    &\frac{\epsilon^4}{4}\|\Delta\psi^{\epsilon}\|^2 + \|v\psi^{\epsilon}\|^2 + \lambda^2\| \psi^{\epsilon}\|_{L^6}^6
    \\ \leq&
    E_1 + \epsilon^2(\Delta v\psi^{\epsilon}, \psi^{\epsilon}) -
    2\lambda(v|\psi^{\epsilon}|^2\psi^{\epsilon}, \psi^{\epsilon}) + \lambda\epsilon^2(\Delta|\psi^{\epsilon}|^2\psi^{\epsilon}, \psi^{\epsilon})
    \\ \leq&
    E_1 - \epsilon^2(\nabla v\psi^{\epsilon}, \nabla\psi^{\epsilon}) +
    2\lambda \| v \|_{\infty} \|\psi^{\epsilon}\|_{L^4}^4 + 3\lambda\epsilon^2 \| \psi^{\epsilon} \|_{\infty}^2 \| \nabla\psi^{\epsilon} \|^2
    \\ \leq&
    E_1 + C\| v \|_{\infty} + \epsilon\|\nabla v\|_{\infty} +
    2\lambda \| v \|_{\infty} \|\psi^{\epsilon}\|_{L^4}^4 + 3\lambda C \| \psi^{\epsilon} \|_{\infty}^2.
  \end{align*}
  Hence, there exists a constant $C$ that depends on $\| v \|_{\infty}$, $\|\nabla v \|_{\infty}$, $E_0$, $E_1$, and $\| \psi^{\epsilon} \|_{\infty}$ such that
  \begin{equation}
    \| \nabla^2\psi^{\epsilon} \| \leq \frac{C}{\epsilon^2}, \quad
    \| \psi^{\epsilon} \|_{L^6}^6 \leq \frac{C}{\lambda^2}.
  \end{equation}
  Furthermore, if $\psi^{\epsilon} \in H^4$, we also have $[\mathcal{H}^s, \mathcal{H}] = 0$ for $s \leq 4$. Repeat the above procedures and we can get
  \begin{equation}
    \| \nabla^s\psi^{\epsilon} \| \leq \frac{C}{\epsilon^s}.
  \end{equation}

  Next, we study the bound of $\|\partial_t\psi^{\epsilon}\|_{H^s}$ with $0 \leq s \leq 2$. Taking the time derivative for \eqref{equ:NLS_equ} yields
  \begin{equation}
    i\epsilon\partial_{tt}\psi^{\epsilon} = -\frac{\epsilon^2}{2}\Delta\partial_t\psi^{\epsilon} + v\partial_t\psi^{\epsilon} + 2\lambda|\psi^{\epsilon}|^2\partial_t\psi^{\epsilon} + \lambda(\psi^{\epsilon})^2\partial_t\bar\psi^{\epsilon}.
    \label{equ:partial-t-for-NLS}
  \end{equation}
  Take inner product of this equation with $\partial_t{\psi^{\epsilon}}$ and we get
  %\begin{align*}
%    i\epsilon(\partial_{tt}\psi^{\epsilon},\partial_t{\psi^{\epsilon}}) = &\frac{\epsilon^2}{2}(\nabla\partial_t\psi^{\epsilon},\nabla\partial_t{\psi^{\epsilon}}) + (v\partial_t\psi^{\epsilon}, \partial_t\psi^{\epsilon}) + \lambda(\partial_t\psi^{\epsilon}\bar{\psi^{\epsilon}}, \partial_t\psi^{\epsilon}\bar\psi^{\epsilon}) + 2\lambda(|\psi^{\epsilon}|^2, |\partial_t\psi^{\epsilon}|^2).
%  \end{align*}
%  Then we get
  \begin{equation}
    i\epsilon\mathrm{d}_t(\partial_{t}\psi^{\epsilon},\partial_t{\psi^{\epsilon}}) = \lambda\int_{\mathcal{D}}(\partial_t\psi^{\epsilon}\bar{\psi^{\epsilon}})^2 - (\partial_t\bar\psi^{\epsilon}{\psi^{\epsilon}})^2\mathrm{d}\bx = 4i\lambda\int_{\mathcal{D}}\Re(\partial_t\psi^{\epsilon}\bar{\psi^{\epsilon}})
    \Im(\partial_t\psi^{\epsilon}\bar{\psi^{\epsilon}})\mathrm{d}\bx.
  \end{equation}
  Thus we have
  \begin{align*}
    \epsilon\mathrm{d}_t\|\partial_{t}\psi^{\epsilon}\|^2 &\leq 2\lambda \| \partial_t\psi^{\epsilon} {\psi^{\epsilon}} \|^2 \leq 2\lambda\| \psi^{\epsilon} \|_{\infty}^2 \| \partial_t\psi^{\epsilon} \|^2,
    \label{equ:est-reguarity-partialt-10}
  \end{align*}
  which indicates
  \begin{equation}
    \|\partial_{t}\psi^{\epsilon}\| \leq \|\partial_{t}\psi_{\mathrm{in}}\|\exp\left( \frac{2\lambda T\| \psi^{\epsilon} \|_{\infty}^2}{\epsilon} \right).
  \end{equation}
  For the initial condition, we have
  \begin{align*}
    \|\partial_t\psi_{\mathrm{in}}\| \leq \frac{\epsilon}{2}\|\nabla\psi_{\mathrm{in}}\| + \frac{1}{\epsilon}(v\psi_{\mathrm{in}},\psi_{\mathrm{in}}) + \frac{\lambda}{\epsilon}\|\psi_{\mathrm{in}}\|_{L^4}^2
    \leq \frac{C}{\epsilon}.
  \end{align*}
  We therefore get
  \begin{equation}
    \|\partial_{t}\psi^{\epsilon}\| \leq \frac{C}{\epsilon}\exp\left( \frac{2\lambda\| \psi^{\epsilon} \|_{\infty}^2 T}{\epsilon} \right).
  \end{equation}

  Take inner product of the equation \eqref{equ:partial-t-for-NLS} with $\partial_t\Delta\psi^{\epsilon}$, and we have
  \begin{multline*}
    \epsilon\mathrm{d}_t\|\nabla\partial_t\psi^{\epsilon}\|^2 = \Im\{2(\nabla v\partial_t\psi^{\epsilon}, \nabla\partial_t\psi^{\epsilon}) +
    4\lambda(\psi^{\epsilon}\partial_t\psi^{\epsilon}\nabla\bar\psi^{\epsilon}, \nabla\partial_t\psi^{\epsilon}) \\
    + 4\lambda(\bar\psi^{\epsilon}\partial_t\psi^{\epsilon}\nabla\psi^{\epsilon}, \nabla\partial_t\psi^{\epsilon}) +
    4\lambda(\psi^{\epsilon}\partial_t\psi^{\epsilon}\nabla\psi^{\epsilon}, \nabla\partial_t\psi^{\epsilon}) +
    2\lambda((\psi^{\epsilon})^2, (\nabla\partial_t\psi^{\epsilon})^2)\}.
  \end{multline*}
  By the inequalities
  \begin{align*}
    &\|\psi^{\epsilon} \partial_t\psi^{\epsilon} \nabla\psi^{\epsilon} \nabla\partial_t\psi^{\epsilon}\|_{L^1}
    \leq
    \| \psi^{\epsilon} \|_{L^6} \| \partial_t\psi^{\epsilon} \|_{L^6} \|  \nabla\psi^{\epsilon} \|_{L^6} \|\nabla\partial_t\psi^{\epsilon}\|
    \\ \leq&
    C\| \psi^{\epsilon} \|_{L^6} \left(\frac d3\|\partial_t\nabla\psi^{\epsilon}\| + \left(1-\frac d3\right)\|\partial_t\psi^{\epsilon}\| \right) \|\nabla^2\psi^{\epsilon}\|^{\frac 12 + \frac d6} \|\nabla\partial_t\psi^{\epsilon}\|
    \\ \leq&
    C\| \psi^{\epsilon} \|_{L^6} \left(\|\partial_t\nabla\psi^{\epsilon}\| + \|\partial_t\psi^{\epsilon}\| \right) \|\nabla^2\psi^{\epsilon}\| \|\nabla\partial_t\psi^{\epsilon}\|
  \end{align*}
  and
  \begin{align*}
    \| (\psi^{\epsilon})^2(\nabla\partial_t\psi^{\epsilon})^2 \|_{L^1} \leq \| \psi^{\epsilon} \|_{L^{\infty}}^2 \| \nabla\partial_t\psi^{\epsilon} \|^2,
  \end{align*}
  we get
  \begin{align*}
    \epsilon\mathrm{d}_t\|\partial_t\nabla\psi^{\epsilon}\|
    \leq&
    2\|\nabla v\|_{\infty} \|\partial_t\psi^{\epsilon}\| + C\lambda\|\nabla^2\psi^{\epsilon}\| \left(\|\partial_t\nabla\psi^{\epsilon}\| + \|\partial_t\psi^{\epsilon}\| \right) +
    2\lambda \| \psi^{\epsilon} \|_{L^{\infty}}^2 \| \nabla\partial_t\psi^{\epsilon} \|.
    \label{equ:est-reguarity-partialt-2}
  \end{align*}
  Then we arrive at
  \begin{align*}
    \|\partial_t\nabla\psi^{\epsilon}\| &\leq \left(\frac{2 \|\nabla v\|_{\infty}}{\epsilon} + \frac{C\lambda \|\nabla
    ^2\psi^{\epsilon}\|}{\epsilon} \right)\|\partial_t\psi^{\epsilon}\| \exp\left( \frac{C\lambda T \|\nabla
    ^2\psi^{\epsilon}\|}{\epsilon} + \frac{2\lambda T \|\psi^{\epsilon}\|_{\infty}^2}{\epsilon} \right)
    \\&\leq
    \frac{C\lambda}{\epsilon^4} \exp\left( \frac{C\lambda T}{\epsilon^3} \right).
  \end{align*}
  Let $d = 3$, and the above result can be replaced with
  \begin{equation}
    \|\partial_t\nabla\psi^{\epsilon}\| \leq \frac{2 \|\nabla v\|_{\infty}}{\epsilon} \|\partial_t\psi^{\epsilon}\| \exp\left( \frac{C\lambda T \|\nabla
    ^2\psi^{\epsilon}\|}{\epsilon} + \frac{2\lambda T \|\psi^{\epsilon}\|_{\infty}^2}{\epsilon} \right).
    % \leq \frac{C \|\nabla v\|_{\infty}}{\epsilon^2} \exp\left( \frac{C\lambda T }{\epsilon^3} \right).
  \end{equation}
  By the similar procedures, we have
  \begin{align*}
    &\epsilon\mathrm{d}_t\|\partial_t\nabla^2\psi^{\epsilon}\|^2
    \leq
    \|\nabla^2v\|_{\infty} \|\partial_t\psi^{\epsilon}\| \|\partial_t\nabla^2\psi^{\epsilon}\| +
    2\|\nabla v\|_{\infty} \|\partial_t\nabla\psi^{\epsilon}\| \|\partial_t\nabla^2\psi^{\epsilon}\| +
    \\& \quad \quad \quad
    C\lambda\|\nabla^3\psi^{\epsilon}\|^{\frac 23+\frac{d}{9}} \|\partial_t\nabla\psi^{\epsilon}\|^{\frac d3} \|\partial_t\psi^{\epsilon}\|^{1-\frac d3} \|\partial_t\nabla^2\psi^{\epsilon}\| +
    \\& \quad \quad \quad
    C\lambda\|\nabla^3\psi^{\epsilon}\|^{\frac {6}{9-d}} \|\psi^{\epsilon}\|_{L^6}^{2-\frac {6}{9-d}} \|\partial_t\nabla\psi^{\epsilon}\|^{\frac d3} \|\partial_t\psi^{\epsilon}\|^{1-\frac d3} \|\partial_t\nabla^2\psi^{\epsilon}\| +
    \\& \quad \quad \quad
    C\lambda\|\nabla^2\psi^{\epsilon}\|^{\frac 12 + \frac d6} \|\partial_t\nabla^2\psi^{\epsilon}\|^{\frac 12 + \frac d6} \|\partial_t\psi^{\epsilon}\|^{\frac 12 - \frac d6} \|\partial_t\nabla^2\psi^{\epsilon}\| +
    C\lambda\|\psi^{\epsilon}\|_{\infty}^2\|\partial_t\nabla^2\psi^{\epsilon}\|^2
    \label{equ:est-reguarity-partialt-3}
  \end{align*}
  in which we use the inequalities
  \begin{align*}
    \|\nabla^2\psi^{\epsilon} \psi^{\epsilon} \partial_t\psi^{\epsilon} \partial_t\nabla^2\psi^{\epsilon}\|_{L^1}
    &\leq
    \|\psi^{\epsilon}\|_{L^6} \|\nabla^2\psi^{\epsilon}\|_{L^6} \|\partial_t\psi^{\epsilon}\|_{L^6} \|\partial_t\nabla^2\psi^{\epsilon}\|
    \\&\leq
    C\|\nabla^3\psi^{\epsilon}\|^{\frac 23+\frac{d}{9}} \|\partial_t\nabla\psi^{\epsilon}\|^{\frac d3} \|\partial_t\psi^{\epsilon}\|^{1-\frac d3} \|\partial_t\nabla^2\psi^{\epsilon}\|,
    \\
    \|\nabla\psi^{\epsilon} \nabla\psi^{\epsilon} \partial_t\psi^{\epsilon} \partial_t\nabla^2\psi^{\epsilon}\|_{L^1}
    &\leq
    \|\nabla\psi^{\epsilon}\|_{L^6}^2 \|\partial_t\psi^{\epsilon}\|_{L^6} \|\partial_t\nabla^2\psi^{\epsilon}\|
    \\&\leq
    C\|\nabla^3\psi^{\epsilon}\|^{\frac {6}{9-d}} \|\psi^{\epsilon}\|_{L^6}^{2-\frac {6}{9-d}} \|\partial_t\nabla\psi^{\epsilon}\|^{\frac d3} \|\partial_t\psi^{\epsilon}\|^{1-\frac d3} \|\partial_t\nabla^2\psi^{\epsilon}\|,
    \\
    \|\psi^{\epsilon} \nabla\psi^{\epsilon} \partial_t\nabla\psi^{\epsilon} \partial_t\nabla^2\psi^{\epsilon}\|_{L^1}
    &\leq \|\psi^{\epsilon}\|_{L^6} \|\nabla\psi^{\epsilon}\|_{L^6} \|\partial_t\nabla\psi^{\epsilon}\|_{L^6} \|\partial_t\nabla^2\psi^{\epsilon}\|
    \\&\leq
    C\|\nabla^2\psi^{\epsilon}\|^{\frac 12 + \frac d6} \|\partial_t\nabla^2\psi^{\epsilon}\|^{\frac 12 + \frac d6} \|\partial_t\psi^{\epsilon}\|^{\frac 12 - \frac d6} \|\partial_t\nabla^2\psi^{\epsilon}\|,
  \end{align*}
  and
  \begin{align*}
    \|(\psi^{\epsilon})^2(\partial_t\nabla^2\psi^{\epsilon})\|_{L^1} \leq \|\psi^{\epsilon}\|_{\infty}^2\|\partial_t\nabla^2\psi^{\epsilon}\|^2.
  \end{align*}
  Then we get
  \begin{equation}
    \|\partial_t\nabla^2\psi^{\epsilon}\| \leq \frac{C\lambda \|\nabla^3\psi^{\epsilon}\|^{\ell}\|\partial_t\nabla\psi^{\epsilon}\|^{\frac d3} \|\partial_t\psi^{\epsilon}\|^{1-\frac d3}}{\epsilon}\exp\left(\frac{C\lambda T\|\nabla^2\psi^{\epsilon}\| + C\lambda T\|\psi^{\epsilon}\|_{\infty}^2}{\epsilon}\right),
  \end{equation}
  where $\ell = \max\{\frac 23 +\frac d9, \frac{6}{9-d}\}$.
  Let $d = 3$ and we get the compact form
  \begin{equation}
    \|\partial_t\nabla^2\psi^{\epsilon}\| \leq \frac{C\lambda\|\nabla^3\psi^{\epsilon}\|}{\epsilon}\|\partial_t\nabla\psi^{\epsilon}\| \exp\left(\frac{C\lambda T\|\nabla^2\psi^{\epsilon}\| + C\lambda T\|\psi^{\epsilon}\|_{\infty}^2}{\epsilon}\right).
  \end{equation}
  Due to $\epsilon \ll 1$, the order of $\|\partial_t\psi^{\epsilon}\|_{H^s}$ with respect to $\epsilon$ directly depends on the estimate $\|\partial_t\nabla^s\psi^{\epsilon}\|$. Thus, there exists a constant $C_{\lambda,\epsilon}$ that depends on $\lambda$ and $\epsilon$ such that $\|\partial_t\psi^{\epsilon}\|_{H^s} \leq C_{\lambda, \epsilon}$. This completes the proof.
\end{proof}

\section{The proof of \cref{lem:regualrity-psi-wrt-random-variables}}
\label{sec:proof-lemma-regularity-ramdom}
\begin{proof}
  Let $|\boldsymbol{\nu}| = 1$, and we take the derivative with respect to $\xi_j(\omega)$ of \eqref{equ:NLS_equ_parameterized}. Denote $\partial_j\psi_m = \partial_{\xi_j}\psi^{\epsilon}_m$ and $\partial_jv_m = \partial_{\xi_j}v^{\epsilon}_m$, and we get
  \begin{equation*}
    i\epsilon\partial_t(\partial_j\psi_m) = -\frac{\epsilon^2}{2}\Delta(\partial_j\psi_m) + (\partial_jv_m)\psi^{\epsilon}_m + v_m^{\epsilon}(\partial_j\psi_m) + \lambda(2|\psi^{\epsilon}_m|^2\partial_j\psi_m + (\psi^{\epsilon}_m)^2\partial_j\bar\psi_m).
  \end{equation*}
  We have
  \begin{align*}
    &\epsilon\mathrm{d}_t\|\partial_j\psi_m\| \leq 2 \|\partial_j v_m\|_{\infty} + 2\lambda\|\psi^{\epsilon}_m\|_{\infty}^2 \|\partial_j\psi_m\|, \\
    &\epsilon\mathrm{d}_t\|\nabla\partial_j\psi_m\| \leq 2 \|\nabla\partial_j v_m\|_{\infty} + 2\|\partial_j v_m\|_{\infty} \|\nabla\psi^{\epsilon}_m\| + 2\|\nabla v_m\|_{\infty} \|\partial_j\psi_m\| +
    \\&\quad
    16\lambda \|\psi^{\epsilon}_m\|_{\infty} \|\partial_j\psi_m\|_{L^4} \|\nabla\psi^{\epsilon}_m\|_{L^4} + 2\lambda \|\psi^{\epsilon}_m\|_{\infty}^2 \|\nabla\partial_j\psi_m\|, \\
    &\epsilon\mathrm{d}_t\|\nabla^2\partial_j\psi_m\| \leq 2\|\nabla^2\partial_jv_m\|_{\infty} + 4\|\nabla\partial_jv_m\|_{\infty} \|\nabla\psi^{\epsilon}_m\| + 2\|\partial_jv_m\|_{\infty} \|\nabla^2\psi^{\epsilon}_m\| +
    \\& \quad
    2\|\nabla^2v_m\|_{\infty} \|\partial_j\psi_m\| + 4\|\nabla v_m\|_{\infty}\|\nabla\partial_j\psi_m\| + 8\lambda\|\psi^{\epsilon}_m\|_{\infty} \|\nabla^2\psi^{\epsilon}_m\|_{L^4} \|\partial_j\psi_m\|_{L^4} +
    \\& \quad
    8\lambda \|\nabla\psi^{\epsilon}_m\|_{L^6}^2 \|\partial_j\psi_m\|_{L^6} + 16\lambda \|\psi^{\epsilon}_m\|_{\infty} \|\nabla\psi^{\epsilon}_m\|_{L^4} \|\nabla\partial_j\psi^{\epsilon}_m\|_{L^4} +
    2\lambda\|\psi^{\epsilon}_m\|_{\infty}^2 \|\nabla^2\partial_j\psi_m\|.
  \end{align*}
  Owing to
  \begin{align*}
    \|\partial_j\psi_m\|_{L^4} \|\nabla\psi^{\epsilon}_m\|_{L^4} &\leq C \|\nabla\partial_j\psi_m\|^{\frac d4}\|\partial_j\psi_m\|^{1-\frac d4} \|\psi_m\|_{H^2}^{\frac 12}\|\psi_m\|_{\infty}^{\frac 12} \\
    &\leq C \|\psi_m\|_{H^2}^{\frac 12}\|\psi_m\|_{\infty}^{\frac 12}\left(\frac d4\|\nabla\partial_j\psi_m\| + \left({1-\frac d4}\right)\|\partial_j\psi_m\|\right),
    \\
    \|\nabla^2\psi^{\epsilon}_m\|_{L^4} \|\partial_j\psi_m\|_{L^4} &\leq C \|\nabla^3\psi_m\|^{ \frac{8+d}{12}}\|\psi_m\|^{\frac{4-d}{12}}\left(\frac d4\|\nabla\partial_j\psi_m\| + \left({1-\frac d4}\right)\|\partial_j\psi_m\|\right),
    \\
    \|\nabla\psi^{\epsilon}_m\|_{L^6}^2 \|\partial_j\psi_m\|_{L^6} &\leq C\|\nabla^2\psi^{\epsilon}_m\|^{1+\frac d3}\|
    \psi^{\epsilon}\|^{1-\frac d3} \|\nabla\partial_j\psi_m\|^{\frac d3} \|\partial_j\psi_m\|^{1-\frac d3} \\
    &\leq C\|\nabla^2\psi^{\epsilon}_m\|^{1+\frac d3}\|
    \psi^{\epsilon}\|^{1-\frac d3} \left(\frac d3\|\nabla\partial_j\psi_m\| + \left(1-\frac d3\right)\|\partial_j\psi_m\|\right),
  \end{align*}
  and
  \begin{multline*}
    \|\nabla\psi^{\epsilon}_m\|_{L^4} \|\nabla\partial_j\psi^{\epsilon}_m\|_{L^4} \leq C \|\psi_m\|_{H^2}^{\frac 12}\|\psi_m\|_{\infty}^{\frac 12} \|\nabla^2\partial_j\psi_m\|^{\frac 12 + \frac d8} \|\partial_j\psi_m\|^{\frac 12 - \frac d8} \\
    \leq C \|\psi_m\|_{H^2}^{\frac 12}\|\psi_m\|_{\infty}^{\frac 12} \left(\left(\frac 12 + \frac d8\right)\|\nabla^2\partial_j\psi_m\| + \left(\frac 12 - \frac d8\right) \|\partial_j\psi_m\|\right).
  \end{multline*}
  We can construct
  \begin{align*}
    \epsilon\mathrm{d}_t\|\partial_j\psi_m\|_{H^2} \leq \frac{C_1}{\epsilon^2}\|\partial_jv_m\|_{H^2} + \frac{C_2}{\epsilon^4}\|\partial_j\psi_m\|_{H^2}.
  \end{align*}
  Then we get for all $t\in (0, T]$
  \begin{equation*}
    \|\partial_j\psi_m\|_{H^2} \leq \frac{C_1t}{\epsilon^3}\|\partial_jv_m\|_{H^2}\exp\left(\frac{C_2t}{\epsilon^4}\right) \leq C(t,\lambda,\epsilon, |\boldsymbol{\nu}|)\sqrt{\lambda_j}\|v_j\|_{H^2},
  \end{equation*}
  where $C(t,\lambda,\epsilon, |\boldsymbol{\nu}|)$ depends on $t,\lambda,\epsilon$ but is independent of dimensions.

  Then for $|\boldsymbol{\nu}| \geq 2$, by the Leibniz rule we have
  \begin{align*}
    &i\epsilon\partial_t\partial^{\boldsymbol{\nu}}\psi^{\epsilon}_m = -\frac{\epsilon^2}{2}\Delta(\partial^{\boldsymbol{\nu}}\psi^{\epsilon}_m) + \sum_{\boldsymbol{\mu}\preceq\boldsymbol{\nu}}\begin{pmatrix}
      \boldsymbol{\nu} \\
      \boldsymbol{\mu}
    \end{pmatrix}\partial^{\boldsymbol{\nu}-\boldsymbol{\mu}}v_m
    \partial^{\boldsymbol{\mu}}\psi^{\epsilon}_m +
    \lambda\sum_{\boldsymbol{\mu}\preceq\boldsymbol{\nu}}\begin{pmatrix}
      \boldsymbol{\nu} \\
      \boldsymbol{\mu}
    \end{pmatrix}\partial^{\boldsymbol{\nu}-\boldsymbol{\mu}}|\psi^{\epsilon}_m|^2
    \partial^{\boldsymbol{\mu}}\psi^{\epsilon}_m \\
    &= -\frac{\epsilon^2}{2}\Delta(\partial^{\boldsymbol{\nu}}\psi^{\epsilon}_m) + v_m\partial^{\boldsymbol{\nu}}\psi^{\epsilon}_m + \lambda(2|\psi^{\epsilon}_m|^2\partial^{\boldsymbol{\nu}}\psi^{\epsilon}_m + (\psi^{\epsilon}_m)^2\partial^{\boldsymbol{\nu}}\bar\psi^{\epsilon}_m) +\\
    &\sum_{\substack{\boldsymbol{\mu}\prec\boldsymbol{\nu},\\|\boldsymbol{\nu}-\boldsymbol{\mu}| = 1}}\begin{pmatrix}
      \boldsymbol{\nu} \\
      \boldsymbol{\mu}
    \end{pmatrix}\partial^{\boldsymbol{\nu}-\boldsymbol{\mu}}v_m
    \partial^{\boldsymbol{\mu}}\psi^{\epsilon}_m +
    \lambda\sum_{\boldsymbol{\mu}\prec\boldsymbol{\nu}}\begin{pmatrix}
      \boldsymbol{\nu} \\
      \boldsymbol{\mu}
    \end{pmatrix}\sum_{\boldsymbol{\eta}\preceq\boldsymbol{\nu}-\boldsymbol{\mu}}\begin{pmatrix}
      \boldsymbol{\nu}-\boldsymbol{\mu} \\
      \boldsymbol{\eta}
    \end{pmatrix}
    \partial^{\boldsymbol{\nu}-\boldsymbol{\mu}-\boldsymbol{\eta}}\psi^{\epsilon}_m
    \partial^{\boldsymbol{\mu}}\bar\psi^{\epsilon}_m\partial^{\boldsymbol{\eta}}\psi^{\epsilon}_m.
  \end{align*}
  Repeat the above procedures, and we get
  \begin{align*}
    &\epsilon\mathrm{d}_t\|\partial^{\boldsymbol{\nu}}\psi^{\epsilon}_m\| \leq 2|\boldsymbol{\nu}|\sum_{|\boldsymbol{\nu}-\boldsymbol{\mu}| = 1}\|\partial^{\boldsymbol{\nu}-\boldsymbol{\mu}}v_m\|_{\infty}
    \|\partial^{\boldsymbol{\mu}}\psi^{\epsilon}_m\| +
    2\lambda\|\psi^{\epsilon}\|_{\infty}^2\|\partial^{\boldsymbol{\nu}}\psi^{\epsilon}_m\| +
    \\ &\quad
    2\lambda \sum_{\boldsymbol{\mu}\prec\boldsymbol{\nu}}
    \begin{pmatrix}
      \boldsymbol{\nu} \\
      \boldsymbol{\mu}
    \end{pmatrix}
    \sum_{\boldsymbol{\eta}\preceq\boldsymbol{\nu}-\boldsymbol{\mu}}
    \begin{pmatrix}
      \boldsymbol{\nu}-\boldsymbol{\mu} \\
      \boldsymbol{\eta}
    \end{pmatrix}
    \|\partial^{\boldsymbol{\nu}-\boldsymbol{\mu}-\boldsymbol{\eta}}\psi^{\epsilon}_m\|_{L^6}
    \|\partial^{\boldsymbol{\mu}}\psi^{\epsilon}_m\|_{L^6}
    \|\partial^{\boldsymbol{\eta}}\psi^{\epsilon}_m\|_{L^6},
    \\
%   \end{align*}
%   \begin{align*}
    &\epsilon\mathrm{d}_t\|\nabla\partial^{\boldsymbol{\nu}}\psi^{\epsilon}_m\| \leq 2\|\nabla v_m\|_{\infty}\|\partial^{\boldsymbol{\nu}}\psi_m\| +
    2\lambda C(\|\nabla\psi^{\epsilon}_m\|_{L^4} \|\partial^{\boldsymbol{\nu}}\psi^{\epsilon}_m\|_{L^4} +
    \|\nabla\partial^{\boldsymbol{\nu}}\psi^{\epsilon}_m\|)
    \\&\quad
    +2|\boldsymbol{\nu}|\sum_{|\boldsymbol{\nu}-\boldsymbol{\mu}| = 1} (\|\nabla\partial^{\boldsymbol{\nu}-\boldsymbol{\mu}}v_m\|_{\infty}
    \|\partial^{\boldsymbol{\mu}}\psi^{\epsilon}_m\| + \|\partial^{\boldsymbol{\nu}-\boldsymbol{\mu}}v_m\|_{\infty}
    \|\nabla\partial^{\boldsymbol{\mu}}\psi^{\epsilon}_m\|) +
    \\&\quad
    2\lambda \sum_{\boldsymbol{\mu}\prec\boldsymbol{\nu}}
    \begin{pmatrix}
      \boldsymbol{\nu} \\
      \boldsymbol{\mu}
    \end{pmatrix}
    \sum_{\boldsymbol{\eta}\preceq\boldsymbol{\nu}-\boldsymbol{\mu}}
    \begin{pmatrix}
      \boldsymbol{\nu}-\boldsymbol{\mu} \\
      \boldsymbol{\eta}
    \end{pmatrix}
    \Big[\|\nabla\partial^{\boldsymbol{\nu}-\boldsymbol{\mu}-\boldsymbol{\eta}}\psi^{\epsilon}_m\|_{L^6}
    \|\partial^{\boldsymbol{\mu}}\psi^{\epsilon}_m\|_{L^6}
    \|\partial^{\boldsymbol{\eta}}\psi^{\epsilon}_m\|_{L^6} +
    \\&\quad
    \|\partial^{\boldsymbol{\nu}-\boldsymbol{\mu}-\boldsymbol{\eta}}\psi^{\epsilon}_m\|_{L^6}
    \|\nabla\partial^{\boldsymbol{\mu}}\psi^{\epsilon}_m\|_{L^6}
    \|\partial^{\boldsymbol{\eta}}\psi^{\epsilon}_m\|_{L^6} +
    \|\partial^{\boldsymbol{\nu}-\boldsymbol{\mu}-\boldsymbol{\eta}}\psi^{\epsilon}_m\|_{L^6}
    \|\partial^{\boldsymbol{\mu}}\psi^{\epsilon}_m\|_{L^6}
    \|\nabla\partial^{\boldsymbol{\eta}}\psi^{\epsilon}_m\|_{L^6}\Big].
  \end{align*}
  and
  \begin{align*}
    &\epsilon\mathrm{d}_t\|\nabla^2\partial^{\boldsymbol{\nu}}\psi^{\epsilon}_m\| \leq 2(\|\nabla^2 v_m\|_{\infty}\|\partial^{\boldsymbol{\nu}}\psi_m\| + \|\nabla v_m\|_{\infty}\|\nabla\partial^{\boldsymbol{\nu}}\psi_m\|)+
    8\lambda \|\nabla\psi^{\epsilon}_m\|_{L^6}^2 \|\partial^{\boldsymbol{\nu}}\psi_m\|_{L^6}
    \\&
    + 8\lambda\|\psi^{\epsilon}_m\|_{\infty} \|\nabla^2\psi^{\epsilon}_m\|_{L^4} \|\partial^{\boldsymbol{\nu}}\psi_m\|_{L^4} + 16\lambda \|\psi^{\epsilon}_m\|_{\infty} \|\nabla\psi^{\epsilon}_m\|_{L^4} \|\nabla\partial^{\boldsymbol{\nu}}\psi^{\epsilon}_m\|_{L^4} +
    \\&
    2|\boldsymbol{\nu}|\sum_{|\boldsymbol{\nu}-\boldsymbol{\mu}| = 1}
    \Big[
    \|\nabla^2\partial^{\boldsymbol{\nu}-\boldsymbol{\mu}}v_m\|_{\infty}
    \|\partial^{\boldsymbol{\mu}}\psi^{\epsilon}_m\| + 2\|\nabla\partial^{\boldsymbol{\nu}-\boldsymbol{\mu}}v_m\|_{\infty}
    \|\nabla\partial^{\boldsymbol{\mu}}\psi^{\epsilon}_m\| +
    \\&
    \|\partial^{\boldsymbol{\nu}-\boldsymbol{\mu}}v_m\|_{\infty}
    \|\nabla^2\partial^{\boldsymbol{\mu}}\psi^{\epsilon}_m\|
    \Big]  +
    2\lambda\|\psi^{\epsilon}_m\|_{\infty}^2 \|\nabla^2\partial^{\boldsymbol{\nu}}\psi_m\| +
    \\&
    6\lambda C \sum_{\boldsymbol{\mu}\prec\boldsymbol{\nu}}
    \begin{pmatrix}
      \boldsymbol{\nu} \\
      \boldsymbol{\mu}
    \end{pmatrix}
    \sum_{\boldsymbol{\eta}\preceq\boldsymbol{\nu}-\boldsymbol{\mu}}
    \begin{pmatrix}
      \boldsymbol{\nu}-\boldsymbol{\mu} \\
      \boldsymbol{\eta}
    \end{pmatrix}
    \|\partial^{\boldsymbol{\nu}-\boldsymbol{\mu}-\boldsymbol{\eta}}\psi^{\epsilon}_m\|_{H^2}
    \|\partial^{\boldsymbol{\mu}}\psi^{\epsilon}_m\|_{H^2}
    \|\partial^{\boldsymbol{\eta}}\psi^{\epsilon}_m\|_{H^2},
  \end{align*}
  in which we use the inequality generalized from Proposition 3.6 in~\cite{taylor2010partial} as
  \begin{gather*}
    \|\nabla^2 f g h\| \leq C\|f\|_{H^2}\|g\|_{H^2}\|h\|_{H^2},\\
    \|(\nabla f) (\nabla g) h\| \leq C\|f\|_{H^2}\|g\|_{H^2}\|h\|_{H^2}.
  \end{gather*}
  %By the inequalities
%  \begin{gather*}
%    \|f\|_{L^6} \leq C\|\nabla f\|^{\frac d3}\|f\|^{1-\frac d3} \leq C\|f\|_{H^2}, \\
%    \|\nabla f\|_{L^6} \leq C\|\nabla^2 f\|^{\frac d6 + \frac 12}\|f\|^{\frac 12-\frac d6} \leq C\|f\|_{H^2},
%  \end{gather*}
  Thus we get
  \begin{multline*}
    \epsilon\mathrm{d}_t\|\partial^{\boldsymbol{\nu}}\psi^{\epsilon}_m\|_{H^2} \leq C_3\|\partial^{\boldsymbol{\nu}}\psi^{\epsilon}_m\|_{H^2} + C_4|\boldsymbol{\nu}|\sum_{|\boldsymbol{\nu}-\boldsymbol{\mu}| = 1}
    \|\partial^{\boldsymbol{\nu}-\boldsymbol{\mu}}v_m\|_{H^2}
    \|\partial^{\boldsymbol{\mu}}\psi^{\epsilon}_m\|_{H^2} +
    \\
    \lambda C_5 \sum_{\boldsymbol{\mu}\prec\boldsymbol{\nu}}
    \begin{pmatrix}
      \boldsymbol{\nu} \\
      \boldsymbol{\mu}
    \end{pmatrix}
    \sum_{\boldsymbol{\eta}\preceq\boldsymbol{\nu}-\boldsymbol{\mu}}
    \begin{pmatrix}
      \boldsymbol{\nu}-\boldsymbol{\mu} \\
      \boldsymbol{\eta}
    \end{pmatrix}
    \|\partial^{\boldsymbol{\nu}-\boldsymbol{\mu}-\boldsymbol{\eta}}\psi^{\epsilon}_m\|_{H^2}
    \|\partial^{\boldsymbol{\mu}}\psi^{\epsilon}_m\|_{H^2}
    \|\partial^{\boldsymbol{\eta}}\psi^{\epsilon}_m\|_{H^2}.
  \end{multline*}
  An application of the Gronwall inequality yields
  \begin{multline*}
    \|\partial^{\boldsymbol{\nu}}\psi^{\epsilon}_m\|_{H^2}\leq \exp\left(\frac{C_3T}{\epsilon}\right)
    \Big\{
    \frac{C_4 T |\boldsymbol{\nu}|}{\epsilon} \sum_{|\boldsymbol{\nu}-\boldsymbol{\mu}| = 1}
    \|\partial^{\boldsymbol{\nu}-\boldsymbol{\mu}}v_m\|_{H^2}
    \|\partial^{\boldsymbol{\mu}}\psi^{\epsilon}_m\|_{H^2} +
    \\
    \frac{\lambda C_5T}{\epsilon} \sum_{\boldsymbol{\mu}\prec\boldsymbol{\nu}}
    \begin{pmatrix}
      \boldsymbol{\nu} \\
      \boldsymbol{\mu}
    \end{pmatrix}
    \sum_{\boldsymbol{\eta}\preceq\boldsymbol{\nu}-\boldsymbol{\mu}}
    \begin{pmatrix}
      \boldsymbol{\nu}-\boldsymbol{\mu} \\
      \boldsymbol{\eta}
    \end{pmatrix}
    \|\partial^{\boldsymbol{\nu}-\boldsymbol{\mu}-\boldsymbol{\eta}}\psi^{\epsilon}_m\|_{H^2}
    \|\partial^{\boldsymbol{\mu}}\psi^{\epsilon}_m\|_{H^2}
    \|\partial^{\boldsymbol{\eta}}\psi^{\epsilon}_m\|_{H^2}\Big\}.
  \end{multline*}
  Use the induction argument and we get
  \begin{equation*}
    \|\partial^{\boldsymbol{\nu}}\psi_m\|_{H^2} \leq C(t,\lambda,\epsilon, |\boldsymbol{\nu}|)\prod_{j}(\sqrt{\lambda_j}\|v_j\|_{H^2})^{\nu_j}.
  \end{equation*}
\end{proof}

\bibliographystyle{amsplain}
\bibliography{references}
\end{document}